\documentclass[leqno,11pt]{amsart}
\usepackage{amssymb, amsmath}
\usepackage{amsthm, amsfonts,mathrsfs}
\def\longformule#1#2{
\displaylines{ \qquad{#1} \hfill\cr \hfill {#2} \qquad\cr } }
\def\inte#1{
\displaystyle\mathop{#1\kern0pt}^\circ }

\let\pa=\partial
\let\al=\alpha
\let\b=\bar

\let\d=\delta
\let\e=\varepsilon

\let\f=\frac

\let\p=\psi

\let\D=\Delta

\let\wt=\widetilde
\let\wh=\widehat

\def\cA{{\mathcal A}}
\def\cB{{\mathcal B}}
\def\cC{{\mathcal C}}

\def\cF{{\mathcal F}}

\def\cS{{\mathcal S}}

\def\cY{{\mathcal Y}}

\def\pa{\partial}

\def\dB{\dot{B}}

\def\v{{\rm v}}
\def\h{{\rm h}}

\def\bY{\bar{Y}}

\def\Djl{\Delta_j\Delta_\ell^{\rm v}}

\def\virgp{\raise 2pt\hbox{,}}
\def\cdotpv{\raise 2pt\hbox{;}}

\def\eqdefa{\buildrel\hbox{\footnotesize def}\over =}

\def\C{\mathop{\mathbb C\kern 0pt}\nolimits}
\def\DD{\mathop{\mathbb D\kern 0pt}\nolimits}
\def\EE{\mathop{{\mathbb E \kern 0pt}}\nolimits}
\def\K{\mathop{\mathbb K\kern 0pt}\nolimits}
\def\N{\mathop{\mathbb N\kern 0pt}\nolimits}
\def\Q{\mathop{\mathbb Q\kern 0pt}\nolimits}
\def\R{\mathop{\mathbb R\kern 0pt}\nolimits}
\def\SS{\mathop{\mathbb S\kern 0pt}\nolimits}
\def\ZZ{\mathop{\mathbb Z\kern 0pt}\nolimits}
\def\TT{\mathop{\mathbb T\kern 0pt}\nolimits}
\def\P{\mathop{\mathbb P\kern 0pt}\nolimits}

\newcommand{\la}{\lambda}

\newcommand{\Z}{{\ZZ}}

\newcommand{\vv}[1]{\boldsymbol{#1}}

\def\dv{\mbox{div}}

\def\dive{\mathop{\rm div}\nolimits}

\def\no{\noindent}
\def\na{\nabla}
\def\p{\partial}
\def\th{\theta}

\newcommand{\w}[1]{\langle {#1} \rangle}
\newcommand{\beq}{\begin{equation}}
\newcommand{\eeq}{\end{equation}}
\newcommand{\ben}{\begin{eqnarray}}
\newcommand{\een}{\end{eqnarray}}
\newcommand{\beno}{\begin{eqnarray*}}
\newcommand{\eeno}{\end{eqnarray*}}
\newcommand{\andf}{\quad\hbox{and}\quad}
\newcommand{\with}{\quad\hbox{with}\quad}
\newtheorem{defi}{Definition}[section]
\newtheorem{thm}{Theorem}[section]
\newtheorem{lem}{Lemma}[section]
\newtheorem{rmk}{Remark}[section]
\newtheorem{col}{Corollary}[section]
\newtheorem{prop}{Proposition}[section]
\renewcommand{\theequation}{\thesection.\arabic{equation}}

\begin{document}

\title[Global solution of 3-D MHD system]
{On the global solution of 3-D  MHD system with initial data near
equilibrium }
\author[H.  Abidi]{Hammadi Abidi}
\address[H.  Abidi]{D\'epartement de Math\'ematiques
Facult\'e des Sciences de Tunis Campus universitaire 2092 Tunis,
Tunisia} \email{habidi@univ-evry.fr}
\author[P. Zhang]{Ping Zhang} \address[P. Zhang]{Academy of Mathematics $\&$ Systems Science
and  Hua Loo-Keng Key Laboratory of Mathematics, Chinese Academy of
Sciences, Beijing 100190, CHINA.} \email{zp@amss.ac.cn}

\date{Nov. 10, 2015}

\maketitle

\begin{abstract}  In this paper, we prove the global existence of
smooth solutions to the
 three-dimensional incompressible magneto-hydrodynamical  system with  initial
data close enough to the equilibrium state, $(e_3,0).$ Compared with
the the previous works \cite{XLZMHD1, XZ15}, here we present a new
Lagrangian formulation of the system, which is a damped wave
equation and which is non-degenerate only in the direction of the
initial magnetic field. Furthermore, we remove the admissible
condition on the initial magnetic field, which was required in
\cite{XLZMHD1, XZ15}. By using Frobenius Theorem and anisotropic
Littlewood-Paley theory for the Lagrangian formulation of the
system, we achieve the global $L^1$ in time Lipschwitz estimate of
the velocity field, which allows us to conclude the global existence
of solutions to this system. In the case when the initial magnetic
field is a constant vector, the large time decay rate of the
solution is also obtained.
\end{abstract}


 \qquad\qquad
 coordinates\

\vskip 0.2cm


\renewcommand{\theequation}{\thesection.\arabic{equation}}

\setcounter{equation}{0}
\section{Introduction}
In this paper, we investigate the  global existence of smooth
solutions to the following three-dimensional incompressible magnetic
hydrodynamical
 system (or MHD in short) with initial data being sufficiently close to the equilibrium state
 $(e_3,0):$
\beq\label{1.1} \left\{\begin{array}{l}
\displaystyle \pa_t b +u\cdot\nabla b=b\cdot\nabla u,\qquad (t,x)\in\R^+\times\R^3, \\
\displaystyle \pa_t  u+u\cdot\nabla u -\Delta u+\na p=b\cdot\nabla b, \\
\displaystyle \dv\, u =\dv\, b= 0, \\
\displaystyle (b,u)|_{t=0}=(b_0,u_0)\with b_0=e_3+\e \phi,
\end{array}\right.
\eeq where $b=(b^1,b^2,b^3)$ denotes the magnetic field, and  $
u=(u^1,u^2,u^3), p$ the velocity and scalar pressure of the fluid
respectively.  This MHD system \eqref{1.1} with zero diffusivity in
 the magnetic field equation  can be applied to model plasmas when
the plasmas are strongly collisional, or the resistivity due to
these collisions are extremely small. One may check the references
\cite{Ca,CP,LL} for detailed explanations to this system.

\medbreak In general,  it is not known whether or not classical
solutions of \eqref{1.1} can develop finite time singularities even
in  two dimension. In the case when there is full magnetic diffusion
in \eqref{1.1}, Duvaut and Lions \cite{DL} established the local
existence and uniqueness of solution in the classical Sobolev space
$H^s(\R^d)$, $s\geq d,$ they also proved the global existence of
solutions to this system with small initial data; Sermange and Temam
\cite{ST} proved the global well-posedness of this system  in the
two space dimension;
 the first author and Paicu \cite{AP} proved similar result as that in \cite{DL} for the so-called inhomogeneous MHD system
  with  initial data in the critical spaces.
With mixed partial dissipation and additional magnetic diffusion in
the two-dimensional MHD system, Cao and Wu \cite{CW} (see also
\cite{CRW}) proved that such a system is globally well-posed for any
data in $H^2(\R^2).$ Very recently, Chemin et al \cite{CMRR} proved
the local well-posedness of \eqref{1.1} with initial data in the
critical Besov spaces. One may check the survey paper \cite{Lin} and
the references therein for the recent progresses in this direction
and also its relations to the incompressible visco-elastic fluid
system.

\medbreak

 Furthermore,
whether there is dissipation or not for the magnetic field of
\eqref{1.1} is a very important problem also from physics of
plasmas. The heating of high temperature plasmas by MHD waves is one
of the most interesting and challenging problems of plasma physics
especially when the energy is injected into the system at the length
scales which are much larger than the dissipative ones. It has been
conjectured that in the three-dimensional MHD system, energy is
dissipated at a rate that is independent of the ohmic resistivity
\cite{ChCa}. In other words, the viscosity (diffusivity) for the
magnetic field equation can be zero yet the whole system may still
be dissipative. As a first step to investigate this problem, Lin and
the second author \cite{LZ} proved the global well-posedness to a
modified three-dimensional MHD system  with initial data
sufficiently close to the equilibrium state (see \cite{Lin-Zhang}
for a simplified proof). This problem was partially solved in 2-D by
Lin, Xu and the second author in \cite{XLZMHD1} and by Xu and the
second author in 3-D in \cite{XZ15} provided that the initial data
is near the equilibrium state $(e_d,0)$ and the initial magnetic
field, $b_0,$ satisfies the following admissible condition, namely
\beq \label{S1eq1} \int_{\R}(b_0-e_3)(Z(t,\al))\,dt=0\quad\mbox{for
all }\quad  \al\in \R^{d-1}\times\{0\} \eeq with $Z(t,\al)$ being
determined by \beno \left\{\begin{array}{l}
\displaystyle \f{d Z(t,\al)}{dt}= b_0(Z(t,\al)), \\
\displaystyle Z(t,\al)|_{t=0}=\al.
\end{array}\right. \eeno
In the 2-D case, the restriction \eqref{S1eq1} was removed by Ren,
Wu, Xiang and Zhang in \cite{RWXZ} by carefully exploiting the
divergence structure of the velocity field. Moreover, the authors
proved that \beq\label{S1eq2}
\|\p_{x_2}^kb(t)\|_{L^2}+\|\p_{x_2}^ku(t)\|_{L^2}\leq
C\w{t}^{-\f{s+k}2} \quad\mbox{for any }\quad s\in ]0, 1/2[\andf
k=0,1,2. \eeq A more elementary existence proof was also given by
Zhang in \cite{ZhangTing}. Very recently,  Ren,  Xiang and Zhang
extended this well-posedness result to the strip domain in
\cite{RXZ}. The goal of this paper is to remove the assumption
\eqref{S1eq1}  and improve the decay estimates \eqref{S1eq2} for the
limiting case $s=\f12$ in three space dimension. \medbreak

Before we present the function spaces we are going to work with in
this context, let us briefly recall some basic facts on
Littlewood-Paley theory (see e.g. \cite{bcd}). Let $\varphi$ and
$\chi$ be smooth functions supported in $\mathcal{C}\eqdefa \{
\tau\in\R^+,\ \frac{3}{4}\leq\tau\leq\frac{8}{3}\}$ and
$\frak{B}\eqdefa \{ \tau\in\R^+,\ \tau\leq\frac{4}{3}\}$
respectively such that
\begin{equation*}
 \sum_{j\in\Z}\varphi(2^{-j}\tau)=1 \quad\hbox{for}\quad \tau>0\quad\mbox{and}\quad  \chi(\tau)+ \sum_{j\geq
0}\varphi(2^{-j}\tau)=1\quad\hbox{for}\quad \tau\geq 0.
\end{equation*}
For $a\in{\mathcal S}'(\R^3),$ we set \beq
\begin{split}
&\Delta_k^\h
a\eqdefa\cF^{-1}(\varphi(2^{-k}|\xi_\h|)\widehat{a}),\qquad
S^\h_ka\eqdefa\cF^{-1}(\chi(2^{-k}|\xi_\h|)\widehat{a}),
\\
& \Delta_\ell^\v a
\eqdefa\cF^{-1}(\varphi(2^{-\ell}|\xi_3|)\widehat{a}),\qquad \
S^\v_\ell a \eqdefa \cF^{-1}(\chi(2^{-\ell}|\xi_3|)\widehat{a}),
 \quad\mbox{and}\\
&\Delta_ja\eqdefa\cF^{-1}(\varphi(2^{-j}|\xi|)\widehat{a}),
 \qquad\ \ \
S_ja\eqdefa \cF^{-1}(\chi(2^{-j}|\xi|)\widehat{a}), \end{split}
\label{1.0}\eeq where  $\xi_\h=(\xi_1,\xi_2),$ $\xi=(\xi_\h,\xi_3),$
$\cF a$ and $\widehat{a}$ denote the Fourier transform of the
distribution $a.$ The dyadic operators satisfy the property of
almost orthogonality:
\begin{equation}\label{C4}
\Delta_k\Delta_j a\equiv 0 \quad\mbox{if}\quad| k-j|\geq 2
\quad\mbox{and}\quad \Delta_k( S_{j-1}a \Delta_j b) \equiv
0\quad\mbox{if}\quad| k-j|\geq 5.
\end{equation}
Similar properties hold for $\D_k^\h$ and $\D_\ell^\v.$

\begin{defi}[Definition 2.15 of \cite{bcd}]\label{def1}
{\sl   Let $(p,r)\in[1,+\infty]^2,$ $s\in\R$ and $a\in{\mathcal
S}_h'(\R^3),$  which means $a\in\cS'(\R^d)$ and
$\lim_{j\to-\infty}\|\chi(2^{-j}D)a\|_{L^\infty}$ $=0,$  we set
$$
\|a\|_{{\dB}^s_{p,r}}\eqdefa\Big(2^{js}\|\Delta_j
a\|_{L^{p}}\Big)_{\ell ^{r}}.
$$
\begin{itemize}

\item
For $s<\frac{3}{p}$ (or $s=\frac{3}{p}$ if $r=1$), we define $
{\dB}^s_{p,r}(\R^3)\eqdefa \big\{a\in{\mathcal S}_h'(\R^3)\;\big|\;
\| a\|_{{\dB}^s_{p,r}}<\infty\big\}.$

\item
If $k\in\N$ and $\frac{3}{p}+k-1\leq s<\frac{3}{p}+k$ (or
$s=\frac{3}{p}+k$ if $r=1$), then $ {\dB}^s_{p,r}(\R^3)$ is defined
as the subset of distributions $a\in{\mathcal S}_h'(\R^3)$ such that
$\partial^\beta a\in {\dB}^{s-k}_{p,r}(\R^3)$ whenever $|\beta|=k.$
\end{itemize}
When $p=2$ and $ r=1,$ we denote $\dB^s_{2,1}$ by $\dB^s$ and
$\dB^s(\R^2_{x_\h})$ by $\dB^s_\h.$}
\end{defi}

Due to the anisotropic spectral properties of the linearized
equation to \eqref{1.1} (see Section \ref{sect3} for more
explanation), we need also the following anisotropic type Besov norm
from \cite{LZ, XLZMHD1}:

\begin{defi}\label{def2}
{\sl  Let  $s_1,s_2\in\R,$ $r_1, r_2\in [1,\infty]$ and
$a\in{\mathcal S}_h'(\R^3),$ we define the norm
$$
\|a\|_{\cB^{s_1,s_2}_{r_1,r_2}}\eqdefa \Bigl(2^{js_1} \bigl(2^{\ell
s_2}\|\D_{j}\D^\v_{\ell}a\|_{L^2}\bigr)_{\ell^{r_2}}\Bigr)_{\ell^{r_1}}.
$$
In particular, when $r_1=r_2=1,$ we denote $
\|a\|_{\cB^{s_1,s_2}}\eqdefa \|a\|_{\cB^{s_1,s_2}_{1,1}}$ }
\end{defi}

The main result of this paper is as follows:

\begin{thm}\label{thm1}
{\sl Let $e_3=(0,0,1),$ $b_0=e_3+\e\phi$ with
$\phi=(\phi_1,\phi_2,\phi_3)\in C_c^3(\R^3)$ and $\dive \phi=0,$ let
$u_0\in H^s(\R^3)$ for $s\in ]3/2,3].$ Then there exist sufficiently
small positive constants $\e_0, c_0$ such that if \beq\label{S1eq3}
\|u_0\|_{\dB^{\f12}}\leq c_0 \andf \e\leq \e_0, \eeq  \eqref{1.1}
has a unique global solution $(b, u)$ so that for any $T>0,$
$b-e_3\in C([0,T]; H^{s}(\R^3)),$ $ u\in C([0,T]; H^{s}(\R^3))$ with
$\na u\in L^2(]0,T[;H^s(\R^3))$ and $ \na p\in C([0,T];
H^{s-1}(\R^3)).$ Moreover, in the case when $\e=0,$   and under the
additional  assumption that \beq\label{S1eq5}
\|u_0\|_{\cB^{0,0}}+\|u_0\|_{\cB^{3,0}}+
\|u_0\|_{\cB^{-1,-\f12}_{\infty,\infty}}
+\|u_0\|_{\cB^{3,-\f12}_{2,\infty}}\leq \d_0, \eeq for some $\d_0$
sufficiently small, one has \beq \label{S1eq4}
\|u(t)\|_{H^2}+\|b(t)-e_3\|_{H^2} \leq C\w{t}^{-\f14}\with
\w{t}=\bigl(1+t^2\bigr)^{\f12}. \eeq }\end{thm}

\begin{rmk}\label{rmk1.1} {\sl (1) Our approach to prove Theorem \ref{thm1} works in both  three space dimension and  two space dimension.
Moreover, for a concise presentation, here we did not optimize the
regularity of the initial magnetic field.

(2) In general, it is impossible to propagate the anisotropic
regularities for the solutions of  hyperbolic systems (it is only
possible for conormal regularities (see \cite{Chemin88} for
instance)). Since we need to use the anisotropic regularities of the
solution in order to prove the decay estimate \eqref{S1eq4},  we are
forced to study the large time  behavior of the solutions to the
Lagrangian formulation of \eqref{1.1}.

 (3) It is easy to observe from \eqref{1.2}, the equivalent Lagrangian
formulation of \eqref{1.1}, that, the solution $(b-e_3,u)$ to
\eqref{1.1} does not decay to zero as time goes to $\infty$ when the
initial magnetic field is not a constant vector. That is the reason
why we only investigate the large time behavior of the solution to
\eqref{1.1}  when $b_0=e_3.$

(4) More detailed decay estimates of the solution in the Lagrangian
coordinate will be presented in Theorem \ref{S2thm1} of Section
\ref{Sect2}.}
\end{rmk}

Let us complete this section by the notations we shall use in this context.\\

\no{\bf Notation.} For any $s\in\R$, we denote by $H^s(\R^3)$ the
classical  $L^2$ based Sobolev spaces with the norm
$\|\cdot\|_{H^s},$ while $\dot{H}^s(\R^3)$ the classical homogenous
Sobolev spaces with the norm $\|\cdot\|_{\dot{H}^s}$. For $a\lesssim
b$, we mean that there is a uniform constant $C,$ which may be
different on different lines, such that  $a\leq Cb,$ and $a\sim b$
means that both $a\lesssim b$ and $b\lesssim a$. We shall denote by
$(a|b)$  the $L^2(\R^3)$ inner product of $a$ and $b.$
$(d_{j,k})_{j,k\in\Z}$ (resp. $(d_j)_{j\in\Z}$) will be a generic
element of $\ell^1(\Z^2)$ (resp. $\ell^1(\Z))$ so that
$\sum_{j,k\in\Z}d_{j,k}=1$ (resp. $\sum_{j\in\Z}d_j=1).$ Finally, we
denote by $L^p_T(L^q_\h(L^r_\v))$ the space $L^p(]0,T[;
L^q(\R_{x_\h}^2;L^r(\R_{x_3})))$ with $x_\h=(x_1,x_2)$.

\setcounter{equation}{0}
\section{Lagrangian formulation of \eqref{1.1}}\label{Sect2}

In view of Proposition 6.1 of \cite{XZ15} (see also Proposition
\ref{p9} below), the main difficulty to prove the existence part of
Theorem \ref{thm1} is to achieve the  $L^1(\R^+; \mbox{Lip}(\R^3))$
estimate of the velocity field to the appropriate approximate
solutions of \eqref{1.1}. Due to the difficulty mentioned in (2) of
Remark \ref{rmk1.1}, we used Lagrangian formulation of \eqref{1.1}
in the previous works \cite{XLZMHD1, XZ15}.

Let us now explain the main idea for the Lagrangian formulation of
\eqref{1.1} in \cite{XLZMHD1, XZ15}. Taking the 3-D case for
example, given $b_0$ satisfying the admissible condition
\eqref{S1eq1}, the authors first construct a matrix $U_0=\bigl(\bar{
b}_0, \tilde{b}_0, b_0\bigr)$ with $\bar{
b}_0=\bigl(\bar{b}^1_0,\bar{b}^2_0,\bar{b}^3_0\bigr)^{t}$ and
$\tilde{
b}_0=\bigl(\tilde{b}^1_0,\tilde{b}^2_0,\tilde{b}^3_0\bigr)^{t},$ so
that there hold \beq \label{S2eq1} \det U_0=1,\quad \dive
\bar{b}_0=0\andf \dive\tilde{b}=0. \eeq Then instead of solving
\eqref{1.1}, the authors proposed to solve
\begin{equation}\label{S2eq1ad}
 \left\{\begin{array}{l}
\displaystyle \p_tU+ u\cdot\na U=\na u U,
\qquad (t,x)\in\R^+\times\R^3, \\
\displaystyle \pa_t u + u\cdot\na u -\D u+\na p=b\cdot\na b, \\
\displaystyle \dv\, u =0\quad\mbox{and}\quad \quad\dv\, U=0, \\
\displaystyle U|_{t=0}=U_0,\quad  u|_{t=0}= u_0.
\end{array}\right.
\end{equation}
Motivated by the Lagrangian formulation of the  visco-elastic system
in \cite{XZZ}, the authors gave the following Lagrangian formulation
of the  System \eqref{S2eq1ad}:
\beq\label{B12}\left\{\begin{aligned}
&Y_{tt}-\Delta_y Y_t-\p_{y_3}^2 Y=(\na_Y\cdot\na_Y-\Delta_y)Y_t-\na_Yq,\\
&\na_y\cdot Y=\na_y\cdot Y_0-\int_0^t(\na_Y-\na_y)\cdot Y_s
ds,\\
&Y|_{t=0}=Y_0,\quad Y_t|_{t=0}=Y_1,
\end{aligned}\right.\eeq
for $Y$ and $\na_Y$ being determined by \eqref{1.1f} below. It is
the restriction \eqref{S2eq1} that requires the admissible condition
\eqref{S1eq1}.

Here we shall give a more direct Lagrangian  formulation of the
System \eqref{1.1}, which will be based on Lemma 1.4 of
\cite{Majda}. In order to do so, let us first give an equivalent
formulation of \eqref{1.1}, which does not involve the pressure
function. Indeed
 we
get, by taking the space divergence to the velocity equation of
\eqref{1.1}, that \beq \label{S2eq2} \Delta p=\dv\,\dv\bigl(b\otimes
b-u\otimes u\bigr),\quad\mbox{or}\quad p\eqdefa
\D^{-1}\dv\,\dv\bigl(b\otimes b-u\otimes u\bigr). \eeq Then
\eqref{1.1} can be equivalently reformulated as \beq\label{S2eq3}
\left\{\begin{array}{l}
\displaystyle \pa_t b +u\cdot\nabla b=b\cdot\nabla u,\qquad (t,x)\in\R^+\times\R^3, \\
\displaystyle \pa_t  u+u\cdot\nabla u -\Delta u+\nabla p=b\cdot\nabla b, \\
\displaystyle (b,u)|_{t=0}=(b_0,u_0),\andf \dv\, b_0=\dv\, u_0=0,
\end{array}\right.
\eeq with $p$ given by \eqref{S2eq2}. And then just as in Chapter 1
of \cite{Chemin98} for the incompressible Euler system, the
divergence free condition of $u$ and $b$ can be derived by the
initial condition $\dv\, b_0=\dv\, u_0=0$ and the evolution equation of
$\dv\,b$ and $\dv\,u.$

Now let $(b,u)$ be a smooth enough solution of \eqref{S2eq3}, we
define the Lagrangian trajectory $X(t,y)$ by \beq \label{1.1b}
\left\{\begin{array}{l}
\displaystyle \f{d}{dt} X(t,y)=u(t,X(t,y)), \\
\displaystyle X(0,y)=y,
\end{array}\right.
\eeq which yields for  $i,j\in \{1,2,3\}$ that \beno \det\Bigl(\f{\p
X}{\p y}\Bigr)=1 \andf \f{d}{dt} \f{\p X^i(t,y)}{\p y_j}=\f{\p
u^i}{\p x_\ell}(t,X(t,y))\f{\p X^\ell(t,y)}{\p y_j}, \eeno and
\beq\label{1.1a} \f{d}{dt}\Bigl(b_0^j(y) \f{\p X^i(t,y)}{\p
y_j}\Bigr)=\Bigl(b_0^j(y)\f{\p X^\ell(t,y)}{\p y_j}\Bigr)\f{\p
u^i}{\p x_\ell}(t,X(t,y)). \eeq On the other hand, it follows from
the magnetic field equation of  \eqref{S2eq3} and \eqref{1.1b} that
\beno \f{d b^i(t,X(t,y))}{dt}=b(t,X(t,y))\cdot\na u^i(t,X(t,y)),
\eeno which together with \eqref{1.1a} ensures that \beq
\label{1.1c} b^i(t,X(t,y))=b_0^j(y)\f{\p X^i(t,y)}{\p
y_j}=b_0(y)\cdot\na_y X^i(t,y)\eqdefa \p_{b_0}X^i(t,y). \eeq For any
smooth function $f$, we deduce from chain rule that \beno \f{\p
f(X(t,y))}{\p y_j}=\bigl(\f{\p f}{\p x_\ell}\bigr)(X(t,y))\f{\p
X^\ell(t,y)}{\p y_j}. \eeno Let us denote the inverse matrix of
$\f{\p X(t,y)}{\p y}$ by $\cA(t,y)=\bigl(a_{ij}(t,y)\bigr).$ Then we
have \beq \label{1.1d} \bigl(\f{\p f}{\p
x_i}\bigr)(X(t,y))=a_{ji}(t,y)\f{\p f(X(t,y))}{\p y_j}\quad
\mbox{or}\quad \bigl(\na_xf\bigr)(X(t,y))=\cA^t\na_y(f(X(t,y)). \eeq
By virtue of \eqref{1.1c} and \eqref{1.1d}, we infer \beq
\label{1.1e}\begin{split} \bigl(b^j\p_jb^i\bigr)(t,
X(t,y))=&b_0^k(y)\f{\p X^j(t,y)}{\p
y_k}a_{\ell j}(t,y)\f{\p b^i(t,X(t,y))}{\p y_\ell}\\
=&b_0^k(y)\d_{k\ell}\p_{ y_\ell}\bigl(\p_{b_0}
X^i(t,y)\bigr)\\
=&\p_{b_0}^2 X^i(t,y). \end{split} \eeq

Let us denote \beq\label{1.1f}
\begin{split}
X(t,y)=&y+\int_0^t u(t',X(t',y))dt'\eqdefa y+Y(t,y),\quad \vv u(t,y)\eqdefa u(t,X(t,y)), \\
 {\vv b}(t,y)\eqdefa &b(t,X(t,y)),\,\, \vv p(t,y)\eqdefa p(t,X(t,y)),\,\, \cA\eqdefa\left(Id+\na_y Y\right)^{-1} \andf \na_Y\eqdefa \cA^{t}\na_y.
\end{split} \eeq
Then thanks to \eqref{S2eq3},  \eqref{1.1c} and \eqref{1.1e},  we
write \beq\label{1.2} \left\{\begin{array}{l}
\displaystyle {\vv b}(t,y) =\p_{b_0}X(t,y),\quad \na_Y\cdot{\vv b}=0,\\
\displaystyle   Y_{tt} -\Delta_y Y_t-\partial_{b_0}^2Y=\p_{b_0}b_0+g,\\
\displaystyle Y_{|t=0}=Y_0=0,\qquad {Y_t}_{|t=0}=Y_1=u_0(y),
\end{array}\right.
\eeq where \beq\label{P}
\begin{split}
&g=\dv_y\bigl[(\mathcal{A}\mathcal{A}^{t}-Id)\na_yY_t\bigr]
-\mathcal{A}^{t}\na_y\vv p,\quad \p_{b_0}\eqdefa b_0\cdot\na_y, \andf\\
& (\Delta_x p)(t, X(t,y))
=\sum_{i,j=1}^3\na_{Y^i}\na_{Y^j}\bigl(\pa_{b_0}X^i\pa_{b_0}X^j-Y^i_tY^j_t\bigr)(t,y).
\end{split}
\eeq Compared with the Lagrangian formulation \eqref{B12} in
\cite{XLZMHD1, XZ15}, $\p_{y_3}^2 Y$ there is now  replaced by
$\partial_{b_0}^2Y,$ which causes new difficulty of the variable
coefficients for the linearized system.

In what follows, we assume that \beq \label{S2eq5}
\mbox{supp}(b_0(x_\h,\cdot)-e_3)\subset [0, K]\andf b_0^3\neq 0.
\eeq

Due to the difficulty of the variable coefficients for the
linearized system of \eqref{1.2}, we shall use Frobenius Theorem
type argument to find a new coordinate system $\{z\}$ so that
$\p_{b_0}=\p_{z_3}.$ Then we can use anisotropic Littlewood-Paley
analysis to achieve the $L^1_t(\mbox{Lip})$ estimate for $Y_t.$
Toward this, let us define \beq\label{S2eq13}
\left\{\begin{array}{l}
\displaystyle \frac{d y_1}{d y_3}=\f{b_0^1}{b_0^3}(y_1,y_2,y_3),\quad y_1|_{y_3=0}=w_1, \\
\displaystyle \frac{d y_2}{d
y_3}=\f{b_0^2}{b_0^3}(y_1,y_2,y_3),\quad y_2|_{y_3=0}=w_2,\\
\displaystyle y_3=w_3,
\end{array}\right.  \eeq
and
 \beq \label{CH}
\begin{split}
& z_1=w_1,\quad z_2=w_2,\quad
z_3=w_3+\int_{0}^{w_3}\Bigl(\f1{b_0^3(y(w))}-1\Bigr)\,dw_3'.
\end{split}
\eeq Then we have \beq\label{CHa}
\begin{split}
&\p_{b_0}f(y)=b_0^3\Bigl(\f{\p y_1}{\p w_3}\f{\p f(y)}{\p y_1}+\f{\p
y_2}{\p w_3}\f{\p f(y)}{\p y_2}+\f{\p f(y)}{\p w_3}\Bigr)\\
&\qquad\quad\  =b_0^3(y(w)))\f{\p f(y(w))}{\p w_3}=\f{\p
f(y(w(z)))}{\p
z_3},\andf\\
& \p_{z_i}\bigl(f(y(w(z)))\bigr) =\f{\p f}{\p y_j}(y(w(z)))\f{\p
y_j(w(z))}{\p
z_i}\quad \mbox{or}\\
& \na_y=\na_Z= {\cB}^{t}(z)\na_z \with {\cB}(z)=\Bigl(\f{\p
y(w(z))}{\p z}\Bigr)^{-1}.
\end{split}
\eeq It is easy to observe that \beno \begin{split}
{\cB}(z)=\Bigl(\f{\p y(w(z))}{\p z}\Bigr)^{-1}=&\Bigl(\f{\p
y(w(z))}{\p w}\times\f{\p w(z)}{\p z}
\Bigr)^{-1}\\
=&\Bigl(\f{\p w(z)}{\p z} \Bigr)^{-1}\Bigl(\f{\p y(w(z))}{\p
w}\Bigr)^{-1}=\Bigl(\f{\p z}{\p w}\Bigr)\Bigl(\f{\p y(w(z))}{\p
w}\Bigr)^{-1}.
\end{split} \eeno
Yet it follows from \eqref{S2eq13} that \beq\label{CHab}
\begin{split}
\Bigl(\f{\p y(w)}{\p w}\Bigr) = & \begin{pmatrix}
1&0&\f{b_0^1}{b_0^3}\\ 0&1&\f{b_0^2}{b_0^3}\\0&0&1\end{pmatrix}+
\begin{pmatrix} \int_0^{w_3}\f{\p}{\p
y_1}\bigl(\f{b_0^1}{b_0^3}\bigr)dy_3'&\int_0^{w_3}\f{\p}{\p
y_2}\bigl(\f{b_0^1}{b_0^3}\bigr)dy_3'&0 \\
\int_0^{w_3}\f{\p}{\p y_1}\bigl(\f{b_0^2}{b_0^3}\bigr)dy_3' & \int_0^{w_3}\f{\p}{\p y_2}\bigl(\f{b_0^2}{b_0^3}\bigr)dy_3' & 0\\
0 & 0 &  0 \end{pmatrix}  \begin{pmatrix} \f{\p y_1}{\p w_1} & \f{\p
y_1}{\p w_2} & \f{\p y_1}{\p w_3}\\  \f{\p y_2}{\p w_1} & \f{\p
y_2}{\p w_2} & \f{\p y_2}{\p w_3}\\  \f{\p y_3}{\p w_1} &
\f{\p y_3}{\p w_2} & \f{\p y_3}{\p w_3}\end{pmatrix}\\
\eqdefa & A_1(y(w))+ A_2(y(w))\Bigl(\f{\p y(w)}{\p w}\Bigr),
\end{split} \eeq which gives \beq \label{1.2fg} \Bigl(\f{\p y(w)}{\p
w}\Bigr) =\bigl(Id-A_2(y(w))\bigr)^{-1}A_1(y(w)). \eeq While it is
easy to observe that \beq \label{1.2fh} \Bigl(\f{\p z(w)}{\p
w}\Bigr) =
\begin{pmatrix}
1&0& 0\\
0&1&0\\
\int_0^{w_3}\f{\p}{\p w_1}\bigl(\f{1}{b_0^3(y(w))}\bigr)dw_3' &
\int_0^{w_3}\f{\p}{\p w_2}\bigl(\f{1}{b_0^3(y(w))}\bigr)dy_3' &
\f{1}{b_0^3}
\end{pmatrix}\eqdefa A_3(w).
\eeq As a consequence, we obtain \beq \label{1.2fk}
\begin{split}
& y(w)=(y_\h(w_\h,w_3),w_3),\quad w(z)=(z_\h,w_3(z)), \andf
y(w(z))=\bigl(y_\h(z_\h,w_3(z)),w_3(z)\bigr),
\\
&{\cB}(z)=A_3(w(z))A_1^{-1}(y(w(z)))\bigl(Id-A_2(y(w(z)))\bigr),
\end{split} \eeq with the matrices $A_1,A_2, A_3$ being given by
\eqref{CHab} and \eqref{1.2fh} respectively.

For simplicity, let us abuse the notation that
$Y(t,z)=Y(t,y(w(z))).$ Then the System \eqref{1.2} becomes
\beq\label{S2eq19} \left\{\begin{array}{l}
\displaystyle  Y_{tt} -\Delta_z Y_t-\partial_{z_3}^2Y=\bigl(\na_Z\cdot \na_Z-\D_z)Y_t+\p_{z_3}b_0(y(w(z)))+g(y(w(z))),\\
\displaystyle Y_{|t=0}=Y_0=0,\qquad
{Y_t}_{|t=0}=Y_1(z)=u_0(y(w(z))),
\end{array}\right.
\eeq for $g$ given by \eqref{1.2}. Since $\p_{z_3}b_0(y(w(z)))$ in
the source term is a time independent function, we now introduce a
correction term $\tilde{Y}$ so that $Y=\tilde{Y}+\bY$ and
\beq\label{1.2a} \p_{z_3}\tilde{Y}(z)=e_3-b_0(y(w(z))). \eeq Then
\beno \p_{z_3}\bigl(\p_{z_3}\tilde{Y}+b_0(y(w(z)))\bigr)=0, \eeno
and $\b Y$ solves
 \beq\label{1.2b}
\left\{\begin{array}{l}
\displaystyle  \bY_{tt} -\Delta_z \bY_t-\partial_{z_3}^2\bY=f,\\
\displaystyle {\bY}_{|t=0}={\bY}_0=-\tilde{Y},\qquad
{\b{Y_t}}_{|t=0}=Y_1,
\end{array}\right.
\eeq with \beq\label{S2-27}
\begin{split}
 \mathcal{A}=&\left(Id+{\cB}^{t}\na_z\wt
Y+{\cB}^{t}\na_z \bY\right)^{-1},\andf\\
f=&{\cB}^{t}\na_z\cdot\bigl[(\mathcal{A}\mathcal{A}^{t}-Id){\cB}^{t}\na_z\b
Y_t\bigr] +{\cB}^{t}\na_z\cdot({\cB}^{t}\na_z\b Y_t\bigr)
-\Delta_z\b Y_t-({\cB}\mathcal{A})^{t}\na_z\vv p.
\end{split}
\eeq In order to handle the term $\na_z\vv p$ in the source term
$f,$ we need the following lemma:

\begin{lem}\label{diff}
 Let $X(y)$ be a $C^1$ diffeomorphism over $\R^3$ and $H$
be a $C^1$ vector field. Then one has \beq \bigl(\dive_x
H\bigr)(X(y))={\rm det}\bigl(\f{\p X}{\p
y}\bigr)^{-1}\dive_y\Bigl({\rm det}\bigl(\f{\p X}{\p
y}\bigr)\bigl(\f{\p X}{\p y}\bigr)^{-1} H(X(y))\Bigr).
\label{1.2fk1} \eeq
\end{lem}

\begin{proof} The proof of this lemma basically follows from that of Lemma A.1 in \cite{DM12}, where
the authors proved \eqref{1.2fk1} for the case when $\det\bigl(\f{\p
X}{\p y}\bigr)=1.$ Let $\psi$ be a test function,  we denote
$\bar{\psi}(y)\eqdefa \psi(X(y)).$ Then in view of \eqref{1.1d}, one
has \beno
\begin{split}
\int_{\R^3}\bar{\psi}(y)\bigl(\dive_xH\bigr)(X(y))dy=&\int_{\R^3}{\psi}(x)\bigl(\dive_xH\bigr)(x){\rm
det}\bigl(\f{\p X}{\p y}\bigr)^{-1}dx\\
=&-\int_{\R^3}\na_x\Bigl({\psi}(x){\rm
det}\bigl(\f{\p X}{\p y}\bigr)^{-1}\Bigr)\cdot H(x)dx\\
=&-\int_{\R^3}\na_y\Bigl(\bar{\psi}(y){\rm det}\bigl(\f{\p X}{\p
y}\bigr)^{-1}\Bigr)\bigl(\f{\p X}{\p y}\bigr)^{-1}\bar{H}(y){\rm
det}\bigl(\f{\p X}{\p
y}\bigr)dy\\
=&\int_{\R^3}\bar{\psi}(y){\rm det}\bigl(\f{\p X}{\p
y}\bigr)^{-1}\dive_y\Bigl(\bigl(\f{\p X}{\p
y}\bigr)^{-1}\bar{H}(y){\rm det}\bigl(\f{\p X}{\p
y}\bigr)\Bigr)dy\end{split} \eeno This leads to \eqref{1.2fk1}.
\end{proof}

In particular, if $\det\bigl(\f{\p X}{\p y}\bigr)=1,$ one has \beq
\bigl(\dive_x H\bigr)(X(y))=\dive_y\Bigl(\bigl(\f{\p X}{\p
y}\bigr)^{-1} H(X(y))\Bigr), \label{diver} \eeq which recovers Lemma
A.1 in \cite{DM12}.

Let us now turn to the calculation of the pressure function in the
Lagrangian coordinate. We denote $\cY(t,y)\eqdefa
Y(t,y)-\wt{\cY}(y)$ with $\wt\cY(y)$ being determined by \beq
\label{S2eq35} \p_{b_0}\wt\cY(y)=e_3-b_0(y). \eeq Then in view of
\eqref{1.1f} and \eqref{S2eq35}, we infer
$$
\begin{aligned}
\sum_{i,j=1}^3\na_{Y^i}\na_{Y^j}\bigl(\pa_{b_0}X^i\pa_{b_0}X^j\bigr)
&=\sum_{i,j=1}^3\na_{Y^i}\na_{Y^j}\bigl((b_0^i+\pa_{b_0}Y^i)
(b_0^j+\pa_{b_0}Y^j)\bigr)
\\&
=\sum_{i,j=1}^2\na_{Y^i}\na_{Y^j}\bigl(\pa_{b_0}\cY^i\pa_{b_0}\cY^j\bigr)
+\na^2_{Y^3}(1+\pa_{b_0}\cY^3)^2
\\&\quad +2\sum_{i=1}^2\na_{Y^3}\na_{Y^i}\bigl(\pa_{b_0}\cY^i(1+\pa_{b_0}\cY^3)\bigr).
\end{aligned}
$$
However note that $\na_Y\cdot{\vv b}=0$ and \eqref{S2eq35}, one has
$$
\begin{aligned}
2\na^2_{Y^3}\pa_{b_0}\cY^3+2\sum_{i=1}^2\na_{Y^3}\na_{Y^i}\pa_{b_0}\cY^i
&=2\na_{Y^3}\sum_{i=1}^3\na_{Y^i}\pa_{b_0}\cY^i
\\&
=2\na_{Y^3}\sum_{i=1}^3\na_{Y^i}\pa_{b_0}\bigl(X^i-y^i-{\wt
\cY}^i\bigr)
\\&
=2\na_{Y^3}\na_{Y}\cdot {\vv b}
\\&
=0,
\end{aligned}
$$
which leads to
$$
\begin{aligned}
\sum_{i,j=1}^3\na_{Y^i}\na_{Y^j}\bigl(\pa_{b_0}X^i\pa_{b_0}X^j\bigr)
=\sum_{i,j=1}^3\na_{Y^i}\na_{Y^j}\bigl(\pa_{b_0}\cY^i\pa_{b_0}\cY^j\bigr).
\end{aligned}
$$
As a consequence, we deduce  from \eqref{P} and \eqref{diver} that
\beq \label{S2eq36} \dv_y\bigl(\mathcal{A}\mathcal{A}^{t}\na_y\vv
p\big) =\dv_{y}\Bigl(\mathcal{A}\dv_{y}
\left(\mathcal{A}\bigl(\pa_{b_0}\cY\otimes\pa_{b_0}\cY-\cY_t\otimes\cY_t\bigr)\right)\Bigr).
\eeq On the other hand, it follows \eqref{S2eq35} that
$\wt{\cY}(y(w(z)))$ solves \eqref{1.2a}. Let us fix
$\wt{Y}(z)=\wt{\cY}(y(w(z))).$ Then we find \beno
\cY(y(w(z)))=Y(t,y(w(z)))-\wt{\cY}(y(w(z)))=Y(t,y(w(z)))-\wt{Y}(z)=\bar{Y}(t,z).
\eeno Hence applying  \eqref{1.2fk1} to \eqref{S2eq36}  gives rise
to
$$
\dv_z\bigl(\det(\cB^{-1})\cB\mathcal{A}\mathcal{A}^{t}\cB^{t}\na_z
\vv p\bigr) =\dv_z\Bigl(\cB\mathcal{A}\dv_z
\left(\det(\cB^{-1})\cB\mathcal{A}\bigl(\pa_{3}\b Y\otimes\pa_{3}\b
Y -\b Y_t\otimes\b Y_t\bigr)\right)\Bigr).
$$
This yields
\begin{equation}\label{PR}
\begin{aligned}
\na_z \vv
p=&-\na_z\D_z^{-1}\dv_z\bigl(\det(\cB^{-1})(\cB\mathcal{A}\mathcal{A}^{t}\cB^{t}-Id)\na_z
\vv p\bigr)\\
& -\na_z\D_z^{-1}\dv_z\bigl((\det(\cB^{-1})Id-Id)\na_z \vv p\bigr)
\\&
+\na_z\D_z^{-1}\dv_z\Bigl(\cB\mathcal{A}\dv_z
\left(\det(\cB^{-1})\cB\mathcal{A}\bigl(\pa_{3}\b Y\otimes\pa_{3}\b
Y -\b Y_t\otimes\b Y_t\bigr)\right)\Bigr).
\end{aligned}
\end{equation}

The local well-posedness  of the System \eqref{1.1} implies the
local well-posedness of the System \eqref{1.2} and thus the System
\eqref{1.2b}. In what follows, we shall only use  the System
\eqref{1.2b} to derive the $L^1(\R^+;\mbox{Lip}(\R^3))$ estimate for
the velocity field $u$ of \eqref{1.1} provided that there holds
\eqref{S1eq3}.

To restrict the length of this paper, we shall present the details
concerning the propagation of regularities of $Y$ and $Y_t$ only in
the case when $b_0=e_3,$ (the general result can be done by the same
strategy),  which will be enough for us to  investigate the decay
estimate \eqref{S1eq4}. In this case, $\cB=Id,$  $\wt{Y}=0,$ and
\eqref{1.2} becomes \beq\label{1.2c} \left\{\begin{array}{l}
\displaystyle  Y_{tt} -\Delta_y Y_t-\partial_{y_3}^2Y=f,\\
\displaystyle {Y}_{|t=0}=Y_0=0,\qquad {{Y_t}}_{|t=0}=Y_1,
\end{array}\right.
\eeq with \beq\label{1.2d}
\begin{split}
& \mathcal{A}=(Id+\na_y Y)^{-1},\quad
f=\na_y\cdot\bigl((\mathcal{A}\mathcal{A}^{t}-Id)\na_y Y_t\bigr)
-\mathcal{A}^{t}\na_y\vv p,\andf\\
& \vv p=-\D_y^{-1}\dv_y\bigl((\mathcal{A}\mathcal{A}^{t}-Id)\na_y
\vv p\bigr) +\D_y^{-1}\dv_y\Bigl(\mathcal{A}\dv_y
\bigl(\mathcal{A}\bigl(\pa_{y_3} Y\otimes\pa_{y_3} Y -Y_t\otimes
Y_t\bigr)\bigr)\Bigr). \end{split} \eeq

The main result concerning the propagation of regularities and the
large time decay estimate  for the solutions of \eqref{1.2c} is
listed as follows:

\begin{thm}\label{S2thm1}
{\sl Let $Y_0\in \cB^{2,0}\cap\cB^{5,0}\cap \cB^{0,1}\cap \cB^{3,1}$
and $Y_1 \in \cB^{0,0}\cap\cB^{3,0}.$ Then under the assumption that
\beq\label{Lip0qd} \frak{g}(0)\le c_0 \with \frak{g}(s)\eqdefa
\|\pa_3Y_0\|_{\mathcal{B}^{s,0}} + \| Y_0\|_{\mathcal{B}^{s+2,0}}+\|
Y_1\|_{\mathcal{B}^{s,0}}, \eeq for some $c_0$ sufficiently small,
\eqref{1.2c} has a unique solution $Y$ so that there holds\beq
\label{S2eq29}
\begin{aligned}
  \|Y_t&\|_{\wt L^\infty_t(\mathcal{B}^{s,0})}
+\|\pa_3Y\|_{\wt L^\infty_t(\mathcal{B}^{s,0})}+\|Y\|_{\wt
L^\infty_t(\mathcal{B}^{s+2,0})}+ \|\pa_3 Y\|_{\wt
L^2_t(\mathcal{B}^{s+1,0})}\\
& +\|Y_t\|_{\wt{L}^2_t(\mathcal{B}^{s+1,0})} +\|
Y_t\|_{L^1_t(\mathcal{B}^{s+2,0})}+\|\na\vv p\|_{{L}^1_t(\cB^{s,0})}
\leq  C( c_0+ \frak{g}(s)) \quad\mbox{for}\ \ s=0,3.
\end{aligned}
\eeq If $(Y_0, Y_1)$ satisfies moreover that \beq \label{S2eq27}
\begin{split} \frak{g}(3)+
\|\p_3Y_0\|_{\cB^{-1,-\f12}_{\infty,\infty}}&+\|Y_0\|_{\cB^{1,-\f12}_{\infty,\infty}}+\|Y_1\|_{\cB^{-1,-\f12}_{\infty,\infty}}
\\
&+\|\p_3Y_0\|_{\cB^{3,-\f12}_{2,\infty}}+\|Y_0\|_{\cB^{5,-\f12}_{2,\infty}}+\|Y_1\|_{\cB^{3,-\f12}_{2,\infty}}\leq
\d_0 \end{split}  \eeq for some $\d_0$ sufficiently small,  then the
solution $Y$ of \eqref{1.2c} satisfies the following decay estimate
\beq\label{S2eq26}
\begin{split}
 \|Y_t(t)\|_{H^2} &+\|\p_3
Y(t)\|_{H^2}
 +\|\D Y(t)\|_{H^1}\\
 +\w{t}^{\f18}&\bigl(\|\p_3Y_t(t)\|_{H^1} +\|\p_3^2Y(t)\|_{H^1}
 +\|\p_3Y(t)\|_{\dot H^2}\bigr) \leq C\w{t}^{-\f14}. \end{split} \eeq }
\end{thm}

Let us remark that with more regularities on the initial data, we
can study the decay rate of the solution in higher Sobolev norms.
For a concise presentation, we shall not pursue this direction here.

\medskip

 \setcounter{equation}{0}
\section{Estimates related to Littlewood-Paley theory}\label{sect3}

The linearized system of \eqref{1.2b} reads
 \beq\label{B19}\left\{\begin{aligned}
&Y_{tt}-\Delta Y_t-\p_3^2 Y= f,\\
&Y|_{t=0}=Y_0,\quad Y_t|_{t=0}=Y_1.
\end{aligned}\right.\eeq
As observed in \cite{LZ,XLZMHD1,XZ15},  the corresponding symbolic
equation to \eqref{B19}, \beno
\la^2+|\xi|^2\la+\xi_3^2=0\quad\mbox{for}\quad
\xi=(\xi_\h,\xi_3)\quad\mbox{and}\quad \xi_\h=(\xi_1,\xi_2), \eeno
 has two different eigenvalues \beq\label{C1} \la_\pm
=-\f{|\xi|^2\pm \sqrt{|\xi|^4-4\xi_3^2}}{2}. \eeq The Fourier modes
correspond to $\la_+$ decays like $e^{-t|\xi|^2}$.  Whereas the
decay property of the Fourier modes corresponding to $\la_-$  varies
with directions of $\xi$ as \beq\label{C2}
\la_-(\xi)=-\f{2\xi_3^2}{|\xi|^2\bigl(1+\sqrt{1-\f{4\xi_3^2}{|\xi|^4}}\bigr)}
\to -1\quad \mbox{as}\quad |\xi|\to \infty \eeq only in the $\xi_3$
direction.  Thus in order to capture this delicate decay property
for the linear equation \eqref{B19}, we shall  decompose our
frequency space into two parts: $\bigl\{ \xi=(\xi_\h,\xi_3):\
|\xi|^2\leq 2|\xi_3|\ \bigr\}$ and $\bigl\{ \xi=(\xi_\h,\xi_3):\
|\xi|^2> 2|\xi_3|\ \bigr\}$. This suggests to use anisotropic
Littlewood-Paley theory in the analysis of \eqref{1.2b}.

In  order to obtain a better description of the regularizing effect
for the transport-diffusion equation, we will use Chemin-Lerner type
spaces $\widetilde{L}^{q}_T(\dB^s_{p,r}(\R^3))$ (see \cite{bcd} for
instance).

\begin{defi}\label{def3}
Let  $(r,q,p)\in[1,+\infty]^3$ and $T\in(0,+\infty]$.
 We define the norms of $\wt{L}^q_T(\dot{B}^s_{p,r}(\R^3))$  and  $\wt{L}^q_T(\cB^{s_1,s_2}(\R^3))$ by
\beno
&&\|u\|_{\wt{L}^q_T(\dot{B}^s_{p,r})}\eqdefa\Bigl(\sum_{j\in\Z}2^{jrs}
\|\D_ju\|_{L^q_T(L^p)}^r\Bigr)^{\f{1}{r}},\quad
\|u\|_{\wt{L}^q_T(\cB^{s_1,s_2})}\eqdefa\sum_{j,\ell\in\Z^2}2^{js_1}2^{\ell
s_2} \|\D_j\D_\ell^\v u\|_{L^q_T(L^2)}, \eeno with the usual change
if $r=\infty$.
\end{defi}

The connection between the Besov space $\dB^s$ and the anisotropic
Besov space $\cB^{s_1,s_2}$ can be illustrated by the following
Lemma:

\begin{lem}[Lemma
3.2 in \cite{XLZMHD1} and \cite{XZ15}] \label{L1} {\sl Let
$s_1,s_2,\tau_1,\tau_2\in\R,$ which satisfy $s_1<\tau_1+\tau_2<s_2$
and $\tau_2>0.$  Then $a\in\cB^{\tau_1,\tau_2}(\R^3)$  (resp.
$\wt{L}^2_T(\cB^{\tau_1,\tau_2})$) if $a\in \dB^{\tau_1+\tau_2}$
(resp. $a\in \wt{L}^2_T({B}^{\tau_1+\tau_2})$)  and there holds \beq
\label{C6}
\begin{split}
&\|a\|_{\cB^{\tau_1,\tau_2}}\lesssim\|a\|_{{B}^{\tau_1+\tau_2}}
\andf
\|u\|_{\wt{L}^2_T(\cB^{\tau_1,\tau_2})}\lesssim\|u\|_{\wt{L}^2_T({B}^{\tau_1+\tau_2})}
.\end{split} \eeq}
\end{lem}

 For the convenience of the readers, we  recall the
following Bernstein type lemma from \cite{bcd, CZ, Pa02}:

\begin{lem}\label{L2} {\sl Let $\frak{B}_{\h}$ (resp.~$\frak{B}_{\v}$) be a ball
of~$\R^2$ (resp.~$\R$), and~$\cC_{\h}$ (resp.~$\cC_{\v}$) a ring
of~$\R^2$ (resp.~$\R$); let~$1\leq p_2\leq p_1\leq \infty$ and
~$1\leq q_2\leq q_1\leq \infty.$ Then there holds:
\smallbreak\noindent If the support of~$\wh a$ is included
in~$2^k\frak{B}_{\h}$, then
\[
\|\partial_\h^\alpha a\|_{L^{p_1}_\h(L^{q_1}_\v)} \lesssim
2^{k\left(|\al|+2\left(\frac1{p_2}-\frac1{p_1}\right)\right)}
\|a\|_{L^{p_2}_\h(L^{q_1}_\v)}.
\]
If the support of~$\wh a$ is included in~$2^\ell\frak{B}_{\v}$, then
\[
\|\partial_3^\beta a\|_{L^{p_1}_\h(L^{q_1}_\v)} \lesssim
2^{\ell\left(\beta+\left(\frac1{q_2}-\frac1{q_1}\right)\right)} \|
a\|_{L^{p_1}_\h(L^{q_2}_\v)}.
\]
If the support of~$\wh a$ is included in~$2^k\cC_{\h}$, then
\[
\|a\|_{L^{p_1}_\h(L^{q_1}_\v)} \lesssim 2^{-kN}
\max_{|\al|=N}\|\partial_\h^\al a\|_{L^{p_1}_\h(L^{q_1}_\v)}.
\]
If the support of~$\wh a$ is included in~$2^\ell\cC_{\v}$, then
\[
\|a\|_{L^{p_1}_\h(L^{q_1}_\v)} \lesssim 2^{-\ell N} \|\partial_3^N
a\|_{L^{p_1}_\h(L^{q_1}_\v)}.
\]}
\end{lem}

As applications of the above basic facts on Littlewood-Paley theory,
we present the following product laws:

\begin{lem}[Lemma 3.3 of \cite{XZ15}]\label{L3}
{\sl Let $s_1,s_2,\tau_1,\tau_2\in\R,$ which satisfy  $s_1,
s_2\leq1$, $\tau_1,\tau_2\leq\f{1}{2}$ and $s_1+s_2>0$,
$\tau_1+\tau_2>0$. Then for $a\in\cB^{s_1,\tau_1}(\R^3)$ and
$b\in\cB^{s_2,\tau_2}(\R^3)$,
$ab\in\cB^{s_1+s_2-1,\tau_1+\tau_2-\f{1}{2}}(\R^3)$ and there holds
\beq\label{C9}
\|ab\|_{\cB^{s_1+s_2-1,\tau_1+\tau_2-\f{1}{2}}}\lesssim\|a\|_{\cB^{s_1,\tau_1}}\|b\|_{\cB^{s_2,\tau_2}}.
\eeq}
\end{lem}

\begin{rmk} Exactly along the same line to the proof of \eqref{C9}, we can show the
following law of product that for any $s>-1$ \beq \label{GH1}
\begin{split}
&\|a\na b\|_{\cB^{s,0}}\lesssim
\|a\|_{\cB^{1,\f12}}\|b\|_{\cB^{s+1,0}}+\|b\|_{\cB^{1,\f12}}\|a\|_{\cB^{s+1,0}},\\
&\|a b\|_{\cB^{s,0}}\lesssim
\|a\|_{\cB^{1,\f12}}\|b\|_{\cB^{s,0}}+\|b\|_{\cB^{1,\f12}}\|a\|_{\cB^{s,0}}.
\end{split}
 \eeq
We skip the details here.
\end{rmk}

\begin{lem}\label{PRO}
{\sl Let $s>-1$ and $\delta\in[0,1],$ then one has \ben &&
\|ab\|_{\cB^{s,-\f12}_{1,\infty}} \lesssim
\|a\|_{\cB^{1,0}}\|b\|_{\cB^{s,0}}
+\|a\|_{\cB^{s+\delta,0}}\|b\|_{\cB^{1-\delta,0}},\label{PROeq1}\\
&& \|ab\|_{\cB^{s',-\f12}_{\infty,\infty}} \lesssim
\|a\|_{\cB^{1,\f12}}\|b\|_{\cB^{s',-\f12}_{\infty,\infty}}\quad\mbox{for}\
\  \forall\ s'\in [-1,1].\label{PROeq2} \een}
\end{lem}
\begin{proof} Let us first recall the isentropic para-differential decomposition
of Bony from \cite{Bo}: let $a, b\in \cS'(\R^3),$  \beq
\label{C7}\begin{split} &
ab=T(a,b)+\bar{T}(a,b)+ R(a,b), \quad\hbox{where}\\
& T(a,b)\eqdefa\sum_{j\in\Z}S_{j-1}a\Delta_jb, \quad
\bar{T}(a,b)\eqdefa T(b,a), \andf\\
&R(a,b)\eqdefa\sum_{j\in\Z}\Delta_ja\tilde{\Delta}_{j}b,\quad\hbox{with}\quad
\tilde{\Delta}_{j}b\eqdefa\sum_{\ell=j-1}^{j+1}\D_\ell b.
\end{split} \eeq
 By using Bony's
decomposition \eqref{C7} for the whole space variables and the
vertical variable simultaneously, we obtain \beq \label{PROeq4}
\begin{split}
ab=&\bigl(T+\bar{T}+R\bigr)\bigl(T^\v+\bar{T}^\v+R^\v\bigr)(a,b)\\
=&\bigl(TT^\v+T\bar{T}^\v+TR^\v+\bar{T}T^\v+\bar{T}\bar{T}^\v+\bar{T}R^\v+RT^\v+R\bar{T}^\v+RR^\v\bigr)(a,b).
\end{split}
\eeq In what follows, we shall deal with the typical terms above. We
first deduce from Lemma \ref{L2} that \beno
\begin{split}
\|\D_j\D_\ell^\v(TR^\v(a,b))\|_{L^2}\lesssim
&2^{\frac{\ell}2}\sum_{\substack{|j'-j|\leq
4\\\ell'\geq\ell-N_0}}\|S_{j'-1}\D_{\ell'}^\v
a\|_{L^\infty_\h(L^2_\v)}\|\D_{j'}\D_{\ell'}^\v b\|_{L^2}\\
\lesssim &2^{\frac{\ell}2}\sum_{\substack{|j'-j|\leq
4\\\ell'\geq\ell-N_0}}d_{j',\ell'}2^{-j'
s}\|a\|_{\cB^{1,0}}\|b\|_{\cB^{s,0}}\\
\lesssim
&d_j2^{-js}2^{\frac{\ell}2}\|a\|_{\cB^{1,0}}\|b\|_{\cB^{s,0}}
\end{split}
\eeno The same estimate holds for $TT^\v(a,b)$ and
$T\bar{T}^\v(a,b).$

Similarly, we get, by applying  Lemma \ref{L2}, that \beno
\begin{split}
\|\D_j\D_\ell^\v(\bar{T}R^\v(a,b))\|_{L^2}\lesssim
&2^{\frac{\ell}2}\sum_{\substack{|j'-j|\leq
4\\\ell'\geq\ell-N_0}}\|\D_{j'}\D_{\ell'}^\v
a\|_{L^2}\|S_{j'-1}\D_{\ell'}^\v b\|_{L^\infty_{\h}(L^2_\v)}\\
\lesssim
&d_j2^{-js}2^{\frac{\ell}2}\|a\|_{\cB^{s+\d,0}}\|b\|_{\cB^{1-\d,0}}
\end{split}
\eeno The same estimate holds for $\bar{T}T^\v(a,b)$ and
$\bar{T}\bar{T}^\v(a,b).$

Finally due to $s>-1,$ we have \beno
\begin{split}
\|\D_j\D_\ell^\v(R\bar{R}^\v(a,b))\|_{L^2}\lesssim
&2^j2^{\frac{\ell}2}\sum_{\substack{j'\geq j-N_0
\\\ell'\geq\ell-N_0}}\|\D_{j'}\D_{\ell'}^\v
a\|_{L^2}\|\wt{\D}_{j'}\wt{\D}_{\ell'}^\v b\|_{L^2}\\
\lesssim &2^j2^{\frac{\ell}2}\sum_{\substack{j'\geq j-N_0
\\\ell'\geq\ell-N_0}}d_{j',\ell'}2^{-j'(
s+1)}\|a\|_{\cB^{1,0}}\|b\|_{\cB^{s,0}}\\
\lesssim
&d_j2^{-js}2^{\frac{\ell}2}\|a\|_{\cB^{1,0}}\|b\|_{\cB^{s,0}}.
\end{split}
\eeno The same estimate holds for $RT^\v(a,b)$ and
$R\bar{T}^\v(a,b).$

Hence in view of \eqref{PROeq4}, we achieve \eqref{PROeq1}. Exactly
along the same line, we can prove \eqref{PROeq2}, the detail of
which is omitted.
\end{proof}

In order to prove the large time decay estimates of the solutions to
\eqref{1.2c}, we need the following interpolation inequalities:

\begin{lem}\label{lem1.1}
{\sl Let $k\in \N$ and $f\in \cS(\R^3).$  Then one has \beno
\begin{split}
&(1)\ \quad \|f\|_{L^2} \lesssim
\|f\|_{\cB^{0,-\f12}_{2,\infty}}^{\f23}\|\p_3f\|_{L^2}^{\f13},
\\&(2)\ \quad
\|\pa_3f\|_{L^2} \lesssim
\|f\|_{\cB^{1,-\f12}_{\infty,\infty}}^{\f23}\|\na\p_3f\|_{L^2}^{\f13},
\\&(3)\ \quad
\|\na^kf\|_{L^2} \lesssim
\|f\|_{\cB^{\f{3k}2,-\f12}_{2,\infty}}^{\f23}\|\p_3f\|_{L^2}^{\f13}
\andf\\&(4)\ \quad  \|\na^kf\|_{L^2} \lesssim
\|f\|_{\cB^{\f{3k-1}2,-\f12}_{2,\infty}}^{\f23}\|\na\p_3f\|_{L^2}^{\f13}.
\end{split}
\eeno}
\end{lem}

\begin{proof} Note that by  virtue of Definition \ref{def2}, for any fixed integer $N,$ one has
\beq\label{lem1.1eqpo}
\begin{split}
\|f\|_{L^2}^2\sim &\sum_{(j,\ell)\in \Z^2}\|\D_j\D_\ell^\v
f\|_{L^2}^2=\sum_{\substack{\ell\leq N\\j\in \Z}}\|\D_j\D_\ell^\v
f\|_{L^2}^2+\sum_{\substack{\ell> N\\j\in \Z}}\|\D_j\D_\ell^\v
f\|_{L^2}^2\\
\lesssim &\sum_{\substack{\ell\leq N\\j\in
\Z}}c_j^22^{\ell}\|f\|_{\cB^{0,-\f12}_{2,\infty}}^2+\sum_{\substack{\ell>
N\\j\in \Z}}c_j^22^{-2\ell}\|\p_3f\|_{L^2}^2\\
\lesssim &2^N\|f\|_{\cB^{0,-\f12}_{2,\infty}}^2+
2^{-2N}\|\p_3f\|_{L^2}^2.
\end{split}
\eeq Here and in all that follows, we always denote $(c_j)_{j\in\Z}$
to be a generic element of $\ell^2(\Z)$ so that
$\sum_{j\in\Z}c_j^2=1.$

 Let us now  choose the integer $N$ in \eqref{lem1.1eqpo} so that \beno
2^N\sim
\biggl(\f{\|\p_3f\|_{L^2}^2}{\|f\|_{\cB^{0,-\f12}_{2,\infty}}^2}\biggr)^{\f13},
\eeno leads to (1) of Lemma \ref{lem1.1}.

To prove (2) of Lemma \ref{lem1.1}, we first deduce from Proposition
2.22 of \cite{bcd} that \beq \label{lem1.1eqa} \|\pa_3f\|_{L^2}
\lesssim \|\pa_3f\|_{
\dB^{-\f12}_{2,\infty}}^{\f23}\|\na\pa_3f\|_{L^2}^{\f13}. \eeq Yet
by virtue of Definition \ref{def1}, we have
$$
\begin{aligned}
\|\pa_3f\|_{
\dB^{-\f12}_{2,\infty}}&=\sup_{j}2^{-\f{j}{2}}\|\D_j\pa_3f\|_{L^2}
\lesssim \sup_{j}2^{-\f{j}{2}}\sum_{\ell\le
j+N_0}\|\D_j\D^\v_{\ell}\pa_3f\|_{L^2}
\\&
\lesssim \sup_{j}2^{-\f{j}{2}}\sum_{\ell\le
j+N_0}2^{\f{3\ell}{2}}2^{-\f{\ell}{2}} \|\D_j\D^\v_{\ell}f\|_{L^2}
\\&
\lesssim \sup_{j}2^{j}\sup_{\ell}2^{-\f{\ell}{2}}
\|\D_j\D^\v_{\ell}f\|_{L^2}=\|f\|_{\cB^{1,-\f12}_{\infty,\infty}}.
\end{aligned}
$$ Resuming the above estimate into \eqref{lem1.1eqa}
gives rise to the second inequality of Lemma \ref{lem1.1}.

Along the same line to \eqref{lem1.1eqpo}, for any integer $k,$ we
write \beno
\begin{split}
\|\na^k f\|_{L^2}^2\sim &\sum_{\substack{\ell-kj\leq N\\j\in
\Z}}2^{2kj}\|\D_j\D_\ell^\v f\|_{L^2}^2+\sum_{\substack{\ell-kj>
N\\j\in \Z}}2^{2kj}\|\D_j\D_\ell^\v
f\|_{L^2}^2\\
\lesssim &\sum_{\substack{\ell-kj\leq N\\j\in
\Z}}c_j^22^{(\ell-kj)}\|f\|_{\cB^{\f{3k}2,-\f12}_{2,\infty}}^2+\sum_{\substack{\ell-kj>
N\\j\in \Z}}c_j^22^{-2(\ell-kj)}\|\p_3f\|_{L^2}^2\\
\lesssim &2^N\|f\|_{\cB^{\f{3k}2,-\f12}_{2,\infty}}^2+
2^{-2N}\|\p_3f\|_{L^2}^2.
\end{split}
\eeno Taking  $N$ in the above inequality so that \beno 2^N\sim
\biggl(\f{\|\p_3f\|_{L^2}^2}{\|f\|_{\cB^{\f{3k}2,-\f12}_{2,\infty}}^2}\biggr)^{\f13},\eeno
leads to (3) of  Lemma \ref{lem1.1}.

Finally a direct application of (3) with $f$ (resp. $k$) there being
replaced by $\na f$ (resp. $k-1$) leads to the last inequality of
Lemma \ref{lem1.1}. This completes the proof of the lemma.
\end{proof}

\begin{lem}\label{S3lem1}
{\sl Let $s\in \R$ and $b\in\cS(\R^3),$   one has \beq\label{S3eq1}
\|b\|_{\dot W^{s,4}}\lesssim
\|b\|_{\cB^{1+2s,-\f12}_{2,\infty}}^{\f12}\|\p_3b\|_{L^2}^{\f12}.\eeq}
\end{lem}
\begin{proof} Note that $\dot B^0_{p,2}\hookrightarrow L^p$ for
$p\in [2,\infty[$ (see Theorem 2.40 of \cite{bcd}), we have \beno
\begin{split}\|b\|_{\dot W^{s,4}}=\||D|^sb\|_{L^4}\lesssim&
\||D|^sb\|_{{\dB}^0_{4,2}}\\
=&\Bigl(\sum_{j\in\Z}2^{2js}\|\D_j b\|_{L^4}^2\Bigr)^{\f12}=\Bigl(\sum_{j\in\Z}2^{2js}\bigl\|\bigl(\sum_{\ell\in\Z}|\D_j \D_\ell^\v b|^2\bigr)^{\f12}\bigr\|_{L^4}^2\Bigr)^{\f12}\\
\lesssim & \Bigl(\sum_{(j,\ell)\in\Z^2}2^{2js}\|\Djl
b\|_{L^4}^2\Bigr)^{\f12}.
\end{split}
\eeno For any integer $N,$ we get, by applying Lemma \ref{L2}, that
\beno
\begin{split}
\sum_{(j,\ell)\in\Z^2}2^{2js}\|\Djl b\|_{L^4}^2 \lesssim &
\sum_{(j,\ell)\in\Z^2}2^{j(1+2s)}2^{\f\ell2}\|\Djl b\|_{L^2}^2\\
\lesssim
&\sum_{\substack{\ell-\f23(1+2s)j\leq N\\
j\in\Z}}c_j^22^{\f32\bigl(\ell-\f23(1+2s)j\bigr)}\|b\|_{\cB^{1+2s,-\f12}_{2,\infty}}^2\\
&+
\sum_{\substack{\ell-\f23(1+2s)j> N\\
j\in\Z}}c_j^22^{-\f32\bigl(\ell-\f23(1+2s)j\bigr)}\|\p_3b\|_{L^2}^2\\
\lesssim&
2^{\f32N}\|b\|_{\cB^{1+2s,-\f12}_{2,\infty}}^2+2^{-\f32N}\|\p_3b\|_{L^2}^2.
\end{split}
\eeno Choosing $N$ in the above inequality so that \beno 2^N\sim
\Bigl(\f{\|\p_3b\|_{L^2}^2}{\|b\|_{\cB^{1+2s,-\f12}_{2,\infty}}^2}\Bigr)^{\f13}
\eeno leads to \eqref{S3eq1}. This finishes the proof of the lemma.
\end{proof}

\setcounter{equation}{0}
\section{$L^1_T(\cB^{2,\f12})$ estimate of $\bY_t$}\label{sect4}

Let $\bar{Y}$ be a smooth enough solution of \eqref{1.2b} on
$[0,T].$ The goal of this section is to present the {\it a priori}
$L^1_T(\mbox{Lip})$ estimate of $\bar{Y}_t.$ Instead of handling the
$L^1_T(\cB^{\f52,0})$ norm of $\bY_t$ as that in
\cite{XLZMHD1,XZ15}, here we shall deal with the
$L^1_T(\cB^{2,\f12})$ norm of $\bY_t.$ For simplicity, we shall
denote $\dv_z=\dv,$ $\nabla_z=\na$ and $\Delta_z=\Delta$ for short
in this section.

\begin{lem}\label{S3-Lem1}
{\sl  Let $Y$ is a smooth enough solution of \eqref{B19} on $[0,T].$
Then for $t\leq T,$ one has \beq
\begin{split}\label{2.1}
\|Y_t\|_{\wt L^\infty_t(\mathcal{B}^{0,\f12})} &+\|\pa_3Y\|_{\wt
L^\infty_t(\mathcal{B}^{0,\f12})} +\|\D Y\|_{\wt
L^\infty_t(\mathcal{B}^{0,\f12})}+\|Y_t\|_{\wt{L}^2_t(\cB^{1,\f12})}+\|\p_3
Y\|_{\wt{L}^2_t(\cB^{1,\f12})}\\
&+\|Y_t\|_{L^1_t(\mathcal{B}^{2,\f12})} \lesssim
\|Y_1\|_{\mathcal{B}^{0,\f12}}+
\|\pa_3Y_0\|_{\mathcal{B}^{0,\f12}}+\|\D
Y_0\|_{\mathcal{B}^{0,\f12}} +\| f\|_{L^1_t(\mathcal{B}^{0,\f12})}.
\end{split}
\eeq }
\end{lem}

\begin{proof} The proof of this lemma basically follows from
Proposition 4.1 of \cite{XLZMHD1, XZ15}. For completeness, we
present the details here. By applying the operator $\D_j\D_\ell^\v$
to \eqref{B19} and then taking the $L^2$ inner product of the
resulting equation with $\D_j\D_\ell^\v Y_t,$ we write
\beq\label{S4eq10}
\begin{aligned}
\f12\f{d}{dt}\bigl(\|\D_j\D_\ell^\v
Y_t\|_{L^2}^2+\|\D_j\D_\ell^\v\pa_{3}Y\|_{L^2}^2\bigr) +\|\na
\D_j\D_\ell^\v{Y}_t\|_{L^2}^2 =\bigl(\D_j\D_\ell^\v f |
\D_j\D_\ell^\v Y_t\bigr)_{L^2}.
\end{aligned}
\eeq Along the same line, one has
 \beno (\D_j\D_\ell^\v Y_{tt} |
\D\D_j\D_\ell^\v Y)-\f12\f{d}{dt}\|\D\D_j\D_\ell^\v
Y\|_{L^2}^2-\|\pa_3\na\D_j\D_\ell^\v Y\|_{L^2}^2=(\D_j\D_\ell^\v f |
\D\D_j\D_\ell^\v Y). \eeno Notice that \beno (\D_j\D_\ell^\v Y_{tt}
| \D\D_j\D_\ell^\v Y)=\f{d}{dt}(\D_j\D_\ell^\v Y_{t}|
\D\D_j\D_\ell^\v Y) +\|\na\D_j\D_\ell^\v Y_{t}\|_{L^2}^2, \eeno so
that there holds \beq\label{d4}
\begin{split}
\f{d}{dt}\Bigl(\f12\|\D\D_j\D_\ell^\v Y\|_{L^2}^2-&(\D_j\D_\ell^\v
Y_{t} |
\D\D_j\D_\ell^\v Y)\Bigr)\\
-&\|\na\D_j\D_\ell^\v Y_t\|_{L^2}^2+\|\pa_3\na\D_j\D_\ell^\v
Y\|_{L^2}^2 =-(\D_j\D_\ell^\v f |\D\D_j\D_\ell^\v Y).
\end{split}
\eeq By summing up \eqref{S4eq10} with $\f14$ of \eqref{d4}, we
obtain \beq\label{d5}
\begin{split}
\f{d}{dt}g_{j,\ell}^2(t) +\f34\|\na\D_j\D_\ell^\v Y_t\|_{L^2}^2&
+\f14\|\pa_3\na\D_j\D_\ell^\v Y\|_{L^2}^2\\
&=\bigl(\D_j\D_\ell^\v  f\ |\ \D_j\D_\ell^\v
Y_t-\f14\D\D_j\D_\ell^\v Y\bigr),
\end{split}
\eeq where
 \beno\begin{split}
g_{j,\ell}^2(t) \eqdefa \f12\Bigl(\|\D_j\D_\ell^\v Y_t(t)\|_{L^2}^2&
+\|\D_j\D_\ell^\v \pa_3Y(t)\|_{L^2}^2 +\f14\|\D_j\D_\ell^\v \D
Y(t)\|_{L^2}^2\Bigr)
\\&\qquad\qquad\qquad
-\f14\bigl(\D_j\D_\ell^\v Y_t(t) | \D_j\D_\ell^\v \D Y(t)\bigr).
\end{split}
\eeno It is easy to observe that
 \beq \label{d6}
g_{j,\ell}^2(t) \sim \|\D_j\D_\ell^\v Y_t(t)\|_{L^2}^2
+\|\D_j\D_\ell^\v\pa_3 Y(t)\|_{L^2}^2 +\|\D_j\D_\ell^\v \D
Y(t)\|_{L^2}^2. \eeq Now according to the heuristic analysis
presented at the beginning of Section \ref{sect3}, we split the
frequency analysis into the following two cases:

\no$\bullet$ \underline{When $j\leq \f{\ell+1}2$}

 In this case,
one has \beno g^2_{j,\ell}(t)\sim  \|\D_j\D_\ell^\v Y_t(t)\|_{L^2}^2
+\|\D_j\D_\ell^\v \pa_3Y(t)\|_{L^2}^2, \eeno and Lemma \ref{L2}
implies that \beno \f34\|\na \D_j\D_\ell^\v Y_t\|_{L^2}^2
+\f14\|\pa_3\na\Djl Y\|_{L^2}^2 \geq  c2^{2j}\bigl(\|\D_j\D_\ell^\v
Y_t\|_{L^2}^2 +\|\Djl\p_3 Y\|_{L^2}^2\bigr). \eeno Hence it follows
from \eqref{d5} that \beq \label{S4eq8}
\begin{split} \|\Djl Y_t\|_{L^\infty_t(L^2)}&+\|\Djl\pa_3 Y\|_{L^\infty_t(L^2)}
+\|\Djl\D Y\|_{L^\infty_t(L^2)}
\\&
+2^{2j}\bigl(\|\Djl Y_t\|_{L^1_t(L^2)}+\|\Djl \pa_3 Y\|_{L^1_t(L^2)}\bigr)\\
\lesssim &\|\Djl Y_1\|_{L^2} +\|\Djl\pa_3 Y_0\|_{L^2}+\|\Djl
f\|_{L^1_t(L^2)}.
\end{split}
\eeq

\no$\bullet$\underline{When $j> \f{\ell+1}2$}

In this case, we have \beno g_{j,\ell}^2(t) \sim \|\D_j\D_\ell^\v
Y_t(t)\|_{L^2}^2 +\|\Djl\D Y(t)\|_{L^2}^2 \eeno and Lemma \ref{L2}
implies that \beno
\begin{split}
\f34\|\na \D_j\D_\ell^\v Y_t\|_{L^2}^2 +\f14\|\pa_3\na\Djl
Y\|_{L^2}^2 \geq & c\bigl(2^{2j}\|\D_j\D_\ell^\v Y_t\|_{L^2}^2
+2^{2j}2^{2\ell}\|\Djl Y\|_{L^2}^2\bigr)\\
\geq & c\f{2^{2\ell}}{2^{2j}}\bigl(\|\D_j\D_\ell^\v Y_t\|_{L^2}^2
+\|\Djl\D Y\|_{L^2}^2\bigr).
\end{split}
\eeno Then we deduce from \eqref{d5} that \beq\label{S4eq9}
\begin{split} \|\Djl Y_t\|_{L^\infty_t(L^2)}
&+\|\Djl\pa_3 Y\|_{L^\infty_t(L^2)} +\|\Djl\D Y\|_{L^\infty_t(L^2)}
\\&
+c\Bigl(\f{2^{2\ell}}{2^{2j}}\|\Djl Y_t\|_{L^1_t(L^2)}+2^{2\ell}\|\Djl Y\|_{L^1_t(L^2)}\Bigr)\\
\lesssim &\|\Djl Y_1\|_{L^2}+\|\D\Djl Y_0\|_{L^2}+\|\Djl
f\|_{L^1_t(L^2)}.
\end{split}
\eeq On the other hand,  it is easy to observe from \eqref{B19} that
$$
\f12\f{d}{dt}\|\D_j\D_\ell^\v Y_t\|_{L^2}^2+\|\na \D_j\D_\ell^\v
Y_t\|_{L^2}^2 =\bigl(\pa_3^2\D_j\D_\ell^\v Y+\D_j\D_\ell^\v f\ |\
\D_j\D_\ell^\v Y_t\bigr)_{L^2},
$$
from which, Lemma \ref{L2} and \eqref{S4eq9}, we deduce that for $
j>\f{\ell+1}2$ \beq\label{d9}
\begin{split}
\|\Djl Y_t\|_{L^\infty_t(L^2)}&+2^{2j}\|\Djl Y_t\|_{L^1_t(L^2)}\\
&\lesssim
\|\Djl Y_1\|_{L^2}+2^{2\ell}\|\Djl Y\|_{L^1_t(L^2)}+\|\Djl f\|_{L^1_t(L^2)}\\
&\lesssim \|\Djl Y_1\|_{L^2}+\|\D\Djl Y_0\|_{L^2}+\|\Djl
f\|_{L^1_t(L^2)}.
\end{split}
\eeq In view of \eqref{S4eq8}-\eqref{d9}, we obtain for all
$(j,\ell)\in\Z^2,$  that \beq\label{S4q10qp}
\begin{split}
\|\Djl Y_t\|_{L^\infty_t(L^2)}+&\|\Djl\pa_3 Y\|_{L^\infty_t(L^2)}
+\|\Djl\D Y\|_{L^\infty_t(L^2)}
+2^{2j}\|\Djl Y_t\|_{L^1_t(L^2)}\\
\lesssim & \|\Djl Y_1\|_{L^2}+\|\Djl\pa_3 Y_0\|_{L^2}+\|\Djl\D
Y_0\|_{L^2}+\|\Djl f\|_{L^1_t(L^2)}.
\end{split}
\eeq Whereas by integrating \eqref{d5} over $[0,t],$ we get \beno
\begin{split}
&\|g_{j,\ell}^2\|_{L^\infty(0,t)} +\f34\|\na\D_j\D_\ell^\v
Y_t\|_{L^2_t(L^2)}^2
+\f14\|\na\D_j\D_\ell^\v\pa_3 Y\|_{L^2_t(L^2)}^2\\
&\leq  \|\Djl
Y_1\|_{L^2}^2+\|\Djl\pa_3 Y_0\|_{L^2}^2+\|\Djl\D Y_0\|_{L^2}^2\\
&\qquad+ \|\D_j\D_\ell^\v f\|_{L^1_t(L^2)}\bigl(\|\D_j\D_\ell^\v
Y_t\|_{L^\infty_t(L^2)}+\f14\|\D\D_j\D_\ell^\v Y\|_{L^\infty_t(L^2)}\bigr)\\
&\leq \|\Djl Y_1\|_{L^2}^2+\|\Djl\pa_3 Y_0\|_{L^2}^2+\|\Djl\D
Y_0\|_{L^2}^2+C\|\D_j\D_\ell^\v
f\|_{L^1_t(L^2)}^2+\f12\|g_{j,\ell}^2\|_{L^\infty(0,t)},
\end{split}
\eeno which in particular gives rise to $$\longformule{
\|\na\D_j\D_\ell^\v Y_t\|_{L^2_t(L^2)}+\|\na\D_j\D_\ell^\v\pa_3
Y\|_{L^2_t(L^2)}}{{}\lesssim  \|\Djl Y_1\|_{L^2}+\|\Djl\pa_3
Y_0\|_{L^2}+\|\Djl\D Y_0\|_{L^2}+\|\Djl f\|_{L^1_t(L^2)}.} $$
Summing up the above inequality with  \eqref{S4q10qp} and
multiplying the inequality by $2^{\f{\ell}2}$ and then summing up
the resulting inequality for $(j,\ell)\in\Z^2,$ we achieve
\eqref{2.1}. This completes the proof of the lemma.
\end{proof}

\begin{prop}\label{Lip} {\sl  Let $\bY$ is a smooth enough solution of \eqref{1.2b} on
$[0,T].$ Then there exist sufficiently small positive constants,
$c_0, \e_0,$ so that if \beq\label{Lip0} \|Y_1\|_{\cB^{0,\f12}}
+\|\p_3\bY_0\|_{\cB^{0,\f12}} +\| \bY_0\|_{\cB^{2,\f12}} \le c_0
\andf \e\leq \e_0, \eeq   we have \beq\label{Lip1}\begin{split} \|\b
Y_t\|_{\wt L^\infty_t(\mathcal{B}^{0,\f12})}& +\|\pa_3\b Y\|_{\wt
L^\infty_t(\mathcal{B}^{0,\f12})} +\|\b Y\|_{\wt
L^\infty_t(\mathcal{B}^{2,\f12})}+\|\b Y_t\|_{\wt
L^2_t(\mathcal{B}^{1,\f12})} +\|\pa_3\b Y\|_{\wt
L^2_t(\mathcal{B}^{1,\f12})}\\
& +\|\b Y_t\|_{L^1_t(\mathcal{B}^{2,\f12})}+\|\na \vv
p\|_{L^1_t(\cB^{0,\f12})} \leq  C\Bigl(\|\b
Y_1\|_{\mathcal{B}^{0,\f12}} +\|\pa_3\b
Y_0\|_{\mathcal{B}^{0,\f12}}+\|\b Y_0\|_{\mathcal{B}^{2,\f12}}
\Bigr)
\end{split}
\eeq for any $ t\leq T.$ }
\end{prop}

\begin{proof}  Let us denote \beq \label{Lip2} T^\star\eqdefa
\sup\bigl\{ \ t\in [0,T]:\,\,
\|\na\bY\|_{L^\infty_t({\cB}^{1,\f12})}\leq \d\,\, \bigr\}. \eeq We
shall prove that for $\e_0$ and $c_0$ sufficiently small,
$T^\star=T.$

According to Lemma \ref{S3-Lem1}, it remains to estimate
$\|f\|_{L^1_t(\cB^{0,\f12})}$ for $f$ given by \eqref{S2-27}. Toward
this and in view of \eqref{S2-27},  we decompose $f$  as
\beq\label{S4eq12}
\begin{split}
f=& f_1+f_2+f_3\with\\
f_1\eqdefa& \cB^{t}\na\cdot\left(\bigl((\mathcal{A}-Id)
(\mathcal{A}-Id)^{t}+(\mathcal{A}-Id)+(\mathcal{A}-Id)^{t}\bigr)\cB^{t}\na\b
Y_t\right),\\
f_2\eqdefa &\cB^{t}\na\cdot\bigl(({\cB}-Id)^{t}\na\b Y_t\bigr)
+({\cB}-Id)^{t}\Delta\b Y_t,\\
f_3\eqdefa &-({\cB}\mathcal{A})^{t}\na\vv p,
\end{split}
\eeq where the matrix $\cB$ and $\na\vv p$ are determined
respectively by \eqref{1.2fk} and \eqref{PR}.

On the other hand,  in view of \eqref{S2eq13} and \eqref{CH},
 for $b_0=e_3+\e\phi$ with $\phi$ satisfying \eqref{S2eq5},  we have \beno
\begin{split}
|z_3-w_3(z)|\leq
&\e\int_0^{w_3(z)}\f{|\phi_3(y_\h(w_\h,w_3'),w_3')|}{1-\e|\phi_3(y_\h(w_\h,w_3'),w_3')|}\,dw_3'\\
\leq &2\e\int_0^K|\phi_3(y_\h(w_\h,w_3'),w_3')|\,dw_3'\leq 2\e
K\|\phi_3\|_{L^\infty},
\end{split}
\eeno whenever $\e\leq \f1{2\|\phi_3\|_{L^\infty}}.$  This proves
that as long as $\e\leq \e_1\eqdefa \min\left(
\f1{4\|\phi_3\|_{L^\infty}}, \f1{4K\|\phi_3\|_{L^\infty}}\right),$
there holds \beq\label{S4eq14qe} |z_3-w_3(z)|\leq \f12. \eeq Now for
$K$ given by \eqref{S2eq5}, let us introduce a smooth cut-off
function $\eta(z_3)$ so that $ \eta(z_3)= \left\{\begin{array}{l}
\displaystyle 0,\,\,z_3\geq 2+K, \\
\displaystyle 1,\,\,-1\leq z_3\leq 1+K, \\
\displaystyle 0,\,\,z_3\leq -2.
\end{array}\right. $ Then thanks to \eqref{S2eq13}, \eqref{CH} and \eqref{S4eq14qe}, we split $A_2(y(w(z)))$ given by \eqref{CHab} as \beq
\label{S4eq14}\begin{split}
 A_2(y(w(z)))=&A_{2,1}(z)+A_{2,2}(z)\with A_{2,2}(z)\eqdefa (1-\eta(z_3))A_2^\h(z_\h),\end{split}
\eeq and \beno \begin{split}
 A_{2,1}(z)\eqdefa&\eta(z_3)\begin{pmatrix}\int_0^{w_3(z)}\f{\p}{\p
y_1}\bigl(\f{b_0^1}{b_0^3}\bigr)(y_\h(z_\h,y_3),y_3)dy_3&\int_0^{w_3(z)}\f{\p}{\p
y_2}\bigl(\f{b_0^1}{b_0^3}\bigr)(y_\h(z_\h,y_3),y_3)dy_3&0 \\
\int_0^{w_3(z)}\f{\p}{\p y_1}\bigl(\f{b_0^2}{b_0^3}\bigr)(y_\h(z_\h,y_3),y_3)dy_3 & \int_0^{w_3(z)}\f{\p}{\p y_2}\bigl(\f{b_0^2}{b_0^3}\bigr)(y_\h(z_\h,y_3),y_3)dy_3 & 0\\
0 & 0 &  0 \end{pmatrix},  \\
A_2^\h(z_\h)\eqdefa& \begin{pmatrix}\int_0^{K}\f{\p}{\p
y_1}\bigl(\f{b_0^1}{b_0^3}\bigr)(y_\h(z_\h,y_3),y_3)dy_3&\int_0^{K}\f{\p}{\p
y_2}\bigl(\f{b_0^1}{b_0^3}\bigr)(y_\h(z_\h,y_3),y_3)dy_3&0 \\
\int_0^{K}\f{\p}{\p y_1}\bigl(\f{b_0^2}{b_0^3}\bigr)(y_\h(z_\h,y_3),y_3)dy_3 & \int_0^{K}\f{\p}{\p y_2}\bigl(\f{b_0^2}{b_0^3}\bigr)(y_\h(z_\h,y_3),y_3)dy_3 & 0\\
0 & 0 &  0 \end{pmatrix}.
\end{split}
\eeno Similarly, we decompose $A_3(w(z)),$ which is given by
\eqref{1.2fh}, as \beq \label{S4eq15}\begin{split}
 A_3(w(z))=&A_{3,1}(z)+A_{3,2}(z)\with A_{3,2}(z)\eqdefa
 (1-\eta(z_3))A_3^\h(z_\h),\end{split}
\eeq and \beno
\begin{split}
 A_{3,1}(z)\eqdefa& \begin{pmatrix}
1&0& 0\\
0&1&0\\
\eta(z_3)\int_{0}^{w_3(z)}\f{\p}{\p
z_1}\bigl(\f{1}{b_0^3(y_\h(z_\h,y_3),y_3)}\bigr)dy_3 &
\eta(z_3)\int_{0}^{w_3(z)}\f{\p}{\p
z_2}\bigl(\f{1}{b_0^3(y_\h(z_\h,y_3),y_3)}\bigr)dy_3 &
\f{1}{b_0^3} \end{pmatrix}, \\
 A_3^\h(z_\h)\eqdefa &
\begin{pmatrix}
0&0& 0\\
0&0&0\\
\int_{0}^{K}\f{\p}{\p
z_1}\bigl(\f{1}{b_0^3(y_\h(z_\h,y_3),y_3)}\bigr)dy_3 &
\int_{0}^{K}\f{\p}{\p
z_2}\bigl(\f{1}{b_0^3(y_\h(z_\h,y_3),y_3)}\bigr)dy_3 & 0
\end{pmatrix}.
\end{split}
\eeno Then by virtue of \eqref{1.2fk}, \eqref{S4eq14} and
\eqref{S4eq15}, we find \beq \label{S4eq16}\begin{split}
\cB-Id=&\underbrace{\bigl(\frak{A}_1^{-1}-Id\bigr)+\bigl(A_{3,1}-Id\bigr)+\bigl(\frak{A}_1^{-1}-Id\bigr)\bigl(A_{3,1}-Id\bigr)
+A_{3,1}\frak{A}_1^{-1}A_{2,1}}_{\cB_1}\\
&+\underbrace{A_{3,2}\frak{A}_1^{-1}+A_{3,1}\frak{A}_1^{-1}A_{2,2}+A_{3,2}\frak{A}_1^{-1}\bigl(A_{2,1}+A_{2,2}\bigr)}_{\cB_2},
\end{split}
\eeq where $\frak{A}_1(z)\eqdefa A_1(y_\h(z_\h,w_3(z)),w_3(z)).$

\begin{lem}\label{S4lem3}
{\sl Under the assumptions of Theorem \ref{thm1}, there exists a
sufficiently small constant $\e_{2}\leq \e_1,$ which depends  on
$\|\na \phi\|_{W^{2,\infty}},$ $\|\na\phi\|_{H^2},$ and
$\|\na_\h\phi\|_{L^\infty_\v(H^2_\h)},$ so that for $\e\leq \e_2,$
we have \beq\label{S4eq30}
\begin{split}
\|\frak{A}_1-Id\|_{\cB^{1,\f12}}+\|A_{2,1}\|_{\cB^{1,\f12}}+&\|A_{3,1}-Id\|_{\cB^{1,\f12}}+\|A_{2}^\h\|_{\dB^1_\h}+\|A_{3}^\h\|_{\dB^1_\h}\\
&\qquad\leq
C\e\bigl(\|\na\phi\|_{H^2}+\|\na_\h\phi\|_{L^\infty_\v(H^2_\h)}\bigr)\leq
1.
\end{split} \eeq}\end{lem}

 While it follows from \eqref{S4eq14qe} and the definition of the cut-off function,
$\eta(z_3),$ that \beno
\begin{split}
(1-\eta(z_3))\int_{-1}^{z_3}\bigl(e_3-&b_0(y_\h(z_\h,
w_3(z_\h,z_3')),w_3(z_\h,z_3'))\bigr)\,dz_3'\\
&=(1-\eta(z_3))\int_{-1}^{K+1}\bigl(e_3-b_0(y_\h(z_\h,
w_3(z_\h,z_3')),w_3(z_\h,z_3'))\bigr)\,dz_3', \end{split} \eeno so
that we deduce from \eqref{1.2fk} and \eqref{1.2a} that  \beq
\label{S4eq17}\begin{split}
\wt{Y}(z)=&\int_{-1}^{z_3}\bigl(e_3-b_0(y(w(z_\h,z_3')))\bigr)\,dz_3'-\int_{-1}^{K+1}\bigl(e_3-b_0(y(w(z_\h,z_3')))\bigr)\,dz_3'\\
=&\eta(z_3)\Bigl(\int_{-1}^{z_3}\bigl(e_3-b_0(y(w(z_\h,z_3')))\bigr)\,dz_3'
-\int_{-1}^{K+1}\bigl(e_3-b_0(y(w(z_\h,z_3')))\bigr)\,dz_3'\Bigr).
\end{split} \eeq

\begin{lem}\label{S4lem4}
{\sl Under the assumptions of Theorem \ref{thm1} and for $\e\leq
\e_3$ with $\e_3\leq \e_2$ and depending  on
$\|\na\phi\|_{W^{2,\infty}},$ $\|\phi\|_{H^3}$ and $
\|\phi\|_{L^\infty_\v(H^3_\h)},$ one has \beq\label{S4eq17a}
\|\p_3\wt{Y}\|_{\cB^{0,\f12}}+ \|\na\wt{Y}\|_{\cB^{1,\f12}}\leq
C\e\bigl(\|\phi\|_{H^3}+\|\phi\|_{L^\infty_\v(H^3_\h)}\bigr)\leq 1.
\eeq }
\end{lem}

We shall postpone the proof of the above two lemmas in the Appendix
\ref{appenda}.

With the above preparations, we now present the estimate of
$\|f\|_{L^1_t(\cB^{0,\f12})}.$ Since the estimate to all the terms
in $f_1$ and $f_2$ given by \eqref{S4eq12} are the same type, let us
present the detailed estimate to the following term: \beno
\cB^{t}\na\cdot\left((\mathcal{A}-Id)^{t}\cB^{t}\na\b Y_t\right)
=\bigl(Id+\cB_1^{t}+\cB_2^t\bigr)\na\cdot\left(\bigl(({Id+\na
\bar{Y}+\na\wt{Y}})^{-1}-Id\bigr)^t\cB^{t}\na\b Y_t\right) \eeno for
$\cB_1,\cB_2$ given by \eqref{S4eq16}.

 It follows from the law of
product, Lemma \ref{L3}, and \eqref{S4eq30} that  for $\e\leq \e_2,$
\beno
\begin{split}
\Bigl\|\bigl(Id+\cB_1^{t}\bigr)\na&\cdot\left(\bigl(({Id+\na
\bar{Y}+\na\wt{Y}})^{-1}-Id\bigr)^t\cB^t\na\b
Y_t\right)\Bigr\|_{L^1_t(\cB^{0,\f12})}\\
\lesssim
&\bigl(1+\|\cB_1\|_{\cB^{1,\f12}}\bigr)\bigl\|\bigl(({Id+\na
\bar{Y}+\na\wt{Y}})^{-1}-Id)^{t}\cB^t\na\b
Y_t\bigr\|_{L^1_t(\cB^{1,\f12})}\\
\lesssim &\bigl\|\bigl(({Id+\na
\bar{Y}+\na\wt{Y}})^{-1}-Id)^{t}\cB^t\na\b
Y_t\bigr\|_{L^1_t(\cB^{1,\f12})},
\end{split}
\eeno and \beno
\begin{split}
\Bigl\|\bigl(A_{3,2}&\frak{A}_1^{-1}\bigr)^t\na\cdot\left(\bigl(({Id+\na
\bar{Y}+\na\wt{Y}})^{-1}-Id\bigr)^{t}\cB^t\na\b
Y_t\right)\Bigr\|_{L^1_t(\cB^{0,\f12})}\\
\lesssim
&\bigl(1+\|\frak{A}_1-Id\|_{\cB^{1,\f12}}\bigr)\|A_3^\h\|_{\dB^{1}_\h}\bigl\|(1-\eta(z_3))\na\cdot\bigl(({Id+\na
\bar{Y}+\na\wt{Y}})^{-1}-Id\bigr)^{t}\cB^t\na\b
Y_t\bigr\|_{L^1_t(\cB^{0,\f12})}\\
\lesssim &\bigl\|\na\cdot\bigl(({Id+\na
\bar{Y}+\na\wt{Y}})^{-1}-Id\bigr)^{t}\cB^t\na\b
Y_t\bigr\|_{L^1_t(\cB^{0,\f12})}.
\end{split}
\eeno Along the same line, one can show that similar estimate holds
with $A_{3,2}\frak{A}_1^{-1}$ in the above inequality being replaced
by the other terms in $\cB_2.$ This proves that \beq \label{S4eq18}
\left\|\cB^{t}\na\cdot\left((\mathcal{A}-Id)^t\cB^t\na\b
Y_t\right)\right\|_{L^1_t(\cB^{0,\f12})}\lesssim
\bigl\|\bigl(({Id+\na
\bar{Y}+\na\wt{Y}})^{-1}-Id\bigr)^{t}\cB^t\na\b
Y_t\bigr\|_{L^1_t(\cB^{1,\f12})}.\eeq Using the fact that
$(Id+A)^{-1}-Id=\sum_{n=1}^\infty A^n$ and  \eqref{Lip2},
\eqref{S4eq17a}, we deduce by the law of product, Lemma \ref{L3},
that for $t\leq T^\star$ and $\e\leq \e_3,$ \beno
\begin{split} \bigl\|\bigl(({Id+\na
\bar{Y}+\na\wt{Y}})^{-1}&-Id)^{t}\cB^t\na\b
Y_t\bigr\|_{L^1_t(\cB^{1,\f12})}\\
\leq & C\bigl(\|\na
\bar{Y}\|_{L^\infty_t(\cB^{1,\f12})}+\|\na\wt{Y}\|_{\cB^{1,\f12}}\bigr)\|\cB^{t}\na\b
Y_t\bigr\|_{L^1_t(\cB^{1,\f12})}\\
\leq &
C(\d+\e)\bigl(1+\|\phi\|_{H^3}+\|\phi\|_{L^\infty_\v(H^3_\h)}\bigr)\|\cB^{t}\na\b
Y_t\bigr\|_{L^1_t(\cB^{1,\f12})}.
\end{split}
\eeno Whereas the proof of \eqref{S4eq18} ensures that \beno
\|\cB^{t}\na\b Y_t\bigr\|_{L^1_t(\cB^{1,\f12})}\lesssim \|\na \bar
Y_t\|_{L^1_t(\cB^{1,\f12})}. \eeno Hence, by virtue of
\eqref{S4eq18}, we infer for $t\leq T^\star$ and $\e\leq\e_3$ \beq
\label{S4eq19}
\left\|\cB^{t}\na\cdot\left((\mathcal{A}-Id)^{t}\cB^t\na\b
Y_t\right)\right\|_{L^1_t(\cB^{0,\f12})}\leq C
(\e+\d)\bigl(1+\|\phi\|_{H^3}+\|\phi\|_{L^\infty_\v(H^3_\h)}\bigr)\|\na\bar
Y_t\|_{L^1_t(\cB^{1,\f12})}. \eeq The same estimate holds for $f_1$
and $f_2$ given by \eqref{S4eq12}.

In order to deal with the estimate of $f_3$ given by \eqref{S4eq12},
we need the following lemma concerning the estimate of the pressure
function:

\begin{lem}\label{S4lem2}
{\sl Let $t\leq T^\star$ and $\e\leq \bar{\e}\leq \e_3, \d\leq
\bar{\d}$ for some sufficiently small constants $\bar{\e}$ and
$\bar{\d}.$ Then  there holds \beq\label{Lip3a}
\begin{split} \|\na \vv p\|_{L^1_t(\cB^{0,\f12})}\leq C\Bigl( \|\pa_3\b
Y\|_{L^2_t(\cB^{1,\f12})}^2+\|\b
Y_t\|_{L^2_t(\cB^{1,\f12})}^2\Bigr). \end{split} \eeq }
\end{lem}

Let us postpone the proof of this lemma after the proof of the
proposition.

In view of \eqref{S4eq12}, we get, by a similar proof of
\eqref{S4eq18}, that \beno
\begin{split}
\|f_3\|_{L^1_t(\cB^{0,\f12})}\lesssim &\bigl\|\bigl(Id+((Id+\na \bar
Y+\na\wt{Y})^{-1}-Id)\bigr)\na \vv p\bigr\|_{L^1_t(\cB^{0,\f12})}\\
\lesssim
&\bigl(1+\|\na\bar{Y}\|_{L^\infty_t(\cB^{1,\f12})}+\|\na\wt{Y}\|_{\cB^{1,\f12}}\bigr)\|\na\vv
p\|_{L^1_t(\cB^{0,\f12})}.
\end{split}
\eeno Hence by virtue of \eqref{Lip2}, \eqref{S4eq17a} and Lemma
\ref{S4lem2}, we obtain for $t\leq T^\star$ and $\e\leq \bar{\e},
\d\leq \bar{\d},$ \beno \|f_3\|_{L^1_t(\cB^{0,\f12})}\leq C\Bigl(
\|\pa_3\b Y\|_{L^2_t(\cB^{1,\f12})}^2+\|\b
Y_t\|_{L^2_t(\cB^{1,\f12})}^2\Bigr), \eeno from which,
\eqref{S4eq12} and \eqref{S4eq19},  we deduce that {for} $t\leq
T^\star $ and $\e\leq \bar{\e}, \d\leq \bar{\d},$ \beq\label{Lip4}
\begin{aligned}
\|f\|_{L^1_t(\mathcal{B}^{0,\f12})} \leq &
C\Bigl((\e+\d)\bigl(1+\|\phi\|_{H^3}+\|\phi\|_{L^\infty_\v(H^3_\h)}\bigr)\|\b
Y_t\|_{L^1_t(\mathcal{B}^{2,\f12})} \\
&\qquad\qquad\qquad\qquad+\|\b Y_t\|_{L^2_t(\cB^{1,\f12})}^2
+\|\partial_3\b Y\|_{L^2_t({\cB}^{1,\f12})}^2\Bigr).
\end{aligned}
 \eeq Let us denote
\beq \label{S4eq45} \begin{split} &\e_0\eqdefa\min\bigl(\bar{\e},
\f1{4C}\bigl(1+\|\phi\|_{H^3}+\|\phi\|_{L^\infty_\v(H^3_\h)}\bigr)^{-1}\bigr)\andf\\
&\d_0\eqdefa \min\bigl(\bar{\d},
\f1{4C}\bigl(1+\|\phi\|_{H^3}+\|\phi\|_{L^\infty_\v(H^3_\h)}\bigr)^{-1}\bigr).
\end{split} \eeq Then we deduce from
 \eqref{2.1},
\eqref{Lip3a} and \eqref{Lip4} that for $t\leq T^\star$ and $\e\leq
\e_0, \d\leq \d_0,$ \beq\label{Lip5}
\begin{aligned}
\|\b Y_t&\|_{\wt L^\infty_t(\mathcal{B}^{0,\f12})} +\|\pa_3\b
Y\|_{\wt L^\infty_t(\mathcal{B}^{0,\f12})} +\|\D\b Y\|_{\wt
L^\infty_t(\mathcal{B}^{0,\f12})}+\|\b Y_t\|_{\wt L^2_t(\mathcal{B}^{1,\f12})} \\
& +\|\p_3\b Y\|_{\wt{L}^2_t(\mathcal{B}^{1,\f12})}+\|\b
Y_t\|_{L^1_t(\mathcal{B}^{2,\f12})}+\|\na \vv p\|_{L^1_t(\cB^{0,\f12})} \\
\leq & C\Bigl(\| Y_1\|_{\mathcal{B}^{0,\f12}} +\|\pa_3\b
Y_0\|_{\mathcal{B}^{0,\f12}}+\|\D\b Y_0\|_{\mathcal{B}^{0,\f12}}
+\|\b Y_t\|_{L^2_t(\cB^{1,\f12})}^2 +\|\partial_3\b
Y\|_{L^2_t(\cB^{1,\f12})}^2\Bigr).
\end{aligned}
\eeq Let us denote \beno T^\ast\eqdefa \sup\bigl\{\ t\leq T^\star,\
\ \|\b Y_t\|_{L^2_t(\cB^{1,\f12})} +\|\partial_3\b
Y\|_{L^2_t(\cB^{1,\f12})}\leq 2Cc_0\ \bigr\}. \eeno Then we deduce
from \eqref{Lip0} and \eqref{Lip5} that for $t\leq T^\ast,$ \beno
\|\b Y_t\|_{L^2_t(\cB^{1,\f12})} +\|\partial_3\b
Y\|_{L^2_t(\cB^{1,\f12})}\leq \f{Cc_0}{1-2C^2c_0}\leq \f32Cc_0,
\eeno provided that $c_0\leq \f1{6C^2}.$ This proves that
$T^\ast=T^\star,$ and there holds \beno
\begin{aligned}
\|\b Y_t&\|_{\wt L^\infty_t(\mathcal{B}^{0,\f12})} +\|\pa_3\b
Y\|_{\wt L^\infty_t(\mathcal{B}^{0,\f12})} +\|\D\b Y\|_{\wt
L^\infty_t(\mathcal{B}^{0,\f12})}+\|\b Y_t\|_{\wt L^2_t(\mathcal{B}^{1,\f12})} \\
& +\|\p_3\b Y\|_{\wt{L}^2_t(\mathcal{B}^{1,\f12})}+\|\b
Y_t\|_{L^1_t(\mathcal{B}^{2,\f12})}+\|\na \vv p\|_{L^1_t(\cB^{0,\f12})} \\
\leq & \f{C}{1-2Cc_0}\Bigl(\| Y_1\|_{\mathcal{B}^{0,\f12}}+\|\pa_3\b
Y_0\|_{\mathcal{B}^{0,\f12}}+\|\D\b Y_0\|_{\mathcal{B}^{0,\f12}}
\Bigr)
\end{aligned} \eeno
  for any $t\leq T^\star.$
Then for $c_0\leq \min\bigl(\f1{6C^2},\f1{8C}\bigr),$ by taking
$\d=\min\bigl(\d_0,2Cc_0\bigr),$ for $\d_0$ given by \eqref{S4eq45},
in \eqref{Lip2} shows that $T^\star=T$ and \eqref{Lip1} holds for
any $t\leq T.$ This completes the proof of Proposition \ref{Lip}.
 \end{proof}

Let us now present the proof of Lemma \ref{S4lem2}.

\begin{proof}[Proof of Lemma \ref{S4lem2}]
We first deduce from \eqref{PR} that
\begin{equation}\label{S4eq20}
\begin{aligned}
\|\na \vv p\|_{L^1_t(\cB^{0,\f12})}\lesssim
&\bigl\|\det(\cB^{-1})(\cB\mathcal{A}\mathcal{A}^t\cB^{t}-Id)\na
\vv p\bigr\|_{L^1_t(\cB^{0,\f12})}\\
&+\bigl\|(\det(\cB^{-1})Id-Id)\na \vv p\bigr\|_{L^1_t(\cB^{0,\f12})}\\
& +\left\|\cB\mathcal{A}\dv
\bigl(\det(\cB^{-1})\cB\mathcal{A}\bigl(\pa_{3}\b Y\otimes\pa_{3}\b
Y -\b Y_t\otimes\b Y_t\bigr)\bigr)\right\|_{L^1_t(\cB^{0,\f12})}.
\end{aligned}
\end{equation}
Note that \beq \label{S4eq31}
\begin{split}
\cB\cA\cA^{t}\cB^t-Id=&(\cA\cA^t-Id)+(\cB-Id)+(\cB-Id)^t\\
&+(\cB-Id)(\cA\cA^t-Id)+(\cB-Id)(\cB-Id)^t\\
&+(\cA\cA^t-Id)(\cB-Id)^t+(\cB-Id)(\cA\cA^t-Id)(\cB-Id)^t.
\end{split}
\eeq Let us deal with the typical term above. Indeed it follows
Lemma \ref{L3} that \beq\label{S4eq40}
\begin{split} \|(\cA\cA^T-Id)(\cB-Id)^t\na\vv p\|_{L^1_t(\cB^{0,\f12})}\lesssim &
\|(\cA\cA^t-Id)\|_{L^\infty_t(\cB^{1,\f12})}\|(\cB-Id)^{t}\na\vv
p\|_{L^1_t(\cB^{0,\f12})}\\
 \lesssim & \bigl(
\|\na\bY\|_{L^\infty_t(\cB^{1,\f12})}+\|\na\wt{Y}\|_{\cB^{1,\f12}}\bigr)\|(\cB-Id)^{t}\na\vv
p\|_{L^1_t(\cB^{0,\f12})}.
\end{split}
\eeq And for $\cB_1$ given by \eqref{S4eq16}, one has \beno
\|\cB_1^{t}\na\vv p\|_{L^1_t(\cB^{0,\f12})}\lesssim
\|\cB_1\|_{\cB^{1,\f12}}\|\na\vv p\|_{L^1_t(\cB^{0,\f12})},\eeno and
Lemmas \ref{L3} and \ref{S4lem3} ensure that for $\e\leq\e_2$ \beno
\begin{split}
\|(A_{3,2}\frak{A}_1^{-1})^{t}\na\vv
p\|_{L^1_t(\cB^{0,\f12})}\lesssim
&\bigl(1+\|\frak{A}_1-Id\|_{\cB^{1,\f12}}\bigr)\|A_3^\h\|_{\dB^1_\h}\|(1-\eta(z_3))\na\vv
p\|_{L^1_t(\cB^{0,\f12})}\\
\lesssim &\|\na\vv p\|_{L^1_t(\cB^{0,\f12})}. \end{split} \eeno The
same estimate holds with $A_{3,2}\frak{A}_1^{-1}$ in the above
inequality being replaced by  the other terms in $\cB_2$ given by
\eqref{S4eq16}. This leads to \beno \|\cB_2^{t}\na\vv
p\|_{L^1_t(\cB^{0,\f12})}\lesssim \|\na\vv
p\|_{L^1_t(\cB^{0,\f12})}, \eeno which together with \eqref{Lip2},
\eqref{S4eq17a} and \eqref{S4eq40} ensures that for $\e\leq \e_3$
\beno \|(\cB-Id)^{t}\na\vv p\|_{L^1_t(\cB^{0,\f12})} \lesssim
\|\na\vv p\|_{L^1_t(\cB^{0,\f12})}, \eeno and \beq\label{S4eq41}
\begin{split} \|(\cA\cA^{t}-Id)(\cB-Id)^{t}\na\vv
p\|_{L^1_t(\cB^{0,\f12})}\leq C(\e+\d)
\bigl(1+\|\phi\|_{H^3}+\|\phi\|_{L^\infty_\v(H^3_\h)}\bigr) \|\na\vv
p\|_{L^1_t(\cB^{0,\f12})}.
\end{split}
\eeq Similar estimate holds for the other terms in \eqref{S4eq31}.
 Furthermore, due to the special structure of the matrix $\cB$ given
 by \eqref{S4eq16}, we get, by a similar derivation of
 \eqref{S4eq41}, that
 \beq \label{S4eq42}
\bigl\|\det(\cB^{-1})(\cB\mathcal{A}\mathcal{A}^t\cB^{t}-Id)\na \vv
p\bigr\|_{L^1_t(\cB^{0,\f12})}\leq C
(\e+\d)\bigl(1+\|\phi\|_{H^3}+\|\phi\|_{L^\infty_\v(H^3_\h)}\bigr)\|\na\vv
p\|_{L^1_t(\cB^{0,\f12})}. \eeq The same estimate holds for
$\bigl\|(\det(\cB^{-1})Id-Id)\na \vv
p\bigr\|_{L^1_t(\cB^{0,\f12})}.$

Along the same line, we can show that for $\e\leq \e_3,$ \beno
\begin{split}
\left\|\cB\mathcal{A}\dv
\bigl(\det(\cB^{-1})\cB\mathcal{A}\bigl(\pa_{3}\b Y\otimes\pa_{3}\b
Y\right. -&\left.\b Y_t\otimes\b
Y_t\bigr)\bigr)\right\|_{L^1_t(\cB^{0,\f12})} \leq  C
\Bigl(\|\pa_3\b Y\|_{L^2_t(\cB^{1,\f12})}^2+\|\b
Y_t\|_{L^2_t(\cB^{1,\f12})}^2\Bigr), \end{split} \eeno from which
and \eqref{S4eq42}, we infer
\begin{equation*}
\begin{aligned}
\|\na \vv p\|_{L^1_t(\cB^{0,\f12})} \leq&
C\Bigl((\e+\d)\bigl(1+\|\phi\|_{H^3}+\|\phi\|_{L^\infty_\v(H^3_\h)}\bigr)\|\na
\vv p\|_{L^1_t(\cB^{0,\f12})}+\|\pa_3\b
Y\|_{L^2_t(\cB^{1,\f12})}^2+\|\b
Y_t\|_{L^2_t(\cB^{1,\f12})}^2\Bigr),
\end{aligned}
\end{equation*} which leads to \eqref{Lip3a} by taking $\e\leq \bar{\e}$ and $\d\leq \bar{\d}$ with $\bar{\e}$ and $\bar{\d}$ being given by
\beno \bar{\e}\eqdefa\min(\e_3, \bar{\d})\andf \bar{\d}\eqdefa
\f1{4C}\bigl(1+\|\phi\|_{H^3}+\|\phi\|_{L^\infty_\v(H^3_\h)}\bigr)^{-1}.
\eeno
 This completes the proof of Lemma \ref{S4lem2}.
\end{proof}

\setcounter{equation}{0}
\section{The decay of the solutions to \eqref{1.2c}}

In this section, let us fix $b_0=e_3,$ then the matrix $\cB$ given
by \eqref{1.2fk} equals to $Id.$   Then the System \eqref{1.2} then
becomes \eqref{1.2c}.
 For simplicity, we shall
denote $\na_y$ by $\na$ in this section.

\begin{prop}\label{prop5.1}
{\sl Let $Y$ be a smooth global  solution of \eqref{1.2c}. Let
$$
E_0\eqdefa\|Y_1\|_{H^1}^2 +\|\p_3 Y_0\|_{H^1}^2  +\|\Delta
Y_0\|_{L^2}^2.
$$
 If we
assume that \beq\label{eq1.3}
\begin{aligned}
E_0+\sup_{t\in\R^+}\Bigl(\|\na Y(t)\|_{\dot
B^{\f32}}&+\|Y(t)\|_{\cB^{2,-\f12}_{2,\infty}}
+\|Y(t)\|_{\cB^{4,-\f12}_{2,\infty}}\\
&+\|\pa_{3} Y(t)\|_{\dot H^{\f52}} +\| Y_t(t)\|_{\dot H^{\f52}}
+\|\D Y(t)\|_{\dot H^{\f34}}\Bigr) \le\eta_0,
\end{aligned}
\eeq and  \beq\label{S5eq5}
\lambda(t)\eqdefa\|Y_t(t)\|_{\cB^{0,-\f12}_{2,\infty}}+
\|Y(t)\|_{\cB^{1,-\f12}_{\infty,\infty}}+\|Y(t)\|_{\cB^{\f52,-\f12}_{2,\infty}}\leq
\la_0 \eeq for some $\la_0>0$ and some sufficiently small $\eta_0.$
Then one has \beq\label{eq1.3a} \|Y_t(t)\|_{H^1}^2 +\|\p_3
Y(t)\|_{H^1}^2
 +\|Y(t)\|_{\dot H^2}^2 \lesssim
\f{(\lambda_0+E_0)^{2}E_0}{(\lambda_0+E_0)^{2}+E_0\sqrt{t}}. \eeq
 }
\end{prop}

Let us remark that the proof of this proposition is motivated by
similar ideas in \cite{GT13,RWXZ}, which are formulated in the
Eulerian coordinates. Moreover, compared with the result in
\cite{RWXZ}, here we work out the limiting decay rate, namely, here
the solution decays like $\w{t}^{-\f14},$ while the solution in
\cite{RWXZ} decays like $\w{t}^{-s}$ for any $s\in ]0,1/4[.$

\begin{proof}
We first get by a similar derivation of \eqref{d5} that \beno
\begin{aligned}
\f{d}{dt}\Bigl(\f{1}{2}\bigl(\|Y_t\|_{L^2}^2 +\|\p_{3} Y\|_{L^2}^2 &
+\f{1}{4}\|\Delta Y\|_{L^2}^2\bigr) -\f{1}{4}(Y_t |\Delta Y)\Bigr)
\\&
\qquad+\f{3}{4}\|\na Y_t\|_{L^2}^2 +\f{1}{4}\|\na\p_{3}
Y\|_{L^2}^2=\bigl(f | (Y_t-\f{1}{4}\Delta Y)\bigr).
\end{aligned}
\eeno While by performing $L^2$ inner product of \eqref{1.2c} with
$-\D Y_t,$ we obtain \beno
\begin{aligned}
\f{1}{2}\f{d}{dt}\bigl(\|\na Y_t(t)\|_{L^2}^2 +\|\na\p_{3}
Y(t)\|_{L^2}^2\bigr) +\|\D Y_t\|_{L^2}^2 =-\bigl(f\ |\Delta
Y_t\bigr).
\end{aligned}
\eeno By summing up  the above two inequalities, we obtain
 \beq\label{eq1.24}
\begin{aligned}
\f{d}{dt}\Bigl(&\f{1}{2}\bigl(\|Y_t\|_{H^1}^2 +\|\p_{3} Y\|_{H^1}^2
 +\f{1}{4}\|\Delta Y\|_{L^2}^2\bigr)-\f{1}{4}(Y_t| \Delta Y)\Bigr)
\\&
\qquad +\f{3}{4}\|\na Y_t\|_{L^2}^2+\|\D Y_t\|_{L^2}^2
+\f{1}{4}\|\na\p_{3} Y\|_{L^2}^2\\
&=\bigl(f |( Y_t-\f{1}{4}\Delta Y-\Delta Y_t)\bigr),
\end{aligned}
\eeq for $f$ given by \eqref{1.2d}. Let us now deal with the last
line of \eqref{eq1.24} term by term.

\no$\bullet$\underline{The estimate of
$\bigl(\na\cdot\bigl((\mathcal{A}\mathcal{A}^{t}-Id)\na Y_t\bigr) |
(Y_t-\f{1}{4}\Delta Y-\Delta Y_t)\bigr)$}

Due to \beq\label{eq1.24a}
\mathcal{A}\mathcal{A}^{t}-Id=(\mathcal{A}-Id)(\mathcal{A}-Id)^t
+\mathcal{A}-Id+(\mathcal{A}-Id)^t \quad\qquad\and\qquad
\mathcal{A}-Id=\sum_{n=1}^\infty(\na Y)^n, \eeq we get, by using the
classical product law: \beq \label{prop5.1eq3} \|ab\|_{\dot
H^s}\lesssim \|a\|_{{\dB}^{\f32}}\|b\|_{\dot
H^s}\quad\mbox{for}\quad |s|<\f32,\eeq and \eqref{eq1.3} that \beno
\begin{split}
\|\mathcal{A}\mathcal{A}^t-Id\|_{\dot B^{\f32}} \leq& C\bigl(1+\|\na
Y\|_{L^\infty_t(\dB^{\f32})}\bigr)\|\na Y\|_{\dB^{\f32}}\\
\leq &C\|\na Y\|_{\dB^{\f32}},  \end{split} \eeno and
\beq\label{prop5.1eq1}
\begin{aligned}
\bigl|&\left(\na\cdot\bigl((\mathcal{A}\mathcal{A}^t-Id)\na
Y_t\bigr)| (Y_t-\Delta Y_t)\right)\bigr| \\
&\lesssim \|(\mathcal{A}\mathcal{A}^t-Id)\na Y_t\|_{L^2}\|\|\na
Y_t\|_{L^2} +\|(\mathcal{A}\mathcal{A}^t-Id)\na Y_t\|_{\dot H^1}
\|\D Y_t\|_{L^2}
\\&
\lesssim \|\mathcal{A}\mathcal{A}^t-Id\|_{\dot B^{\f32}} \bigl(\|\na
Y_t\|_{L^2}^2+\|\D Y_t\|_{L^2}^2\bigr)
\\&
\lesssim \|\na Y\|_{\dot B^{\f32}}\bigl(\|\na Y_t\|_{L^2}^2+\|\D
Y_t\|_{L^2}^2\bigr).
\end{aligned}
\eeq To deal with the term
$\left(\na\cdot\bigl((\mathcal{A}\mathcal{A}^t-Id)\na Y_t\bigr)|
\Delta Y\right),$ we write \beq\label{S5eq9}
\begin{aligned}
-\left(\na\right.&\left.\cdot\bigl((\mathcal{A}\mathcal{A}^{t}-Id)\na
Y_t\bigr) | \Delta Y\right)\\
=& -\f{d}{dt}\left(\na\cdot\bigl((\mathcal{A}\mathcal{A}^t-Id)\na
Y\bigr) | \D
Y\right)+\left(\na\cdot\bigl((\mathcal{A}\mathcal{A}^t-Id)\na
Y\bigr) | \D Y_t\right)
\\&
+\left(\na\cdot\bigl(\bigl(\pa_t((\mathcal{A}-Id)
(\mathcal{A}-Id)^t)+\p_t\mathcal{A}+\p_t\mathcal{A}^t\bigr)\na
Y\bigr) | \D Y\right).
\end{aligned}
\eeq By virtue of \eqref{eq1.3} and  Lemma \ref{S3lem1}, we deduce
$$
\begin{aligned}
\bigl|\left(\na\cdot\bigl((\mathcal{A}\mathcal{A}^t-Id)\na Y\bigr) |
\D Y_t\right)\bigr| &\lesssim \bigl(1+\|\na Y\|_{L^\infty}\bigr)
\|\na Y\|_{L^4}\|\D Y\|_{L^4}\|\D Y_t\|_{L^2}
\\&
\lesssim \|Y\|_{\cB^{2,-\f12}_{2,\infty}}^{\f12}
\|Y\|_{\cB^{4,-\f12}_{2,\infty}}^{\f12}\|\na\pa_{3}Y\|_{L^2}\|\D
Y_t\|_{L^2},
\end{aligned}
$$ and
$$
\begin{aligned}
\bigl|&\left(\na\cdot\bigl(\bigl(\pa_t((\mathcal{A}-Id)
(\mathcal{A}-Id)^t)+\p_t\mathcal{A}+\p_t\mathcal{A}^t\big)\na
Y\bigr)| \D Y\right)\bigr|
\\&
\lesssim \bigl(1+ \|\na Y\|_{L^\infty}\bigr)\bigl(\|\na
Y\|_{L^4}\|\D Y\|_{L^4}\|\D Y_t\|_{L^2} +\|\D Y\|_{L^4}^2\|\na
Y_t\|_{L^2}\bigr)
\\&
\lesssim \Bigl(\|Y\|_{\cB^{2,-\f12}_{2,\infty}}^{\f12}
\|Y\|_{\cB^{4,-\f12}_{2,\infty}}^{\f12}\|\D Y_t\|_{L^2}
+\|Y\|_{\cB^{4,-\f12}_{2,\infty}}\|\na
Y_t\|_{L^2}\Bigr)\|\na\pa_{3}Y\|_{L^2}.
\end{aligned}
$$
Hence, we obtain
 \beq\label{prop5.1eq2}
\begin{aligned}
&-\left(\na\cdot\bigl((\mathcal{A}\mathcal{A}^t-Id)\na Y_t\bigr) |
\Delta Y\right) \lesssim
-\f{d}{dt}\left(\na\cdot\bigl((\mathcal{A}\mathcal{A}^t-Id)\na
Y\bigr) | \D Y\right)
\\&\qquad\qquad\qquad\qquad
+\Bigl(\|Y\|_{\cB^{2,-\f12}_{2,\infty}}
+\|Y\|_{\cB^{4,-\f12}_{2,\infty}}\Bigr)\|\na\pa_{3}Y\|_{L^2}\bigl(\|\na
Y_t\|_{L^2}+\|\D Y_t\|_{L^2}\bigr).
\end{aligned}
\eeq

\no$\bullet$\underline{The estimate of $\bigl(\mathcal{A}^t\na \vv p
|(-Y_t+\f14\D Y+\D Y_t)\bigr)$}

It is easy to observe that
$$
\begin{aligned}
\bigl(\mathcal{A}^t&\na \vv p | (-Y_t+\f14\D Y+\D Y_t)\bigr)\\
=&\bigl(\mathcal{A}^t\na \vv p |(\f14\D Y+\D Y_t)\bigr)
-\bigl(\na\cdot(\mathcal{A}^t\vv p)|Y_t\bigr)
+\bigl(\na\cdot\mathcal{A}^t\vv p| Y_t\bigr)
\\
\lesssim &\bigl(1+\|\na Y\|_{L^\infty}\bigr) \bigl(\|\D
Y\|_{L^2}+\|\D Y_t\|_{L^2}\bigr)\|\na \vv p\|_{L^2}
\\&
+\Bigl(\bigl(1+\|\na Y\|_{L^\infty}\bigr)\|\na Y_t\|_{L^2} +\|\D
Y\|_{L^4}\|Y_t\|_{L^4}\Bigr)\|\vv p\|_{L^2}.
\end{aligned}
$$
On the other hand,  it follows from \eqref{1.2d} that \beno \|\na
\vv p\|_{L^2} \leq C\Bigl(\|\na Y\|_{L^\infty}\|\na \vv
p\|_{L^2}+\bigl\|\cA\bigl(\p_3Y\otimes\p_3Y-Y_t\otimes
Y_t\bigr)\bigr\|_{\dot H^1}\Bigr),\eeno so that as long as $\eta_0$
in \eqref{eq1.3} is sufficiently small, we deduce from the product
law \eqref{prop5.1eq3} that \beno
\begin{split}
\|\na \vv p\|_{L^2} \leq & C\bigl(1+\|\na
Y\|_{L^\infty_t({B}^{\f32})}\bigr)\bigl\|\bigl(\p_3Y\otimes\p_3Y-Y_t\otimes
Y_t\bigr)\bigr\|_{\dot H^1}\\
\leq& C\bigl( \|\pa_{3}Y\|_{\dot H^{\f54}}^2+\| Y_t\|_{\dot
H^{\f54}}^2\bigr)\\
 \leq& C\bigl( \|\na\pa_{3}
Y\|_{L^2}^{\f53}\|\pa_{3} Y\|_{\dot H^{\f52}}^{\f13} +\|\na
Y_t\|_{L^2}^{\f53}\| Y_t\|_{\dot H^{\f52}}^{\f13}\bigr).
\end{split}
\eeno Along the same line, we deduce from \eqref{1.2d} and the law
of product \eqref{prop5.1eq3} that \beno
\begin{split}
\|\vv p\|_{L^2}\leq& C\Bigl(\bigl\|(\cA\cA^t-Id)\na \vv p\|_{\dot
H^{-1}}+\bigl\|\cA\bigl(\p_3Y\otimes\p_3Y-Y_t\otimes
Y_t\bigr)\bigr\|_{L^2}\Bigr)\\
\leq& C\Bigl(\|\na Y\|_{{B}^{\f32}}\|\na \vv p\|_{\dot
H^{-1}}+\bigl(1+\|\na
Y\|_{{B}^{\f32}}\bigr)\bigl\|\bigl(\p_3Y\otimes\p_3Y-Y_t\otimes
Y_t\bigr)\bigr\|_{L^2}\Bigr),
\end{split}
\eeno then under the assumption of \eqref{eq1.3}, we have \beno
\begin{split}
 \|\vv p\|_{L^2} \leq& C\bigl(
\|\pa_{3}Y\|_{L^4}^2+\| Y_t\|_{L^4}^2\bigr)\\
 \leq& C\bigl(
\|\pa_{3}Y\|_{L^2}^{\f12}\|\na\pa_{3}Y\|_{L^2}^{\f32}
+\|Y_t\|_{L^2}^{\f12}\|\na Y_t\|_{L^2}^{\f32}\bigr).
\end{split}
\eeno Therefore,  by applying Lemmas \ref{lem1.1} and \ref{S3lem1},
we arrive at \beq\label{prop5.1eq5}
\begin{aligned}
\bigl|\bigl(\mathcal{A}^{t}&\na \vv p |(-Y_t+\f14\D Y+\D
Y_t)\bigr)\bigr|
\\
\lesssim &
\Bigl(\|Y\|_{\cB^{\f52,-\f12}_{2,\infty}}^{\f23}\|\na\pa_{3}Y\|_{L^2}^{\f13}
+\|\D Y_t\|_{L^2}\Bigr) \Bigl(\|\na\pa_{3} Y\|_{L^2}^{\f53}\|\pa_{3}
Y\|_{\dot H^{\f52}}^{\f13} +\|\na Y_t\|_{L^2}^{\f53}\| Y_t\|_{\dot
H^{\f52}}^{\f13}\Bigr)
\\&
+\Bigl(\|\na Y_t\|_{L^2} +\|\D Y\|_{\dot
H^{\f34}}\|Y_t\|_{L^2}^{\f14}\|\na Y_t\|_{L^2}^{\f34}\Bigr)
\Bigl(\|\pa_{3}Y\|_{L^2}^{\f12}\|\na\pa_{3}Y\|_{L^2}^{\f32}
+\|Y_t\|_{L^2}^{\f12}\|\na Y_t\|_{L^2}^{\f32}\Bigr).
\end{aligned}
\eeq

\no$\bullet$\underline{The closure of the energy estimate}

Let us denote \beq\label{prop5.1eq6}
\begin{aligned}
E_0(t)\eqdefa &\f{1}{2}\Bigl(\|Y_t(t)\|_{H^1}^2 +\|\p_{3}
Y(t)\|_{H^1}^2 +\f{1}{4}\|\Delta Y(t)\|_{L^2}^2\Bigr)
\\&\quad \   -\f{1}{4}(Y_t | \Delta Y)
+\left(\na\cdot\bigl((\mathcal{A}\mathcal{A}^t-Id)\na Y\bigr)
| \D Y\right)\andf\\
D_0(t)\eqdefa& \|\na Y_t(t)\|_{H^1}^2 +\|\p_3\na Y(t)\|_{L^2}^2.
\end{aligned}
\eeq Then by resuming the Estimates \eqref{prop5.1eq1},
\eqref{prop5.1eq2} and \eqref{prop5.1eq5} into \eqref{eq1.24}, we
obtain $$\longformule{\f{d}{dt}E_0(t)+\f14D_0(t) \leq C\Bigl(\|\na
Y\|_{\dot B^{\f32}}+\|Y\|_{\cB^{2,-\f12}_{2,\infty}}
+\|Y\|_{\cB^{4,-\f12}_{2,\infty}}+\|\pa_{3} Y\|_{\dot H^{\f52}}+\|
Y_t\|_{\dot H^{\f52}}  }{{}+\|\pa_{3}Y\|_{L^2}+\|Y_t\|_{L^2}+\|\D
Y\|_{\dot H^{\f34}}\|Y_t\|_{L^2}^{\f14}\|\na
Y_t\|_{L^2}^{\f14}\bigl(\|\pa_{3}Y\|_{L^2}^{\f12}+\|Y_t\|_{L^2}^{\f12}\bigr)\Bigr)D_0(t).
} $$ Thus under the assumption of \eqref{eq1.3} and \beq
\label{S5eq11}
\begin{aligned}
C\sup_{t\in\R^+}\bigl(\|\pa_{3}Y(t)\|_{L^2} +\|Y_t(t)\|_{L^2}\bigr)
\le \f1{16},
\end{aligned}
\eeq  we infer
\begin{equation}\label{MA}
\f{d}{dt}E_0(t)+\f18 D_0(t)\le 0,
\end{equation}
which in particular implies \beq \label{S5eq13} E_0(t)\leq
E_0\quad\mbox{for}\quad t\geq 0. \eeq Note that  $\|\nabla
Y\|_{L^\infty_t(L^\infty)}\leq \eta_0$, we have \beq \label{S5eq13a}
\begin{aligned}
E_0(t)\sim\|Y_t(t)\|_{H^1}^2 +\|\p_3 Y(t)\|_{H^1}^2  +\|\Delta
Y(t)\|_{L^2}^2
\end{aligned}
\eeq Thus if $\eta_0$ in \eqref{eq1.3} is sufficiently small, there
holds \eqref{S5eq11} and \eqref{MA} for all $t\in\R^+.$

On the other hand, it follows from  Lemma \ref{lem1.1} and
\eqref{S5eq5} that
$$
E_0(t) \leq C \bigl(\lambda(t)+\|\na
Y_t(t)\|_{L^2}+\|\na\p_3Y(t)\|_{L^2}\bigr)^{\f43} D_0(t)^{\f13}\leq
C (\lambda_0+E_0)^{\f43} D_0(t)^{\f13}.
$$
Then we deduce from \eqref{MA} that
$$
\f{d}{dt}E_0(t)+c(\lambda_0+E_0)^{-4}E^3_0(t)\le 0,
$$
which together with \eqref{S5eq13a} leads to \eqref{eq1.3a}. This
completes the proof of Proposition \ref{prop5.1}.
\end{proof}

\begin{prop}\label{prop5.2}
{\sl Under the assumptions of Proposition \ref{prop5.1}, if we
assume moreover that \beq\label{prop5.2eq1}
\begin{aligned}
&\sup_{t\in\R^+}\Bigl(\|\na Y(t)\|_{\dot B^{\f52}}+\|\na
Y_t(t)\|_{H^{2}} +\|Y(t)\|_{\cB^{5,-\f12}_{2,\infty}}\Bigr)
\le\eta_0,\andf \|\na
Y\|_{L^\infty_t(\cB^{\f52,-\f12}_{2,\infty})}\leq C,
\end{aligned}
\eeq for some sufficiently small $\eta_0.$ Then one has
\beq\label{prop5.2eq2} \frak{E}_1(t)\eqdefa \|Y_t(t)\|_{H^2}^2
+\|\p_3Y(t)\|_{H^2}^2
 +\|\D Y(t)\|_{H^1}^2 \lesssim
\f{\la_1^{2}E_1}{\la_1^{2}+E_1\sqrt{t}} \eeq with $E_1\eqdefa
\frak{E}_1(0)$ and $\la_1$ being given by \beq\label{prop5.2eq3}
\lambda_1(t)\eqdefa \la_0+E_1+\|\na
Y(t)\|_{\cB^{\f52,-\f12}_{2,\infty}}^2\leq \la_1. \eeq }
\end{prop}

\begin{proof} We first get, by taking $\p_k$ to the System
\eqref{1.2c} and then taking the $L^2$ inner product of the
resulting equation with $\p_{k}Y_t-\f{1}{4}\Delta \p_{k}Y-\Delta
\p_{k}Y_t,$ that \beq\label{S5eq18}
\begin{aligned}
\f{d}{dt}\Bigl(\f{1}{2}&\bigl(\|\p_{k}Y_t\|_{H^1}^2 +\|\p_{k}\p_3
Y\|_{H^1}^2  +\f{1}{4}\|\p_{k} Y\|_{\dot H^2}^2\bigr)
-\f{1}{4}(\p_{k}Y_t | \Delta\p_{k} Y)\Bigr)\\
&\quad+\f{3}{4}\|\na\p_{k} Y_t\|_{L^2}^2+\|\D\p_{k} Y_t\|_{L^2}^2
+\f{1}{4}\|\na \p_k\p_{3} Y\|_{L^2}^2\\
&=\bigl(\p_{k}f |  (\p_{k}Y_t-\f{1}{4}\Delta \p_{k}Y-\Delta\p_{k}
Y_t)\bigr)\quad\mbox{for}\ \ k=1,2,3.
\end{aligned}
\eeq We now deal with the last line of \eqref{S5eq18} term by term.
It follows from the classical product law, \eqref{prop5.1eq3}, that
\beq \label{S5eq19}
\begin{aligned}
&\left|\left(\na\cdot\p_{k}\bigl((\mathcal{A}\mathcal{A}^t-Id)\na
Y_t\bigr)  | (\p_{k}Y_t-\Delta\p_{k} Y_t)\right)\right|
\\
\quad &\lesssim \|\p_k((\mathcal{A}\mathcal{A}^t-Id)\na
Y_t)\|_{\dot{H}^1}\bigl(\|\p_{k} Y_t\|_{L^2} +\|\D\p_k
Y_t\|_{L^2}\bigr)\\
\quad &\lesssim \bigl(\|\na Y\|_{\dB^{\f32}}\|\D\p_k
Y_t\|_{L^2}+\|\na\p_k Y\|_{\dB^{\f32}}\|\D
Y_t\|_{L^2}\bigr)\bigl(\|\p_{k} Y_t\|_{L^2} +\|\D\p_k
Y_t\|_{L^2}\bigr)\\
\quad &\lesssim \bigl(\|\na Y\|_{\dB^{\f32}}+\|\na\p_k
Y\|_{\dB^{\f32}}\bigr)\bigl(\|\p_{k} Y_t\|_{L^2}^2+\|\D
Y_t\|_{L^2}^2 +\|\D\p_k Y_t\|_{L^2}^2\bigr).
\end{aligned}
\eeq Similar to \eqref{S5eq9}, one has
$$
\begin{aligned}
-&\left(\na\cdot\pa_{k}\bigl((\mathcal{A}\mathcal{A}^{t}-Id)\na
Y_t\bigr)  | \Delta\pa_{k} Y\right)\\
=&
-\f{d}{dt}\left(\na\cdot\pa_{k}\bigl((\mathcal{A}\mathcal{A}^{t}-Id)\na
Y\bigr) | \D \pa_{k}
Y\right)+\left(\na\cdot\pa_{k}\bigl((\mathcal{A}\mathcal{A}^t-Id)\na
Y\bigr) | \D\pa_{k} Y_t\right)
\\&
+\left(\na\cdot\pa_{k}\bigl(\bigl(\pa_t((\mathcal{A}-Id)
(\mathcal{A}-Id)^{t})+\p_t\mathcal{A}+\p_t\mathcal{A}^t\bigr)\na
Y\bigr) | \D\pa_{k} Y\right).
\end{aligned}
$$
It follows from Lemma \ref{S3lem1} that \beno
\begin{split}
&\left|\left(\na\cdot\pa_{k}\bigl((\mathcal{A}\mathcal{A}^{t}-Id)\na
Y\bigr) | \D\pa_{k} Y_t\right)\right| \\
&\qquad\lesssim \|\p_k\bigl((\mathcal{A}\mathcal{A}^{t}-Id)\na
Y\bigr)\|_{\dot H^1}\|\D\p_kY_t\|_{L^2}\\
&\qquad\lesssim (1+\|\na Y\|_{L^\infty})\bigl(\|\na
Y\|_{L^4}\|\na^2\p_kY\|_{L^4}+\|\na^2Y\|_{L^4}^2\bigr)\|\D\p_kY_t\|_{L^2}\\
&\qquad\lesssim \bigl(\| Y\|_{\cB^{2,-\f12}_{2,\infty}}+\|
Y\|_{\cB^{5,-\f12}_{2,\infty}}\bigr)\bigl(\|\na\p_3Y\|_{H^1}^2+\|\D\p_kY_t\|_{L^2}^2\bigr),
\end{split}
\eeno and  \beno
\begin{split}
&\left|\left(\na\cdot\pa_{k}\bigl(\bigl(\pa_t((\mathcal{A}-Id)
(\mathcal{A}-Id)^{t})+\p_t\mathcal{A}+\p_t\mathcal{A}^{t}\bigr)\na
Y\bigr) | \D\pa_{k} Y\right)\right|\\
&\ \lesssim(1+\|\na Y\|_{L^\infty})\bigl(\|\na
Y\|_{L^4}\|\D\p_kY_t\|_{L^2}+\|\na^2Y\|_{L^4}\|\D Y_t\|_{L^2}+\|\na
Y_t\|_{L^2}\|\D\p_kY\|_{L^4}\bigr)\|\D\p_kY\|_{L^4}\\
&\ \lesssim \bigl(\| Y\|_{\cB^{2,-\f12}_{2,\infty}}+\|
Y\|_{\cB^{5,-\f12}_{2,\infty}}\bigr)\bigl(\|\na\p_3Y\|_{H^1}^2+\|\na
Y_t\|_{H^2}^2\bigr).
\end{split}
\eeno This gives \beq \label{S5eq21}
\begin{split}
-&\left(\na\cdot\pa_{k}\bigl((\mathcal{A}\mathcal{A}^{t}-Id)\na
Y_t\bigr)  | \Delta\pa_{k} Y\right)\\ &\qquad=
-\f{d}{dt}\left(\na\cdot\pa_{k}\bigl((\mathcal{A}\mathcal{A}^t-Id)\na
Y\bigr) | \D \pa_{k} Y\right)\\
&\quad\qquad+\bigl(\| Y\|_{\cB^{2,-\f12}_{2,\infty}}+\|
Y\|_{\cB^{5,-\f12}_{2,\infty}}\bigr)\bigl(\|\na\p_3Y\|_{H^1}^2+\|\na
Y_t\|_{H^2}^2\bigr).
\end{split}
\eeq While we deduce from the classical product law,
\eqref{prop5.1eq3}, that \beno
\begin{split}
\bigl|\bigl(\pa_k(&\mathcal{A}^t\na\vv p) |
(-\pa_kY_t+\f14\D\pa_kY+\D\pa_k Y_t)\bigr)\bigr|\\
\lesssim &\|\cA^t\na\vv p\|_{\dot H^1}\bigl(\|\na
Y_t\|_{H^2}+\|\D\p_kY\|_{L^2}\bigr)\\
\lesssim& (1+\|\na Y\|_{\dB^{\f32}})\|\na\vv p\|_{\dot
H^1}\bigl(\|\na Y_t\|_{H^2}+\|\D\p_kY\|_{L^2}\bigr).
\end{split}
\eeno On the other hand, we infer from \eqref{1.2d} and the law of
product, \eqref{prop5.1eq3}, that \beno
\begin{split}
\|\na\vv p\|_{\dot H^1}\lesssim &\|(\cA\cA^t-Id)\na\vv p\|_{\dot
H^1}+\|\cA\dive\bigl(\cA(\p_3Y\otimes\p_3Y-Y_t\otimes
Y_t)\bigr)\|_{\dot H^1}\\
\lesssim &\|\na Y\|_{\dB^{\f32}}\|\na\vv p\|_{\dot H^1}+(1+\|\na
Y\|_{\dB^{\f32}})\|\cA(\p_3Y\otimes\p_3Y-Y_t\otimes Y_t)\|_{\dot
H^2},
\end{split}
\eeno which together with, \eqref{eq1.3}, \eqref{prop5.2eq1} and the
interpolation inequality: $\|a\|_{\dB^{\f32}}\lesssim \|a\|_{\dot
H^1}^{\f12} \|a\|_{\dot H^2}^{\f12},$ ensures that\beq\label{S5eq22}
\begin{split}
\|\na\vv p\|_{\dot H^1}\lesssim &\bigl(1+\|\na
Y\|_{\dB^{\f32}}+\|\na
Y\|_{\dB^{\f52}}\bigr)\bigl(\|\p_3Y\|_{\dB^{\f32}}\|\na\p_3Y\|_{H^1}+\|Y_t\|_{\dB^{\f32}}\|\na
Y_t\|_{H^1}\bigr)\\
\lesssim &\|\na \p_3Y\|_{H^1}^2+\|\na Y_t\|_{H^1}^2.
\end{split}
\eeq Hence under the assumptions of \eqref{eq1.3} and
\eqref{prop5.2eq1}, we obtain \beq\label{S5eq23}
\bigl|\bigl(\pa_k(\mathcal{A}^t\na\vv p)|
(-\pa_kY_t+\f14\D\pa_kY+\D\pa_k Y_t)\bigr)\bigr|\lesssim \bigl(\|\na
Y\|_{\dot H^2}+\|\na Y_t\|_{H^2}\bigr)\bigl(\|\na
\p_3Y\|_{H^1}^2+\|\na Y_t\|_{H^2}^2\bigr). \eeq

Let us now denote \beq \label{S5eq24} \begin{split} \dot
{{E}}_1(t)\eqdefa &\f12\bigl(\|\na
Y_t(t)\|_{H^1}^2+\|\na\p_3Y(t)\|_{H^1}^2+\f14\|\na
Y(t)\|_{\dot H^2}^2\bigr)\\
&-\f14(\na Y_t\ |\ \na\D
Y)+\sum_{k=1}^3\left(\na\cdot\pa_{k}\bigl((\mathcal{A}\mathcal{A}^T-Id)\na
Y\bigr) | \D \pa_{k} Y\right),\\
\dot D_1(t)\eqdefa &\|\D Y_t(t)\|_{H^1}^2+\|\p_3 Y(t)\|_{\dot
H^2}^2,\andf\\
 E_1(t)\eqdefa
& E_0(t)+\dot{{E}}_1(t),\quad  D_1(t)\eqdefa D_0(t)+  \dot D_1(t).
\end{split}
\eeq
 Note that \beno \sum_{k=1}^3\bigl|\left(\na\cdot\pa_{k}\bigl((\mathcal{A}\mathcal{A}^T-Id)\na
Y\bigr) | \D \pa_{k} Y\right)\bigr| \lesssim \|\na
Y\|_{\dB^{\f32}}\|\na Y\|_{\dot H^2}^2,\eeno so that there holds
\beq \label{S6eq26b} \begin{split}
 &\dot{{E}}_1(t)\sim \|\na
Y_t(t)\|_{H^1}^2+\|\na\p_3Y(t)\|_{H^1}^2+\|\na Y(t)\|_{\dot
H^2}^2,\\
&{E}_1(t)\sim  \|Y_t(t)\|_{H^2}^2+\|\p_3Y(t)\|_{H^2}^2+\|\D Y(t)\|_{
H^1}^2.  \end{split} \eeq

Then by resuming the inequalities \eqref{S5eq19}, \eqref{S5eq21} and
\eqref{S5eq23} into \eqref{S5eq18}, we obtain \beq \label{S5eq25}
\begin{split} \f{d}{dt}\dot{{E}}_1(t)+\f14\dot D_1(t) \leq C\Bigl(&\|\na
Y\|_{\dB^{\f32}}+\|\na Y\|_{\dB^{\f52}}+\|\na Y_t\|_{H^2}+\|
Y\|_{\cB^{2,-\f12}_{2,\infty}}+\|
Y\|_{\cB^{5,-\f12}_{2,\infty}}\Bigr)D_1(t).
\end{split}
\eeq By summing up \eqref{MA} with \eqref{S5eq25} and using the
smallness assumptions \eqref{eq1.3} and \eqref{prop5.2eq1} leads to
\beq \label{S5eq26} \f{d}{dt}{E}_1(t)+ \f1{16}D_1(t) \leq 0, \eeq
which in particular implies \beq E_1(t)+\f1{16}\int_s^tD_1(t')dt'
\leq E_1(s)\quad \forall s\in [0,t[. \label{S6eq26a} \eeq In
particular \eqref{S6eq26b} and \eqref{S6eq26a} ensures that \beq
\label{S6eq26er} \begin{split}
\|Y_t(t)\|_{H^2}^2+\|\p_3Y(t)\|_{H^2}^2+\|\D Y(t)\|_{
H^1}^2+\int_0^t\bigl(\|\na
Y_t(t')\|_{H^2}^2+\|\na\p_3Y(t)\|_{H^1}^2\bigr)dt' \leq CE_1.
\end{split} \eeq
Then we deduce from Lemma \ref{lem1.1} and \eqref{S6eq26b} that
\beno
\begin{split} E_1(t)\leq & C\bigl(\la_0+E_0+\|\na
Y_t(t)\|_{H^1}^2+\|\na\p_3Y(t)\|_{H^1}^2+\|\na
Y(t)\|_{\cB^{\f52,-\f12}_{2,\infty}}^2\bigr)^{\f23}D_1^{\f13}(t)\\
\leq & C(\la_0+E_1+\|\na
Y(t)\|_{\cB^{\f52,-\f12}_{2,\infty}}^2\bigr)^{\f23}D_1^{\f13}(t)\\
\leq &C\la_1^{\f43}D_1^{\f13}(t), \end{split} \eeno which together
with \eqref{S5eq26} ensures that
$$
\f{d}{dt}{E}_1(t)+c\lambda_1^{-4}{E}^3_1(t)\le 0,
$$
which leads to \eqref{prop5.2eq2}. This completes the proof of
Proposition \ref{prop5.2}.
\end{proof}

\begin{prop}\label{prop5.3}
{\sl Under the assumptions of Proposition \ref{prop5.2}, one has
\beq\label{S5eq26c} \|\p_3Y_t(t)\|_{H^1}^2 +\|\p_3^2Y(t)\|_{H^1}^2
 +\|\p_3Y(t)\|_{\dot H^2}^2 \lesssim
\w{t}^{-\f34}. \eeq  }
\end{prop}

\begin{proof} It follows from the classical product law, \eqref{prop5.1eq3}, that
\beq \label{S5eq28}
\begin{split}
&\left|\left(\na\cdot\pa_{3}\bigl((\mathcal{A}\mathcal{A}^T-Id)\na
Y_t\bigr) \ |\ \Delta\pa_{3} Y\right)\right|\\
&\qquad\qquad\leq \|\pa_{3}\bigl((\mathcal{A}\mathcal{A}^T-Id)\na
Y_t\bigr)\|_{\dot
H^1} \|\pa_{3} Y\|_{\dot H^2}\\
 &\qquad\qquad\lesssim \bigl(\|\na Y\|_{\dB^{\f32}}\|\p_3 Y_t\|_{\dot H^2}+\|\na
Y_t\|_{\dB^{\f32}}\|\p_3 Y\|_{\dot H^2}\bigr)\|\pa_{3} Y\|_{\dot
H^2}.
\end{split}
\eeq We denote
 \beno \begin{split} \dot E_1^3(t)\eqdefa
&\f12\bigl(\|\p_3 Y_t(t)\|_{H^1}^2+\|\p_3^2Y(t)\|_{H^1}^2+\f14\|\p_3
Y(t)\|_{\dot H^2}^2\bigr)-\f14(\p_3 Y_t\ |\ \p_3\D
Y),\andf\\
\dot D_1^3(t)\eqdefa &\|\na\p_3 Y_t(t)\|_{H^1}^2+\|\p_3^2
Y(t)\|_{\dot H^1}^2.
\end{split}
\eeno Then resuming the Inequalities \eqref{S5eq19}, \eqref{S5eq23}
and \eqref{S5eq28} into \eqref{S5eq18} for $k=3$ gives rise to \beno
\begin{split}
\f{d}{dt}\dot E_1^3(t)+\f14\dot D_1^3(t)\leq C\bigl(\|\na
Y\|_{\dB^{\f32}}+&\|\na Y\|_{\dot H^2}+\|\na
Y_t\|_{H^2}\bigr)\\
&\qquad\times \bigl(\|\p_3Y_t\|_{H^2}^2+\|\na
Y_t\|_{H^2}^2+\|\na\p_3Y\|_{H^1}^2\bigr),
\end{split}
\eeno from which and the smallness condition and \eqref{eq1.3},
\eqref{prop5.2eq1}, we infer \beno
\begin{split}
\f{d}{dt}&\dot E_1^3(t)+\f18\dot D_1^3(t)\\
\leq &C\bigl(\|\na
Y\|_{\dB^{\f32}}+\|\na Y\|_{\dot H^2}+\|\na
Y_t\|_{H^2}\bigr)\bigl(\|\p_3Y_t\|_{L^2}^2+\|\na
Y_t\|_{H^2}^2+\|\na\p_3Y\|_{H^1}^2\bigr)\\
\leq &C\bigl(\|\na Y\|_{\dot B^{\f32}}+\|\na Y\|_{\dot
B^{\f52}}+\|\na Y_t\|_{H^2}\bigr) D_1(t)\leq C\eta_0  D_1(t).
\end{split}
\eeno which together with \eqref{S6eq26a} yields
\beq\label{S5eq29}
\dot E_1^3(t)+\int_0^t\dot D_1^3(t')\leq \dot
E_1^3(0)+C\eta_0\int_0^tD_1(t')\,dt'\leq C\bigl(\dot
E_1^3(0)+E_1\bigr). \eeq Moreover, note that $\dot E_1^3(t)\lesssim
D_1(t),$  for any $0<s\leq t,$ we have \beno \begin{split}
\f{d}{dt}\bigl((t-s)\dot E_1^3(t)\bigr)+(t-s)\dot D_1^3(t)\leq &\dot
E_1^3(t)+ C (t-s)\bigl(\|\na Y\|_{\dot H^1}+\|\na Y\|_{\dot
H^2}+\|\na Y_t\|_{H^1}\bigr) D_1(t)\\
\leq &C\Bigl(1+(t-s)\bigl(\|\na Y\|_{\dot H^1}+\|\na Y\|_{\dot
H^2}+\|\na Y_t\|_{H^1}\bigr)\Bigr) D_1(t). \end{split} \eeno Then in
view of \eqref{eq1.3a}, \eqref{prop5.2eq2} and \eqref{S6eq26a}, we
get, by integrating the above inequality over $[s,t]$ and then
taking $s=\f{t}2,$ that \beno
\begin{split}
t\dot E_1^3(t)\leq
&C\left(\int_{\f{t}2}^tD_1(t')\,dt'+t\int_{\f{t}2}^t\bigl(\|\na
Y(t')\|_{\dot H^1}+\|\na Y(t')\|_{\dot H^2}+\|\na
Y_t(t')\|_{H^1}\bigr)D_1(t')\,dt'\right)\\
\leq
&C\left(\dot{{E}}_1(t/2)+\w{t}^{\f34}\int_{\f{t}2}^tD_1(t')\,dt'\right)\\
\leq &C\w{t}^{\f34}\dot{{E}}_1(t/2)\leq C\w{t}^{\f14},
\end{split}
\eeno which together with \eqref{S5eq29} leads to \eqref{S5eq26c}.
This completes the proof of Proposition \ref{prop5.3}.
\end{proof}

\setcounter{equation}{0}
\section{Propagation of regularities in the Lagrangian coordinate}

In this section, we prove the regularity estimates, which are
required by the last section, namely, \eqref{eq1.3}, \eqref{S5eq5}
and  \eqref{prop5.2eq1}.

\begin{prop}\label{S6prop1} {\sl Let $s>-1$ and  $Y$ be a smooth enough solution of \eqref{1.2c} on $[0,T],
$ which satisfies the Inequality \eqref{Lip1}. We denote $$
\begin{aligned}
E_s(t) \eqdefa & \|Y_t\|_{\wt L^\infty_t(\mathcal{B}^{s,0})}
+\|\pa_3Y\|_{\wt L^\infty_t(\mathcal{B}^{s,0})}+\|Y\|_{\wt
L^\infty_t(\mathcal{B}^{s+2,0})}\\
&+ \|\pa_3 Y\|_{\wt L^2_t(\mathcal{B}^{s+1,0})}
+\|Y_t\|_{\wt{L}^2_t(\mathcal{B}^{s+1,0})} +\|
Y_t\|_{L^1_t(\mathcal{B}^{s+2,0})}+\|\na\vv
p\|_{{L}^1_t(\cB^{s,0})}.
\end{aligned}
$$
 Then
under the assumption of \eqref{Lip0qd},   we have
\beq\label{S6eq14a}
\begin{aligned}
E_s(t) \lesssim & c_0+\|\pa_3Y_0\|_{\mathcal{B}^{s,0}} + \|
Y_0\|_{\mathcal{B}^{s+2,0}}+\|
Y_1\|_{\mathcal{B}^{s,0}}\quad\mbox{for}\ \ t\in ]0, T].
\end{aligned}
\eeq  }
\end{prop}

\begin{proof} In view of \eqref{1.2c},
we  get, by a similar derivation  \eqref{2.1}, that
\begin{equation}\label{GH}
\begin{aligned}
\|Y_t\|_{\wt L^\infty_t(\mathcal{B}^{s,0})} &+\|\pa_3Y\|_{\wt
L^\infty_t(\mathcal{B}^{s,0})} +\|Y\|_{\wt
L^\infty_t(\mathcal{B}^{2+s,0})}\\
&+ \|\pa_3 Y\|_{\wt L^2_t(\mathcal{B}^{s+1,0})} + \|Y_t\|_{\wt
L^2_t(\mathcal{B}^{s+1,0})}
+\|Y_t\|_{L^1_t(\mathcal{B}^{s+2,0})}\\
\lesssim & \|\pa_3 Y_0\|_{\mathcal{B}^{s,0}} +\|
Y_0\|_{\mathcal{B}^{s+2,0}}+\| Y_1\|_{\mathcal{B}^{s,0}} +\|
f\|_{L^1_t(\mathcal{B}^{s,0})}.
\end{aligned}
\end{equation}
Let us now handle term by term of $\|
f\|_{L^1_t(\mathcal{B}^{s,0})}$ for $f$ given by \eqref{1.2d}. It
follows  from the law of product, \eqref{GH1}, that  \beno
\begin{split}
\bigl\|\na\cdot\bigl((\mathcal{A} \mathcal{A}^t-Id)\na
Y_t\bigr)\bigr\|_{L^1_t(\cB^{s,0})}
 \lesssim &
\|\mathcal{A}\mathcal{A}^t-Id\|_{{L}^\infty_t(
\cB^{s+1,0})} \|\na Y_t\|_{L^1_t(\cB^{1,\f12})}\\
&+ \|\mathcal{A}\mathcal{A}^t-Id\|_{{L}^\infty_t(\cB^{1,\f12})}\|
Y_t\|_{L^1_t(\cB^{s+2,0})}.
\end{split}
\eeno Note from \eqref{Lip1} that \beno  \|\na
Y\|_{L^\infty_t(\cB^{1,\f12})} \leq C\|
Y\|_{\wt{L}^\infty_t(\cB^{2,\f12})}\leq Cc_0,\eeno and thus by the
law of product, Lemma \ref{L3}, \eqref{eq1.24a} and \eqref{Lip1}, we
have \beq\label{S6eq3}
\begin{split}
&\|\mathcal{A}\mathcal{A}^t-Id\|_{{L}^\infty_t(\cB^{1,\f12})}\leq
C\|\na Y\|_{L^\infty_t(\cB^{1,\f12})}\leq Cc_0,\andf\\
&\|\mathcal{A}\mathcal{A}^t-Id\|_{{L}^\infty_t(\cB^{s+1,0})}\leq
C\bigl(1+\|\na
Y\|_{L^\infty_t(\cB^{1,\f12})}\bigr)\|Y\|_{L^\infty_t(\cB^{s+2,0})}\leq
C\|Y\|_{L^\infty_t(\cB^{s+2,0})}.
\end{split} \eeq  So that by virtue of \eqref{Lip1}, we infer  \beq \label{S6eq13}
\begin{split}
 \bigl\|\na\cdot\bigl((\mathcal{A} \mathcal{A}^t-Id)\na
Y_t\bigr)\bigr\|_{L^1_t(\cB^{s,0})}\leq & C\bigl(\|
Y\|_{L^\infty_t(\cB^{s+2,0})} \|\na Y_t\|_{L^1_t(\cB^{1,\f12})}+c_0\|Y_t\|_{{L}^1_t(\cB^{s+2,0})}\bigr)\\
\leq & Cc_0\bigl(\| Y\|_{L^\infty_t(\cB^{s+2,0})}
+\|Y_t\|_{{L}^1_t(\cB^{s+2,0})}\bigr).
\end{split}
\eeq Similarly we deduce from \eqref{1.2d}, \eqref{GH1} and
\eqref{Lip1} that
 \beno
\begin{split}
\|\na\vv p\|_{{L}^1_t(\cB^{s,0})}\lesssim &\|(\cA\cA^t-Id)\na\vv p
\|_{{L}^1_t(
\cB^{s,0})}+\bigl\|\cA\dive\bigl(\cA\bigl(\p_3Y\otimes\p_3Y-Y_t\otimes
Y_t\bigr)\bigr)\bigr\|_{{L}^1_t(\cB^{s,0})}\\
\lesssim &\|(\cA\cA^t-Id)\|_{L^\infty_t(\cB^{1,\f12})}\|\na\vv p
\|_{{L}^1_t(\cB^{s,0})}+\|(\cA\cA^T-Id)\|_{\wt{L}^\infty_t(\cB^{1+s,0})}\|\na\vv p \|_{{L}^1_t(\cB^{0,\f12})}\\
&+\bigl(1+\|\cA-Id\|_{L^\infty_t(\cB^{1,\f12})}\bigr)\bigl\|\cA\bigl(\p_3Y\otimes\p_3Y-Y_t\otimes
Y_t\bigr)\bigr\|_{{L}^1_t(\cB^{1+s,0})}\\
&+\|\na\cA\|_{{L}^\infty_t(\cB^{1+s,0})}\bigl\|\cA\bigl(\p_3Y\otimes\p_3Y-Y_t\otimes
Y_t\bigr)\bigr\|_{{L}^1_t(\cB^{1,\f12})}.
\end{split}
\eeno Yet it follows from \eqref{Lip1} and the law of product, Lemma
\ref{L3},  that \beno
\begin{split}
\bigl\|\cA\bigl(\p_3Y&\otimes\p_3Y-Y_t\otimes
Y_t\bigr)\bigr\|_{{L}^1_t(\cB^{1,\f12})}\\
\lesssim
&\bigl(1+\|\cA-Id\|_{L^\infty_t(\cB^{1,\f12})}\bigr)\bigl(\|\p_3Y\|_{L^2_t(\cB^{1,\f12})}^2+\|Y_t\|_{L^2_t(\cB^{1,\f12})}^2\bigr)
\leq Cc_0^2, \end{split} \eeno and
 \beno
\begin{split}
\bigl\|\cA\bigl(&\p_3Y\otimes\p_3Y-Y_t\otimes
Y_t\bigr)\bigr\|_{{L}^1_t(\cB^{1+s,0})}\\
\lesssim &\bigl(1+\|\cA-Id\|_{{L}^\infty_t(\cB^{1+s,0})}\bigr)
\bigl(\|\p_3Y\|_{{L}^2_t(\cB^{1,\f12})}^2+\|Y_t\|_{{L}^2_t(\cB^{1,\f12})}^2\bigr)\\
&+\bigl(1+\|\cA-Id\|_{L^\infty_t(\cB^{1,\f12})}\bigr)\bigl(\|\p_3Y\|_{{L}^2_t(\cB^{1,\f12})}\|\p_3Y\|_{{L}^2_t(\cB^{1+s,0})}
+\|Y_t\|_{\wt{L}^2_t(\cB^{1,\f12})}\|Y_t\|_{{L}^2_t(\cB^{1+s,0})}\bigr)\\
\leq
&Cc_0\bigl(c_0+\|Y\|_{L^\infty_t(\cB^{2+s,0})}+\|\p_3Y\|_{{L}^2_t(\cB^{1+s,0})}+\|Y_t\|_{{L}^2_t(\cB^{1+s,0})}\bigr).
\end{split}
\eeno Hence in view of \eqref{Lip1} and  \eqref{S6eq3}, we get
\beq\label{S6eq14} \|\na\vv p\|_{{L}^1_t(\cB^{s,0})}\leq
Cc_0\bigl(c_0+\|Y\|_{{L}^\infty_t(\cB^{2+s,0})}+\|\p_3Y\|_{{L}^2_t(\cB^{1+s,0})}+\|Y_t\|_{{L}^2_t(\cB^{1+s,0})}\bigr).
\eeq Therefore thanks to \eqref{GH1} and  \eqref{Lip1}, we have
\beno
\begin{split}
\|\cA^t\na\vv p\|_{\wt{L}^1_t(\cB^{s,0})}\lesssim
&\bigl(1+\|\cA-Id\|_{L^\infty_t(\cB^{1,\f12})}\bigr)\|\na\vv
p\|_{{L}^1_t(\cB^{s,0})}+\bigl(1+\|\na\cA\|_{{L}^\infty_t(\cB^{s,0})}\bigr)\|\na\vv p\|_{L^1_t(\cB^{0,\f12})}\\
\leq&
Cc_0\bigl(c_0+\|Y\|_{{L}^\infty_t(\cB^{2+s,0})}+\|\p_3Y\|_{{L}^2_t(\cB^{1+s,0})}+\|Y_t\|_{{L}^2_t(\cB^{1+s,0})}\bigr),
\end{split} \eeno which together with \eqref{S6eq13} ensures that
\beno
\begin{split}
\|f\|_{ L^1_t(\cB^{s,0})} \leq Cc_0\Bigl(&c_0+\|
Y\|_{{L}^\infty_t(\cB^{s+2,0})}+\|\pa_3
Y\|_{L^2_t(\cB^{s+1,0})}+\|Y_t\|_{L^2_t(\cB^{s+1,0})}+\|Y_t\|_{{L}^1_t(\cB^{s+2,0})}\Bigr).
\end{split}
\eeno Then by resuming the above estimate  into \eqref{GH} and
taking $c_0$ to be sufficiently small gives rise to
\begin{equation*}
\begin{aligned}
\|Y_t\|_{\wt L^\infty_t(\mathcal{B}^{s,0})} &+\|\pa_3Y\|_{\wt
L^\infty_t(\mathcal{B}^{s,0})} +\|Y\|_{\wt
L^\infty_t(\mathcal{B}^{2+s,0})}+ \|\pa_3 Y\|_{\wt
L^2_t(\mathcal{B}^{s+1,0})}\\
&+\|Y_t\|_{\wt L^2_t(\mathcal{B}^{s+1,0})}^2
+\|Y_t\|_{L^1_t(\mathcal{B}^{s+2,0})}\\
\leq  & C\Bigl(c_0+\|\pa_3 Y_0\|_{\mathcal{B}^{s,0}} +\|
Y_0\|_{\mathcal{B}^{s+2,0}}+\| Y_1\|_{\mathcal{B}^{s,0}}\Bigr),
\end{aligned}
\end{equation*}
from which and \eqref{S6eq14}, we deduce \eqref{S6eq14a}. This
completes the proof of the Proposition.
\end{proof}

An immediate corollary of Proposition \ref{S6prop1} and Definitions
\ref{def1} and \ref{def2} gives

\begin{col}\label{LM}
{\sl Under the assumptions of Proposition \ref{S6prop1}, one has
\beq
\begin{aligned}
\|Y_t\|_{\wt L^\infty_T(\dot H^{s})}& +\|\p_3Y\|_{\wt
L^\infty_T(\dot H^s)} +\|Y\|_{\wt L^\infty_T(\dot H^{2+s})} +\|Y_t\|_{L^2_T(\dot H^{s+1})} \\
&  +\|\p_3Y\|_{L^2_T(\dot
H^{s+1})}+\|Y_t\|_{L^1_T(\dot H^{2+s})}+\|\na\vv p\|_{{L}^1_t(\dot H^{s})}\\
\leq  & C\Bigl(c_0+\|\pa_3 Y_0\|_{\mathcal{B}^{s,0}} +\|
Y_0\|_{\mathcal{B}^{s+2,0}}+\| Y_1\|_{\mathcal{B}^{s,0}}\Bigr)\quad
\mbox{ for any} \ \ s>-1.
\end{aligned}
\eeq}
\end{col}

\begin{prop}\label{S6prop2}
{\sl Let $Y$ be a smooth enough solution of \eqref{1.2c} on $[0,T],
$ which satisfies the Estimate \eqref{Lip1}. Then under the
assumptions of \eqref{Lip0qd} and \eqref{S2eq27}, One has \beq
\label{S6eq1}
\|Y_t\|_{L^\infty_t(\cB^{0,-\f12}_{2,\infty})}+\|Y_t\|_{
L^\infty_t(\cB^{3,-\f12}_{2,\infty})}+\|Y\|_{
L^\infty_t(\cB^{2,-\f12}_{2,\infty})}+\|Y\|_{L^\infty_t(\cB^{5,-\f12}_{2,\infty})}
+\|Y\|_{L^\infty_t(\cB^{1,-\f12}_{\infty,\infty})}\leq C(c_0+\d_0)
\eeq for any $t\leq T.$}
\end{prop}

\begin{proof} We first deduce from Proposition \ref{S6prop1} and
\eqref{S2eq27} that \beq\label{S6eq11} E_0(t)+E_3(t)\leq
C(c_0+\d_0)\eeq for $E_s(t)$ given by Proposition \ref{S6prop1}.
While in view of Definition \ref{def2}, we get, by a similar
derivation of \eqref{2.1}, that for all $s\in\R,$
\begin{equation}\label{DEC}
\begin{aligned}
\|Y_t\|_{\wt L^\infty_t(\cB^{s,-\f12}_{2,\infty})}+\|Y\|_{\wt
L^\infty_t(\cB^{2+s,-\f12}_{2,\infty})} \lesssim \|\pa_3
Y_0\|_{\cB^{s,-\f12}_{2,\infty}}+\|
Y_0\|_{\cB^{2+s,-\f12}_{2,\infty}} +\|
Y_1\|_{\cB^{s,-\f12}_{2,\infty}}+\| f\|_{\wt
L^1_t(\cB^{s,-\f12}_{2,\infty})}.
\end{aligned}
\end{equation}

\no$\bullet$ \underline{\bf The  estimate of $\|Y_t\|_{\wt
L^\infty_t(\cB^{0,-\f12}_{2,\infty})}$ and $\|Y\|_{\wt
L^\infty_t(\cB^{2,-\f12}_{2,\infty})} $}\;

It is easy to observe from Definition \ref{def2} that \beno
\|a\|_{\cB^{s,-\f12}_{2,\infty}}\lesssim
\|a\|_{\cB^{s_1,-\f12}_{\infty,\infty}}^\th\|a\|_{\cB^{s_2,-\f12}_{2,\infty}}^{1-\th}
\with s=\th s_1+(1-\th)s_2\andf \th\in ]0,1[,\eeno from which and
\eqref{S2eq27}, we infer \beq\label{S6eq18} \|\pa_3
Y_0\|_{\cB^{0,-\f12}_{2,\infty}}+\| Y_0\|_{\cB^{2,-\f12}_{2,\infty}}
+\| Y_1\|_{\cB^{0,-\f12}_{2,\infty}}\leq C\d_0. \eeq Then thanks to
\eqref{DEC}, we only need to deal with the estimate of $\| f\|_{\wt
L^1_t(\cB^{0,-\f12}_{2,\infty})}$. Indeed according to \eqref{1.2d},
we deduce from Lemma \ref{PRO}  that \beq\label{S6eq2}
\begin{aligned}
\| f\|_{\wt L^1_t(\cB^{0,-\f12}_{2,\infty})} &\lesssim
\|(\mathcal{A}\mathcal{A}^t-Id)\na Y_t\|_{\wt
L^1_t(\cB^{1,-\f12}_{2,\infty})} +\|\mathcal{A}^t\na \vv p\|_{\wt
L^1_t(\cB^{0,-\f12}_{2,\infty})}
\\&
\lesssim
\|\mathcal{A}\mathcal{A}^t-Id\|_{L^\infty_t(\cB^{1,0})}\|Y_t\|_{L^1_t(\cB^{2,0})}
+\bigl(1+\|\mathcal{A}-Id\|_{L^\infty(\cB^{1,0})}\bigr)\|\na \vv
p\|_{L^1_t(\cB^{0,0})}
\\&
\lesssim\bigl(1+\|\na
Y\|_{L^\infty_t(\cB^{1,\f12})}\bigr) \|Y\|_{L^\infty_t(\cB^{2,0})}\|Y_t\|_{L^1_t(\cB^{2,0})}\\
&\quad+\Bigl(1+\bigl(1+\|\na
Y\|_{L^\infty_t(\cB^{1,\f12})}\bigr)\|\na
Y\|_{L^\infty_t(\cB^{1,0})}\Bigr)\|\na \vv p\|_{L^1_t(\cB^{0,0})}.
\end{aligned}
\eeq Yet it follows from \eqref{1.2d} and the law of product, Lemma
\ref{L3}, that
$$
\begin{aligned}
\|\na \vv p\|_{L^1_t(\cB^{0,0})} &\lesssim
\|(\mathcal{A}\mathcal{A}^T-Id)\na \vv p\|_{L^1_t(\cB^{0,0})}
+\bigl\|\mathcal{A}\dv \bigl(\mathcal{A}\bigl(\pa_{3}
Y\otimes\pa_{3} Y -Y_t\otimes
Y_t\bigr)\bigr)\bigr\|_{L^1_t(\cB^{0,0})}
\\&
\lesssim \|\na Y\|_{L^\infty_t(\cB^{1,\f12})}\|\na
\vv p\|_{L^1_t(\cB^{0,0})}\\
&\quad +\bigl(1+\|\na Y\|_{L^\infty_t(\cB^{1,\f12})}\big)
\bigl\|\mathcal{A}\bigl(\pa_{3} Y\otimes\pa_{3} Y -Y_t\otimes
Y_t\bigr)\bigr\|_{L^1_t(\cB^{1,0})},
\end{aligned}
$$
from which, \eqref{Lip1}, Lemma \ref{L3} and Lemma \ref{L1},   we
infer
\begin{equation}\label{SV}
\begin{aligned}
\|\na \vv p\|_{L^1_t(\cB^{0,0})} &\lesssim
\|\mathcal{A}\bigl(\pa_{3} Y\otimes\pa_{3} Y -Y_t\otimes
Y_t\bigr)\|_{L^1_t(\cB^{1,0})}
\\&
\lesssim \bigl(1+\|\na Y\|_{L^\infty_t(\cB^{1,\f12})}\big)
\bigl(\|\pa_{3}Y\|_{L^2_t(\cB^{1,\f14})}^2+\|
Y_t\|_{L^2_t(\cB^{1,\f14})}^2\bigr)
\\&
\lesssim \|\pa_{3}Y\|_{L^2_t(\dot B^{\f54})}^2+\| Y_t\|_{L^2_t(\dot
B^{\f54})}^2\\
&\lesssim \|\pa_{3}Y\|_{L^2_t(\cB^{\f54,0})}^2+\|
Y_t\|_{L^2_t(\cB^{\f54,0})}^2.
\end{aligned}
\end{equation}
Resuming the Estimate  \eqref{SV} into \eqref{S6eq2} and using
\eqref{S6eq11}, we obtain
\begin{equation}\label{SEVR}
\begin{split}
\| f\|_{\wt L^1_t(\cB^{0,-\f12}_{2,\infty})} \lesssim
&\|Y\|_{L^\infty_t(\cB^{2,0})}\|Y_t\|_{L^1_t(\cB^{2,0})}
+(c_0+\d_0)\bigl(1+\|Y\|_{L^\infty_t(\cB^{2,0})}\bigr)\\
\leq & C(c_0+\d_0). \end{split} \eeq Thus in view of \eqref{S6eq18}
and \eqref{DEC}, we conclude \beq\label{S6eq8}
\begin{aligned}
\|Y_t\|_{\wt L^\infty_t(\cB^{0,-\f12}_{2,\infty})}+\|Y\|_{\wt
L^\infty_t(\cB^{2,-\f12}_{2,\infty})}\leq  C(c_0+\d_0).
\end{aligned}
\eeq

\no$\bullet$ \underline{\bf The  estimate of $\|Y_t\|_{\wt
L^\infty_t(\cB^{3,-\f12}_{2,\infty})}$ and $\|Y\|_{\wt
L^\infty_t(\cB^{5,-\f12}_{2,\infty})} $}\;

In view of \eqref{1.2d}, we get, by applying the law of product
Lemma \ref{PRO},  that
$$
\begin{aligned}
\| f\|_{\wt L^1_t(\cB^{3,-\f12}_{2,\infty})} \lesssim &
\|(\mathcal{A}\mathcal{A}^t-Id)\na Y_t\|_{\wt
L^1_t(\cB^{4,-\f12}_{2,\infty})} +\|\mathcal{A}^t\na \vv p\|_{\wt
L^1_t(\cB^{3,-\f12}_{2,\infty})}
\\
\lesssim&
\|\mathcal{A}\mathcal{A}^T-Id\|_{L^\infty(\cB^{4,0})}\|Y_t\|_{L^1_t(\cB^{2,0})}
+\|\mathcal{A}\mathcal{A}^T-Id\|_{L^\infty(\cB^{1,0})}\|Y_t\|_{L^1_t(\cB^{5,0})}\\
&+\bigl(1+\|\mathcal{A}-Id\|_{L^\infty(\cB^{1,0})}\bigr)\|\na \vv
p\|_{L^1_t(\cB^{3,0})}
+\bigl(1+\|\mathcal{A}-Id\|_{L^\infty(\cB^{4,0})}\bigr)\|\na \vv
p\|_{L^1_t(\cB^{0,0})}
\\
\lesssim &\|Y\|_{L^\infty(\cB^{2,0})}\|Y_t\|_{L^1_t(\cB^{5,0})}
+\bigl(1+\|Y\|_{L^\infty(\cB^{2,0})}\bigr)\|\na \vv
p\|_{L^1_t(\cB^{3,0})}
\\&
+\bigl(1+\|Y\|_{L^\infty(\cB^{5,0})}\bigr)\bigl(\|\na \vv
p\|_{L^1_t(\cB^{0,0})}+\|Y_t\|_{L^1_t(\cB^{2,0})}\bigr),
\end{aligned}
$$
from which, \eqref{S6eq11} and \eqref{SV}, we infer
$$
\| f\|_{\wt L^1_t(\cB^{3,-\f12}_{2,\infty})} \le C(c_0+\d_0).
$$
While due to \eqref{S2eq27}, one has \beno
 \|\pa_3
Y_0\|_{\cB^{3,-\f12}_{2,\infty}}+\| Y_0\|_{\cB^{5,-\f12}_{2,\infty}}
+\| Y_1\|_{\cB^{3,-\f12}_{2,\infty}}\leq C\d_0. \eeno Thus we deduce
from \eqref{DEC} that \beq\label{S6eq9}
\begin{split} \|Y_t\|_{\wt
L^\infty_t(\cB^{3,-\f12}_{2,\infty})} +\|Y\|_{\wt
L^\infty_t(\cB^{5,-\f12}_{2,\infty})}\le C(c_0+\d_0).
\end{split}
\eeq

\, \no$\bullet$ \underline{\bf The estimate of
$\|Y\|_{\wt{L}^\infty_t({\cB^{1,-\f12}_{\infty,\infty}})}$}\;

Again in view of \eqref{1.2c}, we get, by a similar derivation of
\eqref{2.1}, that \beq\label{S6eq12}\begin{split} \|Y_t\|_{\wt
L^\infty_t(\cB^{-1,-\f12}_{\infty,\infty})}+\|Y\|_{\wt
L^\infty_t(\cB^{1,-\f12}_{\infty,\infty})}+&\|\pa_3Y\|_{\wt
L^2_t(\cB^{0,-\f12}_{\infty,\infty})} + \|Y_t\|_{\wt
L^2_t(\cB^{0,-\f12}_{\infty,\infty})}+\|Y_t\|_{\wt
L^1_t(\cB^{1,-\f12}_{\infty,\infty})}\\
  \lesssim&  \|\pa_3
Y_0\|_{\cB^{-1,-\f12}_{\infty,\infty}}+\|
Y_0\|_{\cB^{1,-\f12}_{\infty,\infty}} +\|
Y_1\|_{\cB^{-1,-\f12}_{\infty,\infty}}+\| f\|_{\wt
L^1_t(\cB^{-1,-\f12}_{\infty,\infty})}.
\end{split}\eeq
To deal with the estimate $\| f\|_{\wt
L^1_t(\cB^{-1,-\f12}_{\infty,\infty})},$ we deduce from
\eqref{1.2d}, the law of product, Lemma \ref{PRO}, that
$$
\begin{aligned}
\| f\|_{\wt L^1_t(\cB^{-1,-\f12}_{\infty,\infty})} &\lesssim
\|(\mathcal{A}\mathcal{A}^t-Id)\na Y_t\|_{\wt
L^1_t(\cB^{0,-\f12}_{\infty,\infty})} +\bigl(1+\|\na Y\|_{\wt
L^\infty_t(\cB^{1,\f12})}\bigr) \|\na \vv p\|_{\wt
L^1_t(\cB^{-1,-\f12}_{\infty,\infty})}
\\&
\lesssim \|\mathcal{A}\mathcal{A}^t-Id\|_{\wt
L^\infty_t(\cB^{1,\f12})} \|\na Y_t\|_{\wt
L^1_t(\cB^{0,-\f12}_{\infty,\infty})} +\|\na \vv p\|_{\wt
L^1_t(\cB^{-1,-\f12}_{\infty,\infty})}
\\&
\lesssim \|\na Y\|_{\wt L^\infty_t(\cB^{1,\f12})}\|Y_t\|_{\wt
L^1_t(\cB^{1,-\f12}_{\infty,\infty})} +\|\na \vv p\|_{\wt
L^1_t(\cB^{-1,-\f12}_{\infty,\infty})}.
\end{aligned}
$$
While it follows from \eqref{1.2d}, Lemma \ref{L3} and \eqref{Lip1}
that
$$
\begin{aligned}
\|\na \vv p\|_{\wt L^1_t(\cB^{-1,-\f12}_{\infty,\infty})}  \lesssim
& \|\mathcal{A}\bigl(\pa_{3} Y\otimes\pa_{3} Y -Y_t\otimes
Y_t\bigr)\|_{\wt L^1_t(\cB^{0,-\f12}_{\infty,\infty})}
\\
\lesssim & \|\pa_{3} Y\|_{\wt L^2_t(\cB^{1,\f12})} \|\pa_{3}
Y\|_{\wt L^2_t(\cB^{0,-\f12}_{\infty,\infty})} +\|Y_t\|_{\wt
L^2_t(\cB^{1,\f12})} \|Y_t\|_{\wt
L^2_t(\cB^{0,-\f12}_{\infty,\infty})}\\
\leq & Cc_0\Bigl(\|\pa_{3} Y\|_{\wt
L^2_t(\cB^{0,-\f12}_{\infty,\infty})} + \|Y_t\|_{\wt
L^2_t(\cB^{0,-\f12}_{\infty,\infty})}\Bigr).
\end{aligned}
$$
Therefore, we achieve \beno \| f\|_{\wt
L^1_t(\cB^{-1,-\f12}_{\infty,\infty})}\leq Cc_0\Bigl(\|Y_t\|_{\wt
L^1_t(\cB^{1,-\f12}_{\infty,\infty})} +\|\pa_{3} Y\|_{\wt
L^2_t(\cB^{0,-\f12}_{\infty,\infty})} + \|Y_t\|_{\wt
L^2_t(\cB^{0,-\f12}_{\infty,\infty})}\Bigr). \eeno Resuming the
above estimate into \eqref{S6eq12} and using \eqref{S2eq27} gives
rise to \beq\label{S6eq10} \begin{split} \|Y_t\|_{\wt
L^\infty_t(\cB^{-1,-\f12}_{\infty,\infty})}+\|Y\|_{\wt
L^\infty_t(\cB^{1,-\f12}_{\infty,\infty})}+&\|\pa_3Y\|_{\wt
L^2_t(\cB^{0,-\f12}_{\infty,\infty})} + \|Y_t\|_{\wt
L^2_t(\cB^{0,-\f12}_{\infty,\infty})}+\|Y_t\|_{\wt
L^1_t(\cB^{1,-\f12}_{\infty,\infty})}\leq C\d_0.
\end{split}
\eeq Finally it follows from Definition \ref{def3} that \beno
\|a\|_{L^\infty_t(\cB^{s,-\f12}_{p,\infty})}\leq
\|a\|_{\wt{L}^\infty_t(\cB^{s,-\f12}_{p,\infty})} \quad\mbox{for
any}\ \ p\in [1,\infty]. \eeno Then by summing up  the Estimates
\eqref{S6eq8}, \eqref{S6eq9} and \eqref{S6eq10}, we conclude the
proof of \eqref{S6eq1}.
\end{proof}

\setcounter{equation}{0}
\section{The proof of Theorems \ref{thm1} and \ref{S2thm1} }

 Let us first present  the proof of Theorem \ref{S2thm1}.

\begin{proof}[Proof of Theorem \ref{S2thm1}] In general, the existence
of solutions to a nonlinear partial differential equation can be
obtained by performing uniform estimates to the appropriate
approximate solutions. Here for simplicity, we just present the {\it
a priori} estimate for smooth enough solution, $Y,$ of \eqref{1.2c}.
Indeed under the assumption of \eqref{Lip0qd}, we first deduce from
Proposition \ref{Lip} that $Y$ satisfies the Inequality
\eqref{Lip1}. Then it follows from Proposition \ref{S6prop1} that
there holds \eqref{S2eq29}, which ensures the global existence part
of Theorem \ref{S2thm1}. The uniqueness of such smooth solution is
standard, we omit the details here.

In order to prove the decay estimate \eqref{S2eq26}, we need to
verify the smallness conditions \eqref{eq1.3} and
\eqref{prop5.2eq1}, which are guaranteed by \eqref{S2eq29} and
Proposition \ref{S6prop2} provided that there holds \eqref{S2eq27}.
This completes the proof of Theorem \ref{S2thm1}.
\end{proof}

Now we are in a position to complete the proof of Theorem
\ref{thm1}. Let us first recall Proposition 6.1 from \cite{XZ15}:

\begin{prop}\label{p9}
{\sl Let  $b_0-\vv e_3\in H^{s}(\R^3)$ and $u_0\in H^s(\R^3)$ for
$s>\f32,$ \eqref{1.1} has a unique solution $(b,  u)$ on $[0,T]$ for
some $T>0$ so that $b-e_3\in C([0,T]; H^{s}(\R^3)),$ $  u\in
C([0,T]; H^{s}(\R^3))$ with $\na u\in L^2((0,T);H^s(\R^3))$ and $
\na p\in C([0,T]; H^{s-1}(\R^3)).$ Moreover, if $T^\ast$ is the life
span to this solution, and $T^\ast<\infty,$ one has  \beq
\int_0^{T^\ast}\bigl(\|\na u(t)\|_{L^\infty}+\|
b(t)\|_{L^\infty}^2\bigr)\,dt=\infty. \label{qs2}\eeq}
\end{prop}

\begin{proof}[Proof of  Theorem \ref{thm1}] Given initial
data $(b_0,u_0)$ which satisfies the assumptions of Theorem
\ref{thm1}, we deduce from Proposition \ref{p9} that \eqref{1.1} has
a unique solution $(b,u)$ on $[0,T^\ast[$ such that for any
$T<T^\ast,$ \beno  b-e_3\in C([0,T]; H^{s}(\R^3)),\quad  u\in
C([0,T]; H^{s}(\R^3)) \with \na u\in L^2((0,T);H^s(\R^3)). \eeno
Moreover, it follows from the transport equation of \eqref{1.1} that
\beno \|b(t)\|_{L^\infty}\leq \|b_0\|_{L^\infty}\exp\left(\|\na
u\|_{L^1_t(L^\infty)}\right). \eeno Therefore, by virtue of
Proposition \ref{p9}, in order to complete the existence part of
Theorem \ref{thm1}, it remains to prove that \beq \label{p9eq1}
\int_0^{T^\ast}\|\na u(t)\|_{L^\infty}\,dt<\infty. \eeq Toward this,
we introduce the equivalent Lagrangian formulation \eqref{1.2},
which has been presented in details in Section \ref{Sect2}. Indeed,
according to the derivation in  Section \ref{Sect2}, especially
\eqref{S2eq13} and \eqref{CH}, one has \beq\label{S4eq26}
\begin{split}
&Y_1(z)=u_0(y_\h(z_\h,w_3(z)),w_3(z))\andf
Y_t(t,y)=u(t,y+Y(t,y)) \with\\
&Y(t,(y_\h(z_\h,w_3(z)),w_3(z)))=\wt{Y}(z)+\bar{Y}(t,z),
\end{split} \eeq with $\wt{Y}(z)$ and $\bar{Y}(t,z)$ being
determined by \eqref{S4eq17} and \eqref{1.2b} respectively.

On the other hand,  let us recall (A.3) of \cite{XZ15} that
\beq\label{qqqp} \|u\circ\Phi\|_{\dot{B}^s_{p,r}}\leq
C(\|\na\Psi\|_{L^\infty})\|u\|_{\dot{B}^s_{p,r}}\quad{for}\ \ s\in
]-1,1[. \eeq While it follows from \eqref{1.2fg}, \eqref{Apeq1} and
\eqref{Apeq5} that \beq\label{S7eq1} \|\cB\|_{L^\infty}\leq
\bigl\|\f{\p z}{\p w}\bigr\|_{L^\infty}\bigl\|\f{\p y}{\p
w}\bigr\|_{L^\infty}^{-1}\leq 2\quad \mbox{for}\ \ \e\leq \e_0. \eeq
Thus by virtue of \eqref{1.2fk},  Lemma \ref{L1} and \eqref{qqqp},
we infer \beno
\begin{split}
\|Y_1\|_{\cB^{0,\f12}}\leq
C\bigl(\|\cB\|_{L^\infty}\bigr)\|u_0\|_{\dB^{\f12}}\leq C
\|u_0\|_{\dB^{\f12}}. \end{split} \eeno So that under the assumption
of \eqref{S1eq3}, we deduce from \eqref{S4eq17a} and Proposition
\ref{Lip} that for any $t<T^\ast$ and $\e\leq \e_0,$ \beq
\label{S4eq28} \|\na \bY\|_{L^\infty_t(L^\infty)}+\|\na
\bY_t\|_{L^1_t(L^\infty)}\lesssim \|\na
\bY\|_{\wt{L}^\infty_t(\cB^{1,\f12})}+\|
\bY_t\|_{L^1_t(\cB^{2,\f12})}\leq Cc_0. \eeq
 Hence it
follows from \eqref{1.1f} and \eqref{S4eq26}  that for any
$T<T^\ast$ \beno
\begin{split}
\int_0^{T}\|\na u(t)\|_{L^\infty}\,dt\leq
&\int_0^{T}\bigl\|\cA^t\na_y Y_t(t)\|_{L^\infty}\,dt \leq \int_0^T
\bigl\|\cA^t\cB^t\na_z \bar{Y}_t(t)\|_{L^\infty}\,dt\\
\leq &\bigl(1+\|\na
Y\|_{L^\infty_t(\cB^{1,\f12})}+\|\na\wt{Y}\|_{\cB^{1,\f12}}\bigr)\|\cB\|_{L^\infty_t(L^\infty)}
\int_0^T\|\na_z\bar{Y}_t(t)\|_{L^\infty}\,dt\\
\leq &C\int_0^T\|\na_z \bar{Y}_t(t)\|_{L^\infty}\,dt\leq Cc_0.
\end{split}
\eeno This proves \eqref{p9eq1} and thus the global existence part
of Theorem \ref{thm1} is proved.

In the case when $b_0=e_3,$ by virtue of \eqref{1.1b}, \eqref{1.1c}
and \eqref{1.1f}, we find that $Y(t,y)$ determined by \eqref{1.1f}
solves the System \eqref{1.2c}. Moreover, there holds \beq
\label{S5eq30} Y_0=0,\ \ b(t,X(t,y))=e_3+\p_3Y(t,y) \andf
u(t,X(t,x))=Y_t(t,y). \eeq Then under the assumptions of
\eqref{S1eq3} and \eqref{S1eq5}, there hold the Inequalities
\eqref{Lip0qd} and \eqref{S2eq27} so that by virtue of Theorem
\ref{S2thm1}, \eqref{1.2c} has a unique global solution $Y$ which
satisfies \eqref{S2eq29} and \eqref{S2eq26}.

On the other hand,  due to $\dv\, u=0,$ we have \beno
\begin{split}
&\|b(t,\cdot)-e_3\|_{L^2}=\|b(t,X(t,\cdot))-e_3\|_{L^2}=\|\p_3Y(t)\|_{L^2},\\
&\|\na b(t,\cdot)\|_{L^2}=\|\cA^t\na
b(t,X(t,\cdot))\|_{L^2}=\|\cA^t\na\p_3 Y(t)\|_{L^2}\lesssim
\|\na\p_3 Y(t)\|_{L^2},
\end{split}
\eeno and \beno
\begin{split}
\|\na^2 b(t,\cdot)\|_{L^2}=&\|(\cA^t\na)^2\p_3
Y(t)\|_{L^2}\\
\leq & C\bigl(1+\|\na Y\|_{L^\infty_t(L^\infty)}+\|\na^2
Y\|_{L^\infty_t(L^\infty)}\bigr) \|\na\p_3 Y(t)\|_{H^1}\\
\leq & C\bigl(1+\|\na
Y\|_{L^\infty_t(\cB^{1,\f12})}+\|\na^2Y\|_{L^\infty_t(\cB^{\f32,0})}\bigr)\|\na\p_3
Y(t)\|_{H^1}\leq C\|\na\p_3 Y(t)\|_{H^1}.
\end{split} \eeno
Exactly along the same line, one has \beno
\begin{split}
&\|u(t,\cdot)\|_{L^2}=\|u(t,X(t,\cdot))\|_{L^2}=\|Y_t(t)\|_{L^2},\\
&\|\na u(t,\cdot)\|_{L^2}=\|\cA^t\na
u(t,X(t,\cdot))\|_{L^2}=\|\cA^t\na Y_t(t)\|_{L^2}\lesssim
\|\na Y_t(t)\|_{L^2},\\
&\|\na^2 u(t,\cdot)\|_{L^2}=\|(\cA^t\na)^2 Y_t(t)\|_{L^2}\lesssim
\|\na Y_t(t)\|_{H^1},
\end{split}
\eeno which together with \eqref{S2eq26} ensures \eqref{S1eq4}. This
completes the proof of Theorem \ref{thm1}.
\end{proof}

\appendix

\setcounter{equation}{0}
\section{The proof of Lemma  \ref{S4lem3} }\label{appenda}

 In this
section, we always denote $\phi=(\phi_\h,\phi_3), y=(y_\h,y_3)$ and
$z=(z_\h,z_3).$ The proof of Lemma \ref{S4lem3} will be based on the
following two lemmas.

\begin{lem}\label{Saplem1}
{\sl Let $y(w)$ be determined by \eqref{S2eq13} and $g(z_\h)\eqdefa
\int_0^KG_1(y_\h(z_\h,y_3),y_3)\,dy_3.$ Then for $\e\leq \e_\al,$
which depends only on $\|\na\phi\|_{W^{2,\infty}},$ one has \beq
\label{Apeq3} \|g_1\|_{H^2_\h}\leq CK\Bigl(\sum_{|\al|\leq
2}\|\na_\h^\al G_1\|_{L^\infty_\v(L^2_\h)}\Bigr) \andf \|\na_\h
g_1\|_{H^2_\h}\leq CK\Bigl(\sum_{|\al|\leq 2}\|\na_\h^\al\na_\h
G_1\|_{L^\infty_\v(L^2_\h)}\Bigr). \eeq}
\end{lem}

\begin{proof} We first deduce from \eqref{CHab} that \beq
\label{Apeq0} \|A_2\|_{L^\infty}\leq \e
K\bigl\|\na_\h\bigl(\f{\phi_\h}{1+\e\phi_3}\bigr)\bigr\|_{L^\infty}\leq
3K\e\|\na_\h\phi\|_{L^\infty}\leq \f12, \eeq whenever $\e\leq
\e_a\eqdefa
\min\bigl(\f1{4\|\phi_3\|_{L^\infty}},\f1{6K\|\na_\h\phi\|_{L^\infty}}\bigr).$
Then by virtue of \eqref{1.2fg}, we have \beq \label{Apeq1} \f13\leq
\f{\|A_1\|_{L^\infty}}{1+\|A_2\|_{L^\infty}} \leq \bigl\|\f{\p y}{\p
w}\bigr\|_{L^\infty}\leq
\f{\|A_1\|_{L^\infty}}{1-\|A_2\|_{L^\infty}}\leq 2. \eeq
 Similarly, we deduce from
\eqref{CHab} that
 \beq \label{Apeq2}
  \begin{split} \bigl\|\f{\p^2 y}{\p w^2}\bigr\|_{L^\infty}\leq
 &
 \bigl\|(Id-A_2)^{-1}\bigr\|_{L^\infty}\left(\bigl\|\f{\p A_1}{\p
 y}\bigr\|_{L^\infty}+\bigl\|\f{\p A_2}{\p
 y}\bigr\|_{L^\infty} \bigl\|\f{\p y}{\p
w}\bigr\|_{L^\infty}\right) \bigl\|\f{\p y}{\p
w}\bigr\|_{L^\infty}\\
\leq &C\e\sum_{|\al|\leq 1}\|\na^\al\na\phi\|_{L^\infty}\leq 1,
\end{split}
\eeq and \beq \label{Apeq2a}
  \begin{split} \bigl\|\f{\p^3 y}{\p w^3}\bigr\|_{L^\infty}\leq
 &
 C\bigl\|(Id-A_2)^{-1}\bigr\|_{L^\infty}\left(\Bigl(\bigl\|\f{\p^2 A_1}{\p
 y^2}\bigr\|_{L^\infty}+\bigl\|\f{\p^2 A_2}{\p
 y^2}\bigr\|_{L^\infty} \bigl\|\f{\p y}{\p
w}\bigr\|_{L^\infty}\Bigr)\bigl\|\f{\p y}{\p
w}\bigr\|_{L^\infty}\right.\\
&\left.+ \Bigl(\bigl\|\f{\p A_1}{\p
 y}\bigr\|_{L^\infty}+\bigl\|\f{\p A_2}{\p
 y}\bigr\|_{L^\infty} \bigl\|\f{\p y}{\p
w}\bigr\|_{L^\infty}\Bigr)\bigl\|\f{\p^2 y}{\p
w^2}\bigr\|_{L^\infty}\right) \\
\leq &C\e\sum_{|\al|\leq 2}\|\na^\al\na\phi\|_{L^\infty}\leq 1,
\end{split}
\eeq provided that $\e\leq \e_\al\eqdefa
\min\bigl(\e_a,C^{-1}\bigl(\sum_{|\al|\leq
2}\|\na^\al\na\phi\|_{L^\infty}\bigr)^{-1}\bigr).$

 Then
thanks to \eqref{Apeq1}, one has \beno \|g_1\|_{L^2_\h}\leq CK
\|G_1\|_{L^\infty_\v(L^2_\h)}. \eeno While for $k,\ell,m \in
\{1,2\},$ we have \beno \f{\p G_1(y_\h(z_\h,y_3),y_3)}{\p
z_k}=\sum_{i=1}^2 \f{\p G_1}{\p y_i}(y_\h(z_\h,y_3),y_3)\f{\p
y_i(z_\h,y_3)}{\p z_k}, \eeno and  \beno \begin{split} \f{\p^2
G_1(y_\h(z_\h,y_3),y_3)}{\p z_\ell\p z_k}=&\sum_{i=1}^2 \f{\p
G_1}{\p y_i}(y_\h(z_\h,y_3),y_3)\f{\p^2 y_i}{\p z_\ell\p
z_k}+\sum_{i,j=1}^2 \f{\p^2 G_1}{\p y_j\p
y_i}(y_\h(z_\h,y_3),y_3)\f{\p y_i}{\p z_k}\f{\p y_i}{\p z_\ell},
\end{split}
\eeno and
 \beno \begin{split} &\f{\p^3
G_1(y_\h(z_\h,y_3),y_3)}{\p z_m\p z_\ell\p z_k}=\sum_{i=1}^2 \f{\p
G_1}{\p
y_i}(y_\h(z_\h,y_3),y_3)\f{\p^3 y_i}{\p z_m\p z_\ell\p z_k}\\
&\qquad\qquad+\sum_{i,j=1}^2 \f{\p^2 G_1}{\p y_j\p
y_i}(y_\h(z_\h,y_3),y_3)\Bigl(\f{\p^2 y_i}{\p z_\ell\p z_k}\f{\p
y_j}{\p z_m}+\f{\p^2 y_i}{\p z_k\p
z_m}\f{\p y_j}{\p z_\ell}\\
&\qquad\qquad+\f{\p y_i}{\p z_k}\f{\p^2 y_j}{\p z_m\p
z_\ell}\Bigr)+\sum_{i,j,l=1}^2 \f{\p^3 G_1}{\p y_l\p y_j\p
y_i}(y_\h(z_\h,y_3),y_3)\f{\p y_l}{\p z_m}\f{\p y_i}{\p z_k}\f{\p
y_j}{\p z_\ell},
\end{split}
\eeno from which, \eqref{Apeq1}, \eqref{Apeq2} and \eqref{Apeq2a},
we infer \beno \|\na_\h g_1\|_{L^2_\h}\leq CK\|\na_\h
G_1\|_{L^\infty_\v(L^2_\h)}, \eeno and \beno
\begin{split}
\|\na_\h^2g_1\|_{L^2_\h}\leq
&C\int_0^K\left(\bigl\|\na_\h^2G_1(y_\h(\cdot,y_3),y_3)\bigr\|_{L^2_\h}\bigl\|\na_\h
y_\h\|_{L^\infty}^2\right.\\
&\qquad\quad\left.+\bigl\|\na_\h
G_1(y_\h(\cdot,y_3),y_3)\bigr\|_{L^2_\h}\bigl\|\na_\h^2
y_\h\|_{L^\infty}\right)\,dy_3\\
\leq &CK\left(\|\na_\h G_1\|_{L^\infty_\v(L^2_\h)}+\|\na_\h^2
G_1\|_{L^\infty_\v(L^2_\h)}\right).
\end{split}
\eeno Similarly, one has \beno \|\na_\h^3g_1\|_{L^2_\h}\leq
CK\sum_{|\al|\leq 2} \|\na_\h^\al\na_\h G_1\|_{L^\infty_\v(L^2_\h)}.
\eeno As a consequence, we achieve \eqref{Apeq3}.
\end{proof}

\begin{lem}\label{Saplem2}
{\sl Let $y(w)$ and $w(z)$ be determined respectively by
\eqref{S2eq13} and \eqref{CH}, let $g_2(z)\eqdefa G_2(y(w(z))).$
Then for $\e\leq \e_\beta,$ which depends only on $
\|\na\phi\|_{W^{2,\infty}},$ one has \beq \label{Apeq7}
\|g_2\|_{\cB^{1,\f12}}\leq C\|g_2\|_{\dB^{\f32}}\leq C\|\na
G_2\|_{L^2}^{\f12}\|\na G_2\|_{H^1}^{\f12}. \eeq} \end{lem}

\begin{proof} It follows from \eqref{S2eq5} and  \eqref{1.2fh}
that \beq \label{Apeq5}
\begin{split}
\bigl\|\f{\p z}{\p w}-Id\bigr\|_{L^\infty}\leq
&C\int_{0}^{K}\bigl\|\na_\h\bigl(\f1{1+\e\phi_3}\bigr)(y_\h(\cdot,w_3),w_3)\bigr\|_{L^\infty}\bigl\|\f{\p
y_\h}{\p
w_\h}\bigr\|_{L^\infty}\,dw_3+\bigl\|\f{\e\phi_3}{1+\e\phi_3}\bigr\|_{L^\infty}\\
\leq &
C\e\bigl(\|\phi_3\|_{L^\infty}+\|\na_\h\phi_3\|_{L^\infty}\bigr)\leq
\f12,
\end{split}
\eeq  provided that  $\e\leq \e_b\eqdefa
\min\bigl(\e_\al,(2C)^{-1}\bigl(\|\phi_3\|_{L^\infty}+\|\na_\h\phi_3\|_{L^\infty}\bigr)^{-1}\bigr)$
for $\e_\al$ given by Lemma \ref{Saplem1}.

While by virtue of \eqref{CH} and \eqref{1.2fk}, we have \beno
z_3=w_3(z)-\e\int_0^{w_3(z)}\Phi_\e(y_\h(z_\h,y_3),y_3)\,dy_3,\eeno
where $\Phi_\e(y)\eqdefa \f{\phi_3(y)}{1+\e\phi_3(y)}.$ Then due to
\eqref{1.2fk}, for $k,\ell\in \{1,2\},$ we have \beno
\bigl(1-\e\Phi_\e(y_\h(z_\h,w_3(z)),w_3(z))\bigr)\f{\p w_3(z)}{\p
z_k}=\e\int_0^{w_3(z)}\na_\h\Phi_\e(y_\h(z_\h,y_3),y_3)\cdot\f{\p
y_\h}{\p z_k}(z_\h,y_3)dy_3, \eeno and \beno
\begin{split}
&\bigl(1-\e\Phi_\e(y(w(z)))\bigr)\f{\p^2w_3}{\p z_k\p
z_\ell}-\e\f{\p \Phi_\e}{\p y_3}(y(w(z)))\f{\p
w_3}{\p z_k}\f{\p w_3}{\p z_\ell}\\
&-\e\na_\h\Phi_\e(y(w(z)))\cdot\left(\f{\p y_\h}{\p
z_\ell}(z_\h,w_3(z))+\f{\p y_\h}{\p w_3}(z_\h,w_3(z))\f{\p
w_3}{\p z_\ell}\right)\f{\p w_3}{\p z_k}\\
&=\e\na_\h\Phi_\e(y(w(z)))\cdot\f{\p y_\h}{\p z_k}(z_\h,w_3(z))\f{\p
w_3}{\p
z_\ell}+\e\int_0^{w_3(z)}\na_\h\Phi_\e(y_\h(z_\h,y_3),y_3)\cdot\f{\p^2
y_\h}{\p z_\ell\p z_k}(z_\h,y_3)dy_3\\
&+\e\sum_{i,j=1}^2\int_0^{w_3(z)}\f{\p^2\Phi_\e}{\p y_i\p
y_j}(y_\h(z_\h,y_3),y_3)\f{\p y_j}{\p z_\ell}(z_\h,y_3)\f{\p y_i}{\p
z_k}(z_\h,y_3)dy_3.
\end{split}
\eeno Note that $\e\|\phi_3\|_{L^\infty}\leq \f14,$ for $k,\ell\in
\{1,2\},$ we have \beno
\begin{split}
\bigl\|\f{\p^2 w_3}{\p z_k\p z_\ell}\bigr\|_{L^\infty}\leq
C\e\Bigl(&\|\na\Phi_\e\|_{L^\infty}\bigl(\bigl\|\f{\p w_3}{\p
z}\bigr\|_{L^\infty}+\bigl\|\f{\p y}{\p
w}\bigr\|_{L^\infty}+\bigl\|\f{\p y}{\p
w}\bigr\|_{L^\infty}\bigl\|\f{\p w_3}{\p
z}\bigr\|_{L^\infty}\bigr)\bigl\|\f{\p w_3}{\p
z}\bigr\|_{L^\infty}\\
&+K\|\na\Phi\|_{L^\infty}\bigl\|\f{\p^2y}{\p
w^2}\bigr\|_{L^\infty}+K\|\na^2\Phi\|_{L^\infty}\bigl\|\f{\p y}{\p
w}\bigr\|_{L^\infty}^2\Bigr).
\end{split}
\eeno so that by virtue of \eqref{Apeq1} and \eqref{Apeq2}, we infer
\beq \label{Apeq6} \bigl\|\f{\p^2 w_3}{\p z_k\p
z_\ell}\bigr\|_{L^\infty}\leq   C\e
\bigl(\|\na\phi_3\|_{L^\infty}+\|\na^2\phi_3\|_{L^\infty}\bigr)\leq
1,\eeq provided that $\e\leq \e_\beta\eqdefa \min\bigl(\e_b,
C^{-1}(\|\na\phi_3\|_{L^\infty}+\|\na^2\phi_3\|_{L^\infty})^{-1}\bigr).$
 Similar calculation shows that \eqref{Apeq6} holds for all
$k,\ell\in \{1,2,3\}.$

On the other hand, for $k,\ell\in \{1,2\},$ one has \beno
\begin{split}
\f{\p g_2(z)}{\p z_k}=&\f{\p G_2}{\p
y_3}(y_\h(z_\h,w_3(z)),w_3(z))\f{\p w_3(z)}{\p z_k}\\
&+\sum_{i=1}^2\f{\p G_2}{\p
y_i}(y_\h(z_\h,w_3(z)),w_3(z))\left(\f{\p y_i}{\p z_k}+\f{\p y_i}{\p
w_3}\f{\p w_3(z)}{\p z_k}\right)(z_\h,w_3(z)),
\end{split}
\eeno and \beno
\begin{split}
\f{\p^2 g_2(z)}{\p z_k\p z_\ell}=&\f{\p^2 G_2}{\p
y_3^2}(y(w(z))\f{\p w_3}{\p z_\ell}\f{\p w_3}{\p z_k}+\f{\p G_2}{\p
y_3}(y(w(z))\f{\p^2 w_3}{\p z_\ell\p z_k}\\
&+\sum_{j=1}^2\f{\p^2 G_2}{\p y_3\p y_j}(y(w(z))\left(\f{\p y_j}{\p
z_\ell}(w(z))+\f{\p y_j}{\p w_3}(w(z))\f{\p w_3}{\p
z_\ell}\right)\\
&+\sum_{i,j=1}^2\f{\p^2 G_2}{\p y_i\p y_j}(y_(w(z))\left(\f{\p
y_j}{\p z_\ell}+\f{\p y_j}{\p w_3}\f{\p w_3}{\p
z_\ell}\right)\left(\f{\p y_i}{\p z_k}+\f{\p y_i}{\p w_3}\f{\p
w_3}{\p z_k}\right)(w(z))\\
&+\sum_{i=1}^2\f{\p G_2}{\p y_i}(y(w(z)),w_3(z))\left(\f{\p^2
y_i}{\p z_k\p z_\ell}+\f{\p^2 y_i}{\p w_3\p z_k}\f{\p
w_3}{\p z_\ell}\right.\\
&\qquad+\left.\f{\p^2 y_i}{\p w_3\p z_\ell}\f{\p w_3}{\p
z_k}+\f{\p^2 y_i}{\p w_3^2}\f{\p w_3}{\p z_k}\f{\p w_3}{\p
z_\ell}+\f{\p y_i}{\p w_3}\f{\p^2 w_3}{\p z_k\p
z_\ell}\right)(w(z)),
\end{split}
\eeno from which, \eqref{Apeq1}, \eqref{Apeq5} and \eqref{Apeq6}, we
deduce that \beno
\begin{split}
\|\na_\h g_2\|_{L^2}\leq &C\|\na G_2\|_{L^2}\bigl(1+\bigl\|\f{\p
y}{\p w}\bigr\|_{L^\infty}\bigr)\bigl(1+\bigl\|\f{\p w_3}{\p
z}\bigr\|_{L^\infty}\bigr)\bigl\|{\rm det}\bigl(\f{\p y}{\p
w}\bigr)\bigr\|_{L^\infty}\bigl\|{\rm
det}\bigl(\f{\p w}{\p z}\bigr)\bigr\|_{L^\infty}\\
\leq &C\|\na G_2\|_{L^2},
\end{split}
\eeno and \beno
\begin{split}
\|\na_\h^2g_2\|_{L^2}\leq &C \|\na^2 G_2\|_{L^2}\Bigl(\bigl\|\f{\p
y}{\p w}\bigr\|_{L^\infty}+\bigl\|\f{\p w_3}{\p
z}\bigr\|_{L^\infty}^2+\bigl\|\f{\p y}{\p
w}\bigr\|_{L^\infty}^2\bigl(1+\bigl\|\f{\p w_3}{\p
z}\bigr\|_{L^\infty}^2\bigr)\Bigr)\\
&+C\|\na G_2\|_{L^2}\Bigl(\bigl(1+\bigl\|\f{\p y}{\p
w}\bigr\|_{L^\infty}\bigr)\bigl\|\f{\p^2w_3}{\p
z^2}\bigr\|_{L^\infty}+\bigl(1+\bigl\|\f{\p w_3}{\p
z}\bigr\|_{L^\infty}^2\bigr)\bigl\|\f{\p^2y}{\p
w^2}\bigr\|_{L^\infty}\Bigr)\\
 \leq &C\bigl(\|\na G_2\|_{L^2}+\|\na^2 G_2\|_{L^2}\bigr).
\end{split}
\eeno Similar calculation shows that the above two estimates hold
for the full derivatives of $g_2.$ Hence by virtue of Lemma \ref{L1}
and the interpolation inequality in Besov space, we deduce
\eqref{Apeq7}. This finishes the proof of the lemma.
\end{proof}

Let us now turn to the proof of Lemmas  \ref{S4lem3} and
\ref{S4lem4}.

\begin{proof}[Proof of Lemma  \ref{S4lem3}]  By virtue of  \eqref{Apeq3}, for $A_2^\h$ given by  \eqref{S4eq14}, we get, by using
interpolation inequality in Besov spaces, that for $\e\leq \e_\al,$
\beq \label{Apeq4} \|A_2^\h\|_{\dB^1_\h}\leq C
\|A_2^\h\|_{H^2_\h}\leq C\e\Bigl(\sum_{|\al|\leq
2}\|\na_\h^\al\na_\h\phi \|_{L^\infty_\v(L^2_\h)}\Bigr). \eeq
Similarly by applying the second equality of \eqref{Apeq3}, we
deduce that, $\|A_3^\h\|_{\dB^1_\h},$ for $A_3^\h$ given by
\eqref{S4eq15}, satisfies the same estimate as
$\|A_2^\h\|_{\dB^1_\h}.$

 Note that for $b_0=e_3+\e\phi,$ we deduce from \eqref{CHab} and
\eqref{Apeq7} that  for $\e\leq\e_\beta$ \beq \label{Apeq10}
\|\frak{A}_1-Id\|_{\cB^{1,\f12}}\leq
C\e\|\na\phi\|_{L^2}^{\f12}\|\na\phi\|_{H^1}^{\f12}. \eeq Due to
\eqref{S4eq14}, \eqref{S4eq15} and the fact that $\eta$ is supported
on $[-2,K+2],$ we get, by a similar derivation of \eqref{Apeq7},
that \beq \label{Apeq8}
\|A_{2,1}\|_{\cB^{1,\f12}}+\|A_{3,1}\|_{\cB^{1,\f12}}\leq
C\e\|\na\phi\|_{H^2}. \eeq Let us take \beno
\e_2\eqdefa\min\Bigl(\e_\al,\e_\beta,C^{-1}\bigl(\|\na\phi\|_{H^2}+\|\na_\h\phi\|_{L^\infty_\v(H^2_\h)}\bigr)^{-1}\Bigr).
\eeno Then \eqref{Apeq8} together with \eqref{Apeq4} and
\eqref{Apeq10} leads to \eqref{S4eq30}.
\end{proof}

\begin{proof}[Proof of Lemma  \ref{S4lem4}] Let us denote
\beno \frak{Y}(z)\eqdefa \e
\eta(z_3)\int_{-1}^{z_3}\phi(y_\h(z_\h,w_3(z')),w_3(z'))dz_3' \with
z'=(z_\h,z_3'). \eeno Then  one has \beno \f{\p \frak{Y}(z)}{\p
z_3}=\e
\eta'(z_3)\int_{-1}^{z_3}\phi(y_\h(z_\h,w_3(z_\h,z_3')),w_3(z_\h,z_3'))\,dz_3'+\e\eta(z_3)\phi(y(w(z))),
\eeno and  \beno \begin{split} \f{\p^2 \frak{Y}(z)}{\p^2 z_3}=&\e
\eta''(z_3)\int_{-1}^{z_3}\phi(y(w(z')))dz_3'+2\e\eta'(z_3)\phi(y(w(z)))\\
&+\e\eta(z_3)\Bigl(\f{\p \phi}{\p y_\h}(y(w(z)))\cdot\f{\p y_\h}{\p
w_3}(w(z))\f{\p w_3}{\p z_3}(z)+\f{\p \phi}{\p y_3}(y(w(z)))\f{\p
w_3}{\p z_3}(z)\Bigr), \end{split} \eeno from which and
\eqref{Apeq1}, we infer \beq \label{Apeq12}
\|\p_3\frak{Y}\|_{L^2}\leq
C\e\bigl(\|\phi\|_{L^2}+\|\phi\|_{L^\infty_\v(L^2_\h)}\bigr). \eeq
Similarly, for $k=1,2,$ one has \beno
\begin{split}
\f{\p^2 \frak{Y}(z)}{\p z_3\p
z_k}=&\e\eta'(z_3)\int_{-1}^{z_3}\Bigl(\f{\p \phi}{\p
y_\h}(y(w(z')))\cdot\Bigl(\f{\p y_\h}{\p z_k}(z_\h,w(z'))+\f{\p
y_\h}{\p w_3}(z_\h,w(z'))\f{\p w_3}{\p z_k}(z')\Bigr)\\
&+\f{\p \phi}{\p y_3}(y(w(z')))\f{\p w_3}{\p z_k}(z')\Bigr)dz_3'+\e
\eta(z_3)\Bigl(\f{\p \phi}{\p y_3}(y(w(z)))\f{\p w_3}{\p
z_k}(z)\\
&+\f{\p \phi}{\p y_\h}(y(w(z)))\cdot\Bigl(\f{\p y_\h}{\p
z_k}(z_\h,w(z))+\f{\p y_\h}{\p w_3}(z_\h,w(z))\f{\p w_3}{\p
z_k}(z)\Bigr)\Bigr).
\end{split}
\eeno So that we obtain \beq\label{Apeq13}
\begin{split}
\|\na\p_3\frak{Y}\|_{L^2}\leq
&C\e\Bigl(\|\phi\|_{L^2}+\|\phi\|_{L^\infty_\v(L^2_\h)}+\bigl(\|\na\phi\|_{L^2}+\|\na\phi\|_{L^\infty_\v(L^2_\h)}\bigr)\\
&\qquad\qquad\qquad\qquad\qquad\qquad\times\bigl(1+ \bigl\|\f{\p
y}{\p w}\bigr\|_{L^\infty}\bigr)\bigl(1+ \bigl\|\f{\p
w}{\p z}\bigr\|_{L^\infty}\bigr)\Bigr)\\
\leq &C\e\bigl(\|\phi\|_{H^1}+\|\phi\|_{L^\infty_\v(H^1_\h)}\bigr).
\end{split}
\eeq Then by virtue of Lemma \ref{L1} and the classical
interpolation inequality in Besov spaces (see \cite{bcd}), we deduce
from \eqref{Apeq12} and \eqref{Apeq13} that \beq \label{Apeq14}
\begin{split}
\|\p_3\frak{Y}\|_{\cB^{0,\f12}}\leq
C\|\p_3\frak{Y}\|_{\dB^{\f12}}\leq &
C\|\p_3\frak{Y}\|_{L^2}^{\f12}\|\na\p_3\frak{Y}\|_{L^2}^{\f12}\\
\leq & C\e\bigl(\|\phi\|_{H^1}+\|\phi\|_{L^\infty_\v(H^1_\h)}\bigr).
\end{split} \eeq

On the other hand, it follows from a similar derivation of
\eqref{Apeq6} that \beq \label{Apeq15} \bigl\|\f{\p^3 w_3}{\p z_l\p
z_k\p z_\ell}\bigr\|_{L^\infty}\leq   C\e
\|\na\phi_3\|_{W^{2,\infty}}. \eeq Then for $\e\leq \e_3$ with
$\e_3$ being determined by \beq\label{Apeq16} \e_3\eqdefa
\min\Bigl(\e_2,
C^{-1}\bigl(1+\|\na\phi_3\|_{W^{2,\infty}}\bigr)^{-1},
C^{-1}\bigl(\|\phi\|_{H^3}+\|\phi\|_{L^\infty_\v(H^3_\h)}\bigr)^{-1}\Bigr),\eeq
we get, by a similar derivation of \eqref{Apeq14}, that \beno
\|\na\frak{Y}\|_{\cB^{1,\f12}}\leq
C\e\bigl(\|\phi\|_{H^3}+\|\phi\|_{L^\infty_\v(H^3_\h)}\bigr),\eeno
which together with \eqref{Apeq14} and \eqref{Apeq16}  ensures
\eqref{S4eq17a}.
\end{proof}

\medskip

\noindent {\bf Acknowledgments.} We would like to thank Professors
Fanghua Lin, Jiahong Wu, Zhifei Zhang, Zhen Lei and Li Xu for
profitable discussions on this topic. Part of this work was done
when H. Abidi was visiting Morningside Center of Mathematics of the
Academy of Mathematics and System Sciences, CAS  in the summer of
2015.
 He would like to thank the hospitality of the Center and AMSS. Both authors are supported by innovation grant from National
Center for Mathematics and Interdisciplinary Sciences.
 P. Zhang is partially supported by NSF  11371037.

\end{document}